\documentclass[11pt]{amsart}
\usepackage[latin1]{inputenc}
\usepackage[english]{babel}
\usepackage[T1]{fontenc}
\usepackage{mathtools}
\usepackage{amssymb,amsmath,amsthm,amsfonts,enumerate}
\usepackage[autostyle]{csquotes}
\usepackage{alphalph,cite}
\usepackage{nameref}
\usepackage{verbatim}
\usepackage{mathrsfs}
\usepackage{manfnt}
\usepackage{pifont}
\usepackage{yfonts}
\usepackage[perpage]{footmisc}
\usepackage{enumitem}
\usepackage[all]{xy} 
\usepackage{wasysym}

\usepackage{stmaryrd}

\usepackage{apptools}
\usepackage{stackengine}
\usepackage{scalerel}
\usepackage{manfnt}
\usepackage{datetime}
\usepackage{xparse} 
\usepackage{tikz-cd}
\usepackage{rotating}
\usepackage{esvect}
\usepackage[colorlinks=true, linkcolor={blue}, citecolor={magenta}, urlcolor={green}]{hyperref}
\usepackage{mathtools}
\usepackage{soul}

\makeatletter
\@namedef{subjclassname@2010}{ \textup{2010} Mathematics Subject Classification}
\makeatother
\theoremstyle{definition}
\newtheorem{theorem}{Theorem}[section]

\theoremstyle{definition}
\newtheorem{lemma}[theorem]{Lemma}
\theoremstyle{definition}
\newtheorem{corollary}[theorem]{Corollary}
\theoremstyle{definition}
\newtheorem{proposition}[theorem]{Proposition}
\theoremstyle{definition}

\theoremstyle{definition}
\newtheorem{remark}[theorem]{Remark}
\theoremstyle{definition}
\newtheorem{example}[theorem]{Example}
\theoremstyle{definition}

\numberwithin{equation}{section}

\numberwithin{equation}{section}

\numberwithin{equation}{section}
\theoremstyle{definition}

\theoremstyle{definition}
\newtheorem{notation}[theorem]{Notation}
\theoremstyle{definition}

\theoremstyle{definition}
\newtheorem{scholium}[theorem]{Digression}
\theoremstyle{definition}
\newtheorem{construction}[theorem]{Construction}
\theoremstyle{definition}
\newtheorem{problem}[theorem]{Problem}

\DeclareMathOperator{\supp}{supp}
\DeclareMathOperator{\sign}{sign}

\DeclareMathOperator*{\esssup}{ess\,sup}

\newcommand{\norm}[1]{\left\lVert#1\right\rVert}
\newcommand{\vertiii}[1]{{\left\vert\kern-0.25ex\left\vert\kern-0.25ex\left\vert #1 
		\right\vert\kern-0.25ex\right\vert\kern-0.25ex\right\vert}}
\newcommand{\abs}[1]{\left\lvert#1\right\rvert}

\newcommand{\thickbar}[1]{\mathbf{\bar{\text{$#1$}}}} 

\renewenvironment{proof}{\noindent {\bf Proof.}}{\qed}

\makeatletter
\def\author@andify{
	\nxandlist {\unskip ,\penalty-1 \space\ignorespaces}
	{\unskip {} \@@and~}
	{\unskip \penalty-2 \space \@@and~}
}
\makeatother

\newcommand\xqed[1]{%
	\leavevmode\unskip\penalty9999 \hbox{}\nobreak\hfill
	\quad\hbox{#1}}
\newcommand\demo{\xqed{{\Large $\blacktriangle$}}}

\makeatletter
\def\l@subsection{\@tocline{2}{0pt}{1pc}{5pc}{}} \def\l@subsection{\@tocline{2}{0pt}{2pc}{6pc}{}}
\makeatother

\begin{document}
\title{Essential norms of pointwise multipliers in the non-algebraic setting}

\author{Tomasz Kiwerski}
\address[Tomasz Kiwerski]{Pozna\'{n} University of Technology, Institute of Mathematics, Piotrowo 3A, 60-965 Pozna\'{n}, Poland}
\email{\href{mailto:tomasz.kiwerski@gmail.com}{\tt tomasz.kiwerski@gmail.com}}

\author{Jakub Tomaszewski}
\address[Jakub Tomaszewski]{Pozna\'{n} University of Technology, Institute of Mathematics, Piotrowo 3A, 60-965 Pozna\'{n}, Poland}
\email{\href{mailto:tomaszewskijakub@protonmail.com}{\tt tomaszewskijakub@protonmail.com}}

\maketitle

\begin{abstract}
	Motivated by some recent results, but also referring to recognized classics, we compute the essential norm and the weak essential norm of
	multiplication operators acting between two distinct K{\" o}the spaces both defined over the same $\sigma$-finite measure space.
	A by-product of the technology we have developed here are some applications to Banach sequence spaces related to decreasing functions
	and to Banach spaces of analytic functions on the unit disc, in particular, Hardy spaces.
	We will close our work with some specific examples illustrating the previously obtained results including Musielak--Orlicz sequence spaces,
	in particular, Nakano sequence spaces, as well as Lorentz and Marcinkiewicz sequence spaces.
\end{abstract}

\tableofcontents

\footnote{{\tiny}
	{\it Date:} \longdate{\today}.
		
	
	}

\section{{\bf Introduction}}

\subsection{Motivation and goals}
Problems focused around the task of precisely determining the essential norm, that is, the quotient norm modulo the compact operators,
of {\it multiplication operators} $M_{\lambda} \colon f \rightsquigarrow \left[ \omega \rightsquigarrow \lambda(\omega) f(\omega) \right]$
and {\it composition operators} $C_{\varphi} \colon f \rightsquigarrow \left[ \omega \rightsquigarrow (f \circ \varphi)(\omega) \right]$
(or, a mix of both, the so-called {\it weighted composition operators})
acting between various function spaces has a long history and a well-established reputation (see, for example,
$\bullet$ \cite{Axl82}, \cite{AJS80}, \cite{AS83}, \cite{BDL99b}, \cite{BRV99}, \cite{LR93}, \cite{ST94} and \cite{Sch80} for multiplication operators;
$\bullet$ \cite{BDL99} and \cite{Sha87} for composition operators; and
$\bullet$ \cite{MR00} for weighted composition operators).
Nevertheless, even more recently, this subject attracts the attention of many researchers and seems to be still an inexhaustible
source of inspirations (to name just a few, let us mention
$\bullet$ \cite{ABR18}, \cite{BR20}, \cite{DSZ11}, \cite{ES05}, \cite{GJL06}, \cite{GSZZ09}, \cite{DJL19}, \cite{Mle09} and \cite{Sch22} in the context of multiplication operators;
and $\bullet$ \cite{Bay23}, \cite{CHD03}, \cite{CZ03}, \cite{CZ07}, \cite{Dem11}, \cite{HKLRS12}, \cite{IIO11} and \cite{Lef09} in the context of weighted composition operators).

However, what sets the overwhelming majority of papers dealing with the essential norms of multiplication operators apart
is the fact that they focus almost exclusively on the case when the multiplication operator
$M_{\lambda} \colon f \rightsquigarrow \left[ \omega \rightsquigarrow \lambda(\omega) f(\omega) \right]$
acts from one to the same space (perhaps modified by different weights). In the simplest setting, this situation looks as follows
\begin{equation} \label{EQ: ess norm lp -> lp}
	\norm{M_{\lambda} \colon \ell_p \rightarrow \ell_p}_{ess} = \limsup_{n \rightarrow \infty} \abs{\lambda_n}
\end{equation}
and
\begin{equation*}
	\norm{M_{\lambda} \colon L_p \rightarrow L_p}_{ess} = \norm{\lambda}_{L_{\infty}} = \norm{M_{\lambda} \colon L_p \rightarrow L_p}.
\end{equation*}
This simply means that while the compact multiplication operators acting between $\ell_p$'s are precisely those whose symbols form a null sequences,
there are no compact multiplication operators acting between $L_p$'s except for the one generated by the function constantly equal to zero.

Our object in studying essential norms of the pointwise multipliers acting between two K{\" o}the spaces is basically twofold.

What can be seem as the main novelty here, we will consider the {\it non-algebraic setting}, that is, the situation in which the multiplication
operator $M_{\lambda} \colon f \rightsquigarrow \left[ \omega \rightsquigarrow \lambda(\omega) f(\omega) \right]$ acts between two {\it distinct} spaces.
In other words, we are interested in a precise estimation of the quantity
\begin{equation*}
	\norm{M_{\lambda} \colon X \rightarrow Y}_{ess} \coloneqq \inf \left\{ \norm{M_{\lambda} - K} \colon K \colon X \rightarrow Y \text{ is compact} \right\},
\end{equation*}
where $X$ and $Y$ are two K{\" o}the spaces defined on the same $\sigma$-finite measure space $(\Omega,\Sigma,\mu)$.
Let us observe that in this situation the symbol $\lambda$ of $M_{\lambda}$ belongs to the vector space
\begin{equation*}
	M(X,Y) \coloneqq \left\{ g \in L_0 \colon fg \in Y \text{ for all } f \in X \right\}
\end{equation*}
which, when equipped with the natural operator norm
\begin{equation*}
	\norm{g}_{M(X,Y)} \coloneqq \norm{M_g \colon f \rightsquigarrow fg}_{X \rightarrow Y} = \sup_{\norm{f}_X \leqslant 1} \norm{fg}_Y,
\end{equation*}
becomes a K{\" o}the space itself commonly called the {\it space of pointwise multipliers} between $X$ and $Y$.

\begin{scholium}
	The idea of considering the non-algebraic setting seems to be in the air for some time (see, for example, \cite{ABR18}, \cite{Bay23}, \cite{CHD03},
	\cite{CZ07}, \cite{Dem11}, \cite{ES05} and \cite{LM21}).
	In the context of K{\" o}the spaces and essential norms of multiplication operators acting between them, it seems to be a blank page.
	The reason for this, it seems, is not due to lack of interest.
	This is easy to justify because, on the one hand, our results answer a rather old Djakov and Ramanujan's question from \cite{DR00},
	five of six questions posed quite recently by Ramos-Fern{\' a}ndez in \cite{RF16},
	and the question asked by Ramos-Fern{\' a}ndez, Rivera-Sarmiento and Salas-Brown in \cite{RFRSSB19}.
	On the other, we also re-prove and improve some results from \cite{AGHLM18}, \cite{ABR18}, \cite{Ben96}, \cite{BR20} and \cite{Mle09}
	(in the context of algebraic setting, this list can easily be significantly extended, but see Section~\ref{SECTION: Neverending Story}).
\end{scholium}

Now, in stark contrast to the algebraic setting, when simply the space $M(X,X)$ coincide (up to the equality of norms) with $L_{\infty}$,
the space $M(X,Y)$ can be significantly different than $L_{\infty}$. For example, for $1 \leqslant q < p \leqslant \infty$ with $1/r = 1/q - 1/p$,
it is straightforward to see that $M(L_p, L_q) \equiv L_r$.
Moreover, since bounded multiplication operators $M_{\lambda} \colon X \rightarrow Y$ are
in correspondence with the K{\" o}the space $M(X,Y)$, so it is natural to expect\footnote{Directly from (\ref{EQ: ess norm lp -> lp}) it follows
that $M_{\lambda} \colon \ell_p \rightarrow \ell_p$ is compact if, and only if, $\lambda = \{\lambda_n\}_{n=1}^{\infty} \in c_0$. On the other hand,
the separable part of the space $M(\ell_p,\ell_p) \equiv \ell_{\infty}$, that is, $M(\ell_p,\ell_p)_o$, coincide with $c_0$.} that compact
multiplication operators $M_{\lambda} \colon X \rightarrow Y$
somehow correspond to the separable part of the space $M(X,Y)$, that is, the ideal $M(X,Y)_o$. Ideologically, it is a bit reminiscent of Calkin's
correspondence \cite{Cal41}, which in simplified version says, that there is a bijection between the operator ideal of bounded (respectively, compact)
operators on a separable Hilbert space and a K{\" o}the sequence space $\ell_{\infty}$ (respectively, $c_0$).

Secondly, and this seems to be the most pragmatic here, since the class of K{\" o}the spaces often serve as a basic skeleton for building other spaces
(to name just a few, Bergman spaces, Bloch spaces, Ces{\` a}ro spaces, Hardy spaces, Sobolev spaces and Tandori spaces can be build upon Lebesgue, Orlicz,
Lorentz, Marcinkiewicz and Musielak--Orlicz spaces), so it can be expected\footnote{The
reason for this is actually the result of Anderson and Shields saying that if $E$ is a K{\" o}the sequence space and $V$ is a vector space then
$M(V,E) = M(\text{solid}(V),E)$, where $\text{solid}(V)$ is the {\bf solid envelope} of $V$, that is, the intersection of all vector spaces
with the ideal property that contain $V$ as a subspace (see \cite[Lemma~2 and Lemma~3(i)]{AS76}).}
that the techniques developed by us will find some applications in certain classes of function spaces (sometimes not belonging to the family of K{\" o}the spaces).
We make here the first step in this direction by improving Mleczko's \cite{Mle09} result on compact Fourier multipliers acting on Banach spaces of
analytic functions on the unit disc (in particular, acting on abstract Hardy spaces), and Albanese, Bonet and Ricker's
\cite{ABR18}, Bonet and Ricker \cite{BR20}, and Bennett's \cite{Ben96} results on compact multipliers acting between classical Ces{\` a}ro and Tandori spaces.

What also seems to be noteworthy is that the non-algebraic setting is not just a simple generalization of what is currently known.
It requires a fresh approach avoiding, for example, the spectral theory techniques (after all, $M(X,X) \equiv L_{\infty}$
is a Banach algebra, which cannot be said about $M(X,Y)$ in general). Moreover, since the description of compact multipliers between
$X$ and $Y$ is inextricably linked with the description of a separable part of the space $M(X,Y)$, some more advanced theory related
to the so-called {\it arithmetic of K{\" o}the spaces}\footnote{Arithmetic of K{\" o}the spaces, broadly speaking, can be described as a certain collection
of techniques dealing with constructions such as the pointwise product, the pointwise multipliers and the symmetrization (we refer to \cite{KLM19} and Section~\ref{SECTION: examples}
for some illustrations of these ideas and examples; cf. also \cite{CS14}, \cite{CDSP08}, \cite{KLM12, KLM14} and \cite{Sch10}).}
and {\it local properties}\footnote{The {\it order continuity} property will play the main role here, because \enquote{globally} it is an analogue
of the separability, but can also be \enquote{localized} at a single point, that is, it make perfect sense to ask if a function in a K{\" o}the space
is order continuous or not.} of K{\" o}the spaces, comes into play.

After all this propaganda, let us only add that the proofs we propose seems to lie much closer to the methods of functional analysis, in particular,
to the theory of function spaces and to the operator theory, than to rely on some {\it ad hoc} tricks. As usual, more sophisticated methods put certain
phenomena in a more complete and interesting context and, we hope, that this is the case here.

\subsection{Neverending story} \label{SECTION: Neverending Story}

We have one more, slightly less official, but short, story to tell. Initially, the reason for our work was Voigt's
preprint \cite{Voi22} posted on {\tt arXiv.org} in late March, where the essential norm of the multiplication operators acting from $L_p$ into itself is calculated.
Already a superficial reading of this work, as well as the works of authors cited therein, almost immediately convinced us that everything should work
in much more general framework. However, we also could not shake off the impression that the problem raised in \cite{Voi22} is too classical not to have
a solution in the existing literature. 
To our surprise, just a few weeks later, a commentary on Voigt's article in the form of Schep's preprint \cite{Sch22} appeared. In introduction to his paper, Schep writes
\begin{quote}
	\begin{small}
		{\it \enquote{We refer to a recent preprint of J. Voigt \cite{Voi22} for some references. It seems that these authors are not aware that these multiplication
		operators are examples of order bounded band preserving operators (or orthomorphisms). (...) The current paper was prompted by the preprint \cite{Voi22},
		which didn't use our results form \cite{Sch80}. By using the term abstract multiplication operator in our title we hope that future duplication
		of our results by authors working on these questions for concrete function spaces will be avoided\footnote{Unfortunately, in retrospect, it seems
		that this appeal did not necessarily fall on fertile ground (see, for example, \cite{RFRSSB24}).} by a web search or MathSciNet search.
		(...) The current paper describes completely the essential spectrum and norm of an operator $T \in Z(E)$ for an arbitrary Banach lattice $E$.}}
	\end{small}
\end{quote}
Although it is not our intention to breach open doors here, we feel that the last word has not yet been given, because as we have already tried to explain above
an analogous questions for something, what we call here the {\it non-algebraic setting}, seems to be almost completely untouched.

\subsection{Outline}

Now, we describe the organization of this work and mention some other results contained herein.

We begin in Section~\ref{SECTION: Toolbox} with rather rapid recollection of required definitions and notation that will be used through this paper.
As it seems to us, the only slightly unconventional approach is presented here in the context of Rademacher functions (see Construction~\ref{DEF: Abstract Rademachers}).

Section~\ref{SECTION: ESS NORMS} is the main part of the paper.
At the beginning we compute the essential norms of multiplication operators acting between two distinct K{\" o}the function spaces
(Theorem~\ref{THM: essnorm multipliers function}) or two distinct K{\" o}the sequence spaces (Theorem~\ref{THM: Essential norm Kothe sequence})
culminating in Theorem~\ref{THM: ess norm func + seq}, where the essential norm of the multiplication operator acting between two
distinct K{\" o}the spaces is calculated.
We then explain how these results can be viewed through the prism of measures of non-compactness (Theorem~\ref{thm: ess norm ri})
and Pitt's theorem (Proposition~\ref{PROP: Pitt's multiplication}).
Here we also answer Djakov and Ramanujan's question from \cite{DR00} (Remark~\ref{REMARK: Djakov-Ramanujan}).
In the following, we will focus our attention on the weak essential norms of multiplications operators acting between two distinct K{\" o}the
function spaces (Theorem~\ref{THM: weak ess norm}). {\it En route} to the proof on this result, we will use Flores, Tradacete and Troitsky's
notion of disjointly homogeneous spaces in tandem with some recent results from \cite{LTM22} about the Dunford--Pettis criterion in K{\" o}the
function spaces. Closing this thread, we will give an example of a multiplication operator acting between two Marcinkiewicz spaces which
is not weakly compact (Example~\ref{EXP: not weakly compact between Marcinkiewicz}), and of a weakly compact multiplication operator
acting between Lorentz spaces with a symbol which is not order continuous (Example~\ref{EXP: Lorentz symbol not OC but WC}).
Finally, the last theme that we will devote our attention is the interplay between multiplier and its symbol.
More precisely, relying on certain ideas from \cite{KKT22} we will show an elegant result (Theorem~\ref{Proposition: odleglosci = odleglosci})
which equates the (weak) essential norm of the multiplication operator $M_{\lambda} \colon X \rightarrow Y$ with the distance from its symbol
$\lambda$ to the ideal $M(X,Y)_o$ of all order continuous functions in $M(X,Y)$.
We also propose aesthetically pleasing expression of the weak essential norm of multipliers acting between two rearrangement invariant function
spaces (Corollary~\ref{COR: ess r.i. function}).

Then, in Section~\ref{SECTION: Applications}, we will harness our techniques and results obtained so far (along with some ideas developed in \cite{KT24}) to work.
To begin with, by calculating the essential norms of multiplication operators acting between two distinct abstract Ces{\` a}ro sequence spaces
or two distinct abstract Tandori sequence spaces (Theorem~\ref{THM: ess norm Cezaro and Tandori function}), we answer Ramos-Fern{\' a}ndez,
Rivera-Sarmiento and Salas-Brown's question from \cite{RFRSSB19}.
To illustrate our, as nomenclature itself suggests, \enquote{abstract} approach, we will consider (see Example~\ref{EXAMPLE: compact cesp -> cesq faktoryzacja},
Example~\ref{Remark: ces_p -> ces_q [ABR18]} and Remark~\ref{REMARK: jak wynioskowac Tandori lp -> lq}) pointwise multipliers acting between
the classical Ces{\` a}ro sequence $ces_p$ spaces, and, their duals, the so-called Tandori sequence spaces $\widetilde{\ell_p}$.
This way we will recreate Albanese, Bonet and Ricker's result from \cite{ABR18}, Bonet and Ricker's result from \cite{BR20} and some of Bennett's results from \cite{Ben96}.
Next, we do more or less the same in the case of abstract Ces{\` a}ro function spaces (Theorem~\ref{THM: ess norm Cezaro and Tandori function}).
Here we re-prove Al Alam, Gaillard, Habib, Lef{\` e}vre and Maalouf's result from \cite{AGHLM18}
We will close this section by computing the essential norms of Fourier multipliers acting from a Banach space of analytic functions on the unit disc
$\mathbb{D}$, in particular, from an abstract Hardy space built upon a K{\" o}the function space defined the unit circle $\mathbb{T}$, into some
K{\" o}the sequence space (Theorem~\ref{THM: ess norm holomorphic}). This improves Mleczko's main result from \cite{Mle09}.

The last Section~\ref{SECTION: examples} is devoted to examples. Here, we calculate the space $M(X,Y)$ for the different configurations of spaces $X$ and $Y$,
coming from some more concrete classes of K{\" o}the sequence spaces, like Musielak--Orlicz sequence spaces (most importantly, Nakano sequence spaces;
see Remark~\ref{REMARK: compact multipliers MO} and Theorem~\ref{PROPOSITION: Compact multipliers between Nakano})
or Lorentz and Marcinkiewicz spaces (see Theorem~\ref{Prop: computing multipliers from m and to lambda}).
This can be also thought of as a tutorial to the arithmetic of K{\" o}the spaces having their roots in \cite{KLM12}, \cite{KLM14}, \cite{KLM19} and \cite{Sch10}
(see Lemma~\ref{LEMMA: a'la KLM12/14}), and as a reincarnation of some Kolwicz's ideas about symmetrization construction from \cite{Kol16} (see Lemma~\ref{LEMMA: a'la Kolwicz OC}).

All sections are sprinkled with examples and open problems which we collectively place in Section~\ref{SECTION: Open ends}.

Finally, in Appendix~\ref{Appendix B} we present an example (Example~\ref{PROPOSITION: Appendix A - counterexample factorization})
of a certain Orlicz space $\ell_{\widetilde{M}}$ with the property that
\begin{equation*}
	\ell_{\widetilde{M}} \odot M(\ell_{\widetilde{M}},\ell_2) \neq \ell_2,
\end{equation*}
that is, $\ell_{\widetilde{M}}$ does not factorize through a Hilbert space $\ell_2$, but $M(\ell_{\widetilde{M}}, \ell_2) \neq \ell_{\infty}$.
To the best of our knowledge, this is the first example of such a space.

\section{{\bf Statements and declarations}}

\subsection{Acknowledgments}
We are very grateful to Lech Maligranda for careful reading of an early draft for our work and all his helpful advise,
and to Pawe{\l} Kolwicz for all the hints that led us back on the right track.
We thank Anton Schep for drawing our attention to \cite{PS88}.
Moreover, both authors would like to express their thanks to Damian Kubiak for sharing his paper \cite{Kub14} with them.
We also thank Karol Le{\' s}nik for some useful discussions.
Last but not least, we feel obligated to our dear friends Przemys{\l}aw Trela and Mariusz {\. Z}yluk for some, as usual,
interesting comments about our work.
We also apologize for the long delay in preparing of the final version of this work.

\subsection{Funding}
Our research was carried out at the Pozna{\' n} University of Technology as a part of the Rector's grant.

\subsection{Conflict of interest}
The authors declare that they have no known competing financial interest or personal relationship that could have appeared
to influence the work reported in this paper.

\subsection{Data availibility}
All data generated or analysed during this study are included in this article.

\section{{\bf Toolbox}} \label{SECTION: Toolbox}

Our notation is rather standard and is, more or less, a mix of notations used in Bennett and Sharpley \cite{BS88},
Lindenstrauss and Tzafriri \cite{LT77}, \cite{LT79} and Pietsch \cite{Pie80}.
Below, for the reader's convenience, we provide in detail the most important definitions and the necessary facts.
Other more specific terminology will be introduced in context when needed.

\begin{notation}[Vinogradov's notation]
	Given two quantities, say $A$ and $B$, depending (maybe) on certain parameters, we will write $A \approx B$ understanding that there exist
	absolute constants (that is, independent of all involved parameters), say $c, C > 0$, such that $cA \leqslant B \leqslant CA$.
\end{notation}

\begin{notation}[Definition]
	Defined terms are always distinguished from the rest of the text by {\bf bold} font (also when, instead of providing a full definition,
	we only refer to the literature). Sometimes (but there is no rule here), we specify them additionally in the form of {\bf Definition}.
\end{notation}

\begin{notation}[Digression]
	Some additional comments not necessary for the general understanding of this work have been placed in specific
	parts of the text titled {\bf Digression}.
\end{notation}

\begin{notation}
	By tradition, we will use \enquote{halmos} $\blacksquare$ at the end of each proof. Anyway, we find it useful to also use the symbol
	{\large $\blacktriangle$} to clearly indicate the end of each Remark, Example and Digression.
\end{notation}

\begin{notation}
	Let, as usual, $\mathbb{N} = \{1,2, ...\}$, $\mathbb{R}$ and $\mathbb{C}$ denote the set of natural numbers,
	reals and complex numbers, respectively. Moreover, let $\mathbb{N}_0 = \{0, 1, 2, ...\}$.
\end{notation}

\begin{notation}
	Depending on the situation, we will write $x = \{ x_n \}_{n=1}^{\infty}$ and $x = \sum_{n=1}^{\infty} x_n e_n$.
	Both notations mean exactly the same. Similarly, sometimes we will write $x = \sum_{i=n}^{\infty} x_n e_n$ and
	sometimes $x\chi_{\{n,n+1,...\}}$. Here, $\chi_A$ is the characteristic function of a set $A$.
\end{notation}

\begin{notation}
	Sometimes, in various contexts, we will use the following abbreviation $f(\infty) \coloneqq \lim_{t \rightarrow \infty} f(t)$.
\end{notation}

\subsection{Recollections on function spaces}
Let $(\Omega, \Sigma, \mu)$ be a complete\footnote{This means that any subset of a set of measure zero is measurable.} and
$\sigma$-finite measure space. Let, moreover, $L_0(\Omega,\Sigma,\mu)$ (briefly, $L_0(\Omega)$, or even $L_0$, as long as it does
not lead to confusion) be the set consisting of all equivalence classes, modulo equality almost everywhere, of locally
integrable\footnote{We will say, that a function $f$ is {\bf locally integrable} if it is measurable and for every $A \in \Sigma$
with $\mu(A) < \infty$ we have that $\int_{A} \abs{f(\omega)} d\mu(\omega) < \infty$.} real or complex valued functions on $\Sigma$.
A Banach space $X$ is called a {\bf K{\" o}the space} (also a {\bf Banach ideal space} or a {\bf Banach function space}) if the following two conditions hold:
\begin{itemize}
	\item[$\bullet$] $X$ is a linear subspace of $L_0(\Omega,\Sigma,\mu)$;
	\item[$\bullet$] if $\abs{f(\omega)} \leqslant \abs{g(\omega)}$ almost everywhere on $\Omega$, with $f$ measurable
	and $g \in X$, then $f \in X$ and $\norm{f}_X \leqslant \norm{g}_X$
\end{itemize}
(see \cite[pp.~1--3]{BS88}, \cite[Definition~1.1, p.~40]{KPS82}, \cite[Definition~1.b.17, p.~28]{LT79} or \cite[p.~161]{Ma89}).
However, unless we write otherwise, we will also assume that
\begin{itemize}
	\item[$\bullet$] for every $A \in \Sigma$ with $\mu(A) < \infty$ the characteristic function $\chi_{A}$ belongs to $X$.
\end{itemize}

\begin{remark}[K{\" o}the function and sequence spaces]
	Hereafter, whenever we use the term {\bf K{\" o}the sequence space} we mean a K{\" o}the space over a purely atomic measure space, while, the name
	{\bf K{\" o}the function space} is reserved for a K{\" o}the space over a non-atomic measure space (the choice of vocabulary
	seems, perhaps, a bit misplaced here, because formally every sequence is also a function, but this nomenclature seems to be
	well-rooted in the literature).
	It follows from Maharam's theorem that every K{\" o}the space is, roughly speaking, a mixture
	of a certain K{\" o}the sequence and K{\" o}the function space.
	\demo
\end{remark}

\begin{notation}[Continuous inclusions]
	Due to the closed graph theorem, a formal inclusion of two Banach ideal spaces, say $X$ and $Y$, is a bounded operator and to clearly
	emphasize this fact we will write $X \hookrightarrow Y$ instead of just $X \subset Y$. Moreover, in some situations, we will use the
	symbol $X \overset{C}{\hookrightarrow} Y$ to indicate that the norm of the inclusion $X \subset Y$ is not greater than $C > 0$.
\end{notation}

\begin{notation}
	The symbols $X = Y$ and $X \equiv Y$ mean that the spaces $X$ and $Y$ are the same as sets and their norms are equivalent, respectively, equal.
\end{notation}

Given a K{\" o}the space $X$ on $(\Omega,\Sigma,\mu)$ the {\bf K{\" o}the dual space $X^{\times}$}
(or the {\bf associated space}) of the space $X$ is defined by the norm .
Recall that $X \overset{1}{\hookrightarrow} X^{\times \times}$ and we have the equality $X \equiv X^{\times \times}$ if, and only if, the norm in $X$
has the so-called {\bf Fatou property}, that is, the closed unit ball $\text{Ball}_X$ is also closed in $L_0(\Omega)$
with respect to the topology of local convergence in measure.

\begin{notation}[Duality bracket]
	Sometimes it will be convenient to use the symbol $\langle \cdot, \cdot \rangle$ for the natural pairing between the K{\" o}the space $X$ and its
	topological dual $X^*$ or its K{\" o}the dual $X^{\times}$.
\end{notation}

Moreover, for a K{\" o}the space $X$ on $(\Omega,\Sigma,\mu)$ and a positive and locally integrable function $w \colon \Omega \rightarrow \mathbb{R}$
(briefly, a {\bf weight} on $\Omega$) by a {\bf weighted K{\" o}the space $X(w)$} we will understand a vector space
\begin{equation*}
	X(w) \coloneqq \{f \in L_0(\Omega) \colon fw \in X\}
\end{equation*}
together with the norm $\norm{f}_{X(w)} \coloneqq \norm{fw}_X$. There should be no doubt that the space $X(w)$ is a K{\" o}the space itself.
It is also clear that $X(w)^{\times} \equiv X^{\times}(1/w)$.

We will say that a function, say $f$, in a K{\" o}the space $X$ is {\bf order continuous} (or it has an {\bf order continuous norm})
if for any sequence $\{f_n\}_{n=1}^\infty$ in $X$ with $0 \leqslant f_n \leqslant \abs{f}$ and $f_n \downarrow 0$ point-wise we have
that $\norm{f_n}_X \downarrow 0$.
By $X_o$ we denote the subspace of all order continuous elements of $X$. It turns out that $X_o$ is an {\bf order ideal} in $X$, that is to say,
$X_o$ is a closed linear subspace of $X$ with the ideal property\footnote{This means that if $f \in X_o$, $g \in X$ and $\abs{g(\omega)} \leqslant \abs{f(\omega)}$,
then also $g \in X_o$.} (see \cite[Theorem~3.8, p.~16]{BS88}).
Whenever $X = X_o$, that is, every element of $X$ is order continuous, the space $X$ is said to be {\bf order continuous} (pedantically, {\bf $\sigma$-order continuous}).
Notice that this fact is equivalent to the separability of the space $X$ (see \cite[Theorem~5.5, p.~27]{BS88}).

In general, we refer to \cite{Wn99} for much more information about the order continuity property. However, let us add that
\begin{itemize}
	\item[$\bullet$] The dual space $X^*$ of a K{\" o}the space $X$ can be naturally identified with the the K{\" o}the dual space $X^{\times}$ if, and only if,
		the space $X$ is order continuous (see \cite[Corollary~4.3, p.~23]{BS88});
	\item[$\bullet$] A K{\" o}the space $X$ with the Fatou property is reflexive if, and only if, the space $X$ itself and its K{\" o}the dual $X^{\times}$
		are both order continuous (see \cite[Corollary~4.4, p.~23]{BS88}).
\end{itemize}

\begin{remark}
	In practice, it is also useful to observe that a function $f$ in a K{\" o}the space $X$ is order continuous
	if, and only if, $\lim_{n \rightarrow \infty} \norm{f\chi_{A_n}}_X = 0$ for every decreasing sequence of measurable
	sets $\{A_n\}_{n=1}^{\infty}$ with the intersection of measure zero (see \cite[Proposition~3.5, p.~15]{BS88}).
	Consequently, in the particular case when $X$ is a K{\" o}the sequence space, a function $x = \{x_n\}_{n=1}^{\infty}$
	from $X$ is order continuous if, and only if, $\lim_{n \rightarrow \infty} \norm{ \sum_{i=n}^{\infty} x_i e_i }_X = 0$.
	\demo
\end{remark}

Let $(\Omega, \Sigma, \mu)$ be a {\bf resonant measure space}, that is, either non-atomic or purely atomic with all atoms having equal measure,
complete and $\sigma$-finite measure space (see, e.g., \cite[Theorem~2.7, p.~51]{BS88}).
A K{\" o}the space $X$ of measurable functions on $\Omega$ is called {\bf rearrangement invariant} (or {\bf symmetric}; sometimes we will
also write briefly {\bf r.i.}) if it satisfies the following condition:
\begin{itemize}
	\item[$\bullet$] if $f$ and $g$ are equi-measurable and $g \in X$, then $f \in X$ and $\norm{f}_X = \norm{g}_X$;
\end{itemize}
cf. \cite[Definition~4.1, p.~59]{BS88}, \cite[Definition~2.a.1, pp.~117--118]{LT79} or \cite[p.~90]{KPS82}. Equivalently,
\begin{itemize}
	\item[$\bullet$] if $f^*(t) \leqslant g^*(t)$ for almost all $t > 0$ and $g \in X$, then $f \in X$ and $\norm{f}_X \leqslant \norm{g}_X$,
\end{itemize}
where $f^*$ denote the {\bf non-decreasing rearrangement} of $f$ (see \cite[p.~90]{KPS82}; cf. \cite[Definition~1.5, p.~39]{BS88} and \cite[p.~117]{LT79}).
Every K{\" o}the space is a Banach lattice in a sense of Meyer-Nieberg \cite{MN91}. Moreover, every order continuous Banach lattice with the weak
unit is order isometric to a K{\" o}the space (see \cite[Theorem~1.b.14, p.~25]{LT79} and references given there).
It follows from the {\it Luxemburg representation theorem} (see \cite[Theorem~4.10, p.~62]{BS88}) that when considering rearrangement invariant
spaces it is enough to limit oneself to the case of one of the three separable measure spaces:
\begin{itemize}
	\item[$\bullet$] the natural numbers $\mathbb{N} = \{1,2,...\}$ with the counting measure $\#$,
	\item[$\bullet$] the unit interval $(0,1) $ with the usual Lebesgue measure $m$,
	\item[$\bullet$] the semi-axis $(0,\infty)$ with the usual Lebesgue measure $m$;
\end{itemize}
see \cite[pp.~114--115]{LT79} (note, however, that in \cite{KPS82} one consider only r.i. spaces defined on the semi-axis and in \cite{BS88}
the Fatou property is essentially included in their definition; see~\cite[p.~90]{KPS82} and \cite[Definition~1.1, p.~2]{BS88}).

\begin{notation}
	Again, we take here the convention that whenever a rearrangement space $X$ is defined over $(\mathbb{N}, 2^{\mathbb{N}}, \#)$ we will call it
	a {\bf rearrangement invariant sequence space}, but in the other two cases, we will use the name {\bf rearrangement invariant function space}.
\end{notation}

\begin{remark}
	The class of rearrangement invariant spaces includes many important and classical examples of spaces, like
		$\bullet$ Lebesgue--Riesz spaces $L_p$;
		$\bullet$ Orlicz spaces $L_M$ (see \cite[\S4.8]{BS88}, \cite[\S3]{Mal83});
		$\bullet$ Lorentz spaces $\Lambda_{\varphi}$ (see \cite[Definition~5.12, pp.~71--72]{BS88}, \cite[p.~107]{KPS82});
		$\bullet$ Lorentz spaces $L_{p,q}$ (see \cite[Definition~2.b.8, p.~142]{LT79}; cf. also Section~\ref{SECTION: ESS NORMS});
		$\bullet$ Marcinkiewicz spaces $M_{\varphi}$ (see \cite[Definition~5.7, p.~69]{BS88}, \cite[p.~114]{KPS82});
		$\bullet$ Lorentz--Zygmund spaces $L_{p,q}(\log L)^{\alpha}$ (see \cite[Definition~6.13, p.~253]{BS88});
		$\bullet$ Grand Lebesgue spaces $L_{p)}$ and, their duals, small Lebesgue spaces $L_{(p}$;
		$\bullet$ Morrey--Campanato spaces $L_{p,\lambda}$;
		and
		$\bullet$ some combined forms of these spaces, like, for example, Orlicz--Lorentz spaces $\Lambda_{M,w}$, Morrey--Orlicz spaces $L_{M,\lambda}$ and so on.
	\demo
\end{remark}

For general properties of K{\" o}the and rearrangement invariant spaces we refer to Bennett and Sharpley \cite{BS88},
Kre{\u \i}n, Petunin and Semenov \cite{KPS82}, Lindenstrauss and Tzafriri \cite{LT77} and \cite{LT79}, Luxemburg \cite{Lux67}, Maligranda \cite{Ma89}
and Zaanen \cite[Chapter~15]{Zaa67}.
More information on rearrangement invariant spaces defined {\it not} necessary on a $\sigma$-finite measure space can be found, for example, in Folland \cite{Fol99}.
For much more information about Banach lattices and, even more generally, Riesz spaces see, for example, Aliprantis and Burkinshaw \cite{AB85},
Meyer-Nieberg \cite{MN91} and Zaanen \cite{Za83}.

\subsection{Boyd indices}

The {\bf lower Boyd index} $\alpha_X$ and {\bf upper Boyd index} $\beta_X$ of a rearrangement invariant function space $X$
is defined by\footnote{With the convention that, if $\norm{D_s}_{X \rightarrow X} = 1$
for some $s > 1$, we put $\alpha_X = \infty$, and, similarly, if $\norm{D_s}_{X \rightarrow X} = 1$
for all $0 < s < 1$, we put $\beta_X = \infty$.}
\begin{equation*}
	\alpha_X \coloneqq \lim\limits_{s \rightarrow \infty} \frac{\log s}{\log \norm{D_s}_{X \rightarrow X}}
\end{equation*}
and, respectively,
\begin{equation*}
	\beta_X \coloneqq \lim\limits_{s \rightarrow 0^{+}} \frac{\log s}{\log \norm{D_s}_{X \rightarrow X}}.
\end{equation*}
Here, $D_s$ with $s > 0$ stands for the {\bf dilation operator} defined for a measurable function $f$ on $(0,\infty)$
as $(D_s f)(t) \coloneqq f(t/s)$, while for a measurable function $f$ on $(0,1)$ as
	\[
	(D_s f)(t) \coloneqq \left\{ 
	\begin{array}{ccc}
		f(t/s)  & \text{ for } & t \leqslant \min\{1,s\} \\ 
		0  & \text{ for } & s \leqslant t < 1
	\end{array}
	\right.
	\]
Dilation operators $D_s$ are bounded on any rearrangement invariant function space $X$ with $\norm{D_s}_{X \rightarrow X} \leqslant \max\{1,s\}$
(see \cite[p.~148]{BS88}, \cite[pp.~96--98]{KPS82} and \cite[pp.~130--131]{LT79}).
Moreover, it follows from the definition, that $1 \leqslant \alpha_X \leqslant \beta_X \leqslant \infty$.

Following Lindenstrauss and Tzafriri \cite{LT79} (see \cite[p.~131]{LT79}; see also \cite[p.~165]{KPS82}),
the Boyd indices can be defined also for rearrangement invariant sequence spaces.
The difference here is that the dilation operators $D_m$ and $D_{1/m}$ with $m \in \mathbb{N}$ are defined by
\begin{equation*}
	D_m(x) \coloneqq \left\{ x_{[\frac{m-1+n}{m}]} \right\}_{n=1}^\infty = \{ \underbrace{x_1, x_1, ..., x_1}_{m}, \underbrace{x_2, x_2, ...,x_2}_{m}, x_3, ... \}
\end{equation*}
and
\begin{equation*}
	D_{1/m}(x) \coloneqq \left\{ \frac{1}{m} \sum_{k=(n-1)m+1}^{nm} x_k \right\}_{n=1}^\infty
	= \left\{ \frac{1}{m}\sum_{k=1}^m x_k, \frac{1}{m}\sum_{k=m+1}^{2m} x_k, ..., \frac{1}{m}\sum_{(n-1)m+1}^{nm}x_k, ... \right\}.
\end{equation*}

The {\bf lower dilation exponent $\gamma_\varphi$} and the {\bf upper dilation exponent $\delta_\varphi$} of an arbitrary positive function
$\varphi$ on $(0,\infty)$ is defined as
\begin{equation}
	\gamma_\varphi \coloneqq \lim\limits_{s \rightarrow \infty} \frac{\log s}{ \log \thickbar{\varphi}(s) }
\end{equation}
and, respectively,
\begin{equation}
	\delta_\varphi \coloneqq \lim\limits_{s \rightarrow 0^{+}} \frac{\log s}{ \log \thickbar{\varphi}(s) }.
\end{equation}
Here, $\thickbar{\varphi}(s) \coloneqq \sup_{t > 0} \varphi(st) / \varphi(t)$ for $s > 0$.
For a quasi-concave function\footnote{A non-negative function, say $\varphi$, defined on the half-line $(0,\infty)$ is said to be {\bf quasi-concave}
if $\varphi$ is non-decreasing, vanishes only at zero and $t \rightsquigarrow \varphi(t) / t$ is decreasing. Clearly, every non-negative concave function
on $(0,\infty)$ that vanishes only at the origin is quasi-concave, but not {\it vice versa} (see \cite[p.~69]{BS88}).} $\varphi$ on $(0,\infty)$,
we have $1 \leqslant \gamma_\varphi \leqslant \delta_\varphi \leqslant \infty$ (see \cite[Theorem~4.12, p.~107]{KPS82} and \cite[pp.~81--95]{Ma89}).

\subsection{Essential norms}
Denote by $\mathscr{L}$ the class of all operators, which we {\it always} assumed to be linear and bounded, between all Banach spaces.
For two Banach spaces, say $X$ and $Y$, let $\mathscr{L}(X,Y)$ be the set of all operators from $X$ into $Y$. More generally, for
$\mathscr{I} \subset \mathscr{L}$ any two Banach spaces $X$ and $Y$ we set $\mathscr{I}(X,Y) \coloneqq \mathscr{I} \cap \mathscr{L}(X,Y)$.
If the need arises, we will shorten the symbols $\mathscr{L}(X,X)$ and $\mathscr{I}(X,X)$ to $\mathscr{L}(X)$ and, respectively, $\mathscr{I}(X)$.
Following Pietsch \cite{Pie80} (see also \cite{DJP01}) we recall that a class $\mathscr{I}$ is called a {\bf operator ideal} (in the sense of Pietsch)
if the following conditions are satisfied:
\begin{itemize}
	\item[$\bullet$] for any one-dimensional Banach space $X$ the identity operator $\text{id}_X \colon X \rightarrow X$ belongs to $\mathscr{I}$,
	\item[$\bullet$] for all Banach spaces, say $X$ and $Y$, the set $\mathscr{I}(X,Y)$ is a linear subspace of $\mathscr{L}(X,Y)$,
	\item[$\bullet$] for any four Banach spaces, say $W$, $X$, $Y$ and $Z$, the composite operator $UTS$ belongs to $\mathscr{I}(W,Z)$
	whenever $S$ belongs to $\mathscr{L}(Y,Z)$, $T$ to $\mathscr{I}(X,Y)$, and $U$ to $\mathscr{L}(W,X)$.
\end{itemize}
An operator ideal $\mathscr{I}$ is said to be a {\bf closed operator ideal} provided $\mathscr{I}(X,Y)$ is closed in $\mathscr{L}(X,Y)$ for any two Banach
spaces $X$ and $Y$. Some very well-known classes of operator ideals, like compact operators, weakly compact operators, strictly singular operators
or completely continuous operators are examples of closed operator ideals.

For any closed operator ideal $\mathscr{I}$ and any operator, say $T$, acting between two Banach spaces $X$ and $Y$ by the {\bf essential
norm $\norm{ T \colon X \rightarrow Y }_{\mathscr{I}}$} of $T$ with respect to an operator ideal $\mathscr{I}$ we mean
\begin{equation}\label{def: general essnorm}
	\norm{T \colon X \rightarrow Y}_{\mathscr{I}} \coloneqq \norm{T \colon X \rightarrow Y}_{{\mathscr L}(X,Y)/\mathscr{I}(X,Y)}
		= \inf\{\norm{T - J}_{X \rightarrow Y} \colon J \in \mathscr{I}(X,Y)\}
\end{equation}
(cf. with the definitions given in \cite[p.~96]{Tyl95} and \cite[Definition~0.3]{Lef09}). The last expression on the right in (\ref{def: general essnorm})
is, of course, nothing else but a natural quotient norm on the space ${\mathscr L}(X,Y)/\mathscr{I}(X,Y)$, that is to say,
\begin{equation*}
	\norm{T \colon X \rightarrow Y}_{{\mathscr L}(X,Y)/\mathscr{I}(X,Y)} = \text{dist}_{\mathscr{L}(X,Y)}(T \colon X \rightarrow Y,\mathscr{I}(X,Y)).
\end{equation*}
The above definition, in its two most common versions, looks as follows
\begin{itemize}
	\item[$\bullet$] the quotient algebra ${\mathscr L}(X,Y)/{\mathscr K}(X,Y)$, where ${\mathscr K}(X,Y)$ denote the ideal of compact operators from $X$ to $Y$,
		is called the {\bf Calkin algebra} and the corresponding quotient norm $\norm{T \colon X \rightarrow Y}_{{\mathscr L}(X,Y)/\mathscr{K}(X,Y)}$
		(or, a bit more concisely, $\norm{T \colon X \rightarrow Y}_{ess}$) is called the {\bf essential norm} of $T$;
	\item[$\bullet$] alike, the quotient algebra ${\mathscr L}(X,Y)/{\mathscr W}(X,Y)$, where ${\mathscr W}(X,Y)$ denote the ideal of weakly compact operators
		acting between $X$ and $Y$, is called the {\bf weak Calkin algebra} and the corresponding quotient norm $\norm{T \colon X \rightarrow Y}_{{\mathscr L}(X,Y)/\mathscr{W}(X,Y)}$
		(briefly, just $\norm{T \colon X \rightarrow Y}_w$) is called the {\bf weak essential norm} of $T$.
\end{itemize}
Clearly, an operator $T$ acting between $X$ and $Y$ belongs to the closed operator ideal $\mathscr{I}(X,Y)$ if, and only if,
\begin{equation*}
	\norm{T \colon X \rightarrow Y}_{{\mathscr L}(X,Y)/\mathscr{I}(X,Y)} = 0.
\end{equation*}

The concept of the essential norm has in general much more to offer
(see, for example, \cite{AT87}, \cite{AS83}, \cite{BV99}, \cite{AJS80}, \cite{Dil87}, \cite{ES05}, \cite{LM08}, \cite{LMN22},
\cite{LS71}, \cite{PS88}, \cite{ST94}, \cite{Sch89} and \cite{Tyl95}).

\subsection{Pointwise multipliers}

Let $X$ and $Y$ be two K{\" o}the spaces defined on the same $\sigma$-finite measure space $(\Omega,\Sigma,\mu)$.
The space of {\bf pointwise multipliers $M(X,Y)$} is defined as a vector space
\begin{equation*}
	M(X,Y) \coloneqq \{f \in L_0(\Omega) \colon fg \in Y \text{ for all } g \in X\}
\end{equation*}
equipped with the natural operator norm $\norm{f}_{M(X,Y)} \coloneqq \sup_{\norm{g}_X = 1} \norm{fg}_Y$.

\begin{scholium}
	Note that the space $M(X,L_1)$ is nothing else but the K{\" o}the dual space $X^{\times}$ of the space $X$, so we can naturally
	think of the construction $M(X,Y)$ as a generalized K{\" o}the dual of $X$ (cf. \cite{CDSP08} and \cite{MP89}). Moreover,
	looking from a different perspective, the space $M(X,Y)$ can naturally be understood as a \enquote{pointwise quotient} of the
	space $Y$ by $X$.
	\demo
\end{scholium}

Some properties and examples of this construction can be found, for example, in \cite{CS14}, \cite{CDSP08}, \cite{KLM12}, \cite{KLM14}, \cite{KLM19},
\cite{LT21}, \cite{LTM22}, \cite{MP89} and \cite{Sch10} (see also Bennett's memoir \cite[\S14]{Ben96}, Grosse-Erdmann's book \cite{GE98} and references given there).
Below, we will collect only a few of the most important from the point of view of our future needs (see \cite{AS76}, \cite{CDSP08}, \cite{KLM12}, \cite{MP89} and \cite{Sch10}).
Hereinafter, unless stated otherwise, $X$, $Y$ and $Z$ will denote three K{\" o}the spaces defined over the same measure space.
\begin{itemize}
	\item[$\bullet$] If $X \hookrightarrow Y$, then $M(Y,Z) \hookrightarrow M(X,Z)$. Moreover, if $Y \hookrightarrow Z$, then $M(X,Y) \hookrightarrow M(X,Z)$.
	\item[$\bullet$] $L_{\infty} \hookrightarrow M(X,Y)$ if, and only if, $X \hookrightarrow Y$. In particular, $M(X,X) \equiv L_{\infty}$.
	\item[$\bullet$] $M(X,Y) \hookrightarrow M(Y^{\times},X^{\times}) = M(X^{\times \times}, Y^{\times \times})$. Moreover, if $Y$ has the Fatou property,
		then $M(X,Y) = M(Y^{\times},X^{\times})$.
	\item[$\bullet$] If $Y$ has the Fatou property, then $M(X,Y)$ share this property too.
	\item[$\bullet$] If $X$ and $Y$ are two r.i. function spaces, then $M(X,Y)$ is also a r.i. function space (the same is true for r.i. sequence spaces;
		see the first part of the proof of Lemma~\ref{LEMMA: a'la KLM12/14}).
\end{itemize}

\begin{example}
	Note that for $1 \leqslant q < p < \infty$ with $1/r = 1/q - 1/p$ we have the following identifications $M(L_p,L_q) \equiv L_r$ and $M(\ell_p,\ell_q) \equiv \ell_r$,
	On the other hand, if $1 \leqslant p < q < \infty$, then $M(L_p,L_q) = \{0\}$, but $M(\ell_p,\ell_q) \equiv \ell_{\infty}$.
	\demo
\end{example}

\begin{remark}
	The space of pointwise multipliers $M(X,Y)$ and the space of all {\bf multiplication operators $M_{\lambda}$}
	(sometimes also called {\bf diagonal operators}) acting between $X$ and $Y$ are isometrically isomorphic via
	the mapping $\Upsilon \colon \lambda \rightsquigarrow \left[ M_{\lambda} \colon f \rightsquigarrow \lambda f \right]$.
	\demo
\end{remark}

\begin{notation}
	Function $\lambda$ defining the multiplication operator $M_{\lambda}$ will be called its {\bf symbol}.
\end{notation}

\subsection{Pointwise products}

Let $X$ and $Y$ be two Banach ideal spaces both defined on the same measure space $(\Omega,\Sigma,\mu)$.
The {\bf pointwise product $X \odot Y$} is defined by
\begin{equation*}
	X \odot Y \coloneqq \left\{ gh \colon g \in X \text{ and } h \in Y \right\},
\end{equation*}
and equipped with the functional
\begin{equation*}
	\norm{f}_{X \odot Y} \coloneqq \inf \bigl\{ \norm{g}_X \norm{h}_Y \colon f = gh \text{ with } g \in X \text{ and } h \in Y \bigr\}.
\end{equation*}
Note that the product $X \odot Y$ is a vector space due to the ideal property of $X$ and $Y$.

\begin{example}
	Let $1 \leqslant p,q \leqslant \infty$ with $1/r = 1/p + 1/q$. Then, it is straightforward to see, that $L_p \odot L_q \equiv L_r$.
	\demo
\end{example}

For more about the spaces of pointwise products we refer to \cite{CS14}, \cite{KLM14} and \cite{Sch10}.

\subsection{Symmetrizations} \label{SUBSEC: Symmetrization}

Let $X$ be a K{\" o}the space defined on $\sigma$-finite measure space $(\Omega,\Sigma,\mu)$.
By the {\bf symmetrization $X^{\bigstar}$} of $X$ we will understand the set
\begin{equation*}
	X^{\bigstar} \coloneqq \left\{ f \in L_0(\Omega) \colon f^{\star} \in X \right\}
\end{equation*}
together with the functional $\norm{f}_{X^{\bigstar}} \coloneqq \norm{f^{\star}}_X$.

In general, even the fact that the symmetrization $X^{\bigstar}$ is a vector space is not so obvious
(see \cite[Example~3.1]{CKMP04} for a counter-example). In fact, it was proved by Kami{\' n}ska and Raynaud \cite[Lemma~1.4]{KR09}
(and later supplemented by Kolwicz, Le{\' s}nik and Maligranda in \cite[Corollary~1]{KLM19}) that the symmetrization $X^{\bigstar}$ of
a K{\" o}the space $X$ is a vector spaces if, and only if, the dilation operator $D_2$ is bounded on the cone of non-negative and
decreasing functions from $X$. Even then, however, the functional $\norm{\bullet}_{X^{\bigstar}}$ may be only a quasi-norm\footnote{In quasi-Banach
spaces, the triangle inequality is allowed to hold approximately, that is, $\norm{f + g}_X \leqslant \Delta \left( \norm{f}_X + \norm{g}_X \right)$
for some constant $\Delta \geqslant 1$ (see, for example, \cite{Kal03} for more).}.

\begin{remark}
	Many classical rearrangement invariant spaces, like Lorentz spaces $\Lambda_{\varphi}$ and $L_{p,q}$, Marcinkiewicz spaces $M_{\varphi}$
	and Orlicz--Lorentz spaces $\Lambda_{M,w}$ are particular cases of this construction, namely,
	\begin{equation*}
		\Lambda_{\varphi} \equiv \left[ L_1(\varphi') \right]^{\bigstar}, \quad L_{p,q} \equiv \left[ L_q\left( t^{\frac{1}{p} - \frac{1}{q}} \right) \right]^{\bigstar},
		\quad M_{\varphi} \equiv \left[ L_{\infty}(\varphi) \right]^{\bigstar} \quad \text{and} \quad \Lambda_{M,w} \equiv L_{\Phi}^{\bigstar},
	\end{equation*}
	where $L_{\Phi}$ is a Musielak--Orlicz space generated by the Musielak--Orlicz function $\Phi(u,t) = M(u)w(t)$ and $M$ is a Young function.
	\demo
\end{remark}

This construction enjoys some interest nowadays (see \cite{Kol16}, \cite{KLM19}, \cite{KR09} and references given there).

\begin{scholium}
	For us, as we will see later, its usefulness stems primarily from the fact that it often allows us to reduce some rather difficult
	questions about the structure of rearrangement invariant spaces to simpler questions about the structure of weighted $L_p$-spaces;
	for example, it takes little skill to notice that
	\begin{equation*}
		M(L_{\infty}(w), F) = F\left( 1/w \right),
	\end{equation*}
	where $F$ is a rearrangment invariant space. At the same time, it is not entirely clear how to identify the space
	$M(M_{\varphi},F)$, until we notice that the following equalities should be true
	\begin{equation*}
		M(M_{\varphi},F)
			= M\left( [L_{\infty}(\varphi)]^{\bigstar}, F^{\bigstar} \right)
			\overset{\textdbend}{=} \left[ M(L_{\infty}(\varphi), F) \right]^{\bigstar}
			= \left[ F\left( 1/\varphi \right) \right]^{\bigstar}.
	\end{equation*}
	The crux, and the hardest part, is to make the second equality somehow rigorous (see Section~\ref{SECTION: examples} for details).
	\demo
\end{scholium}

\subsection{Rademachers} \label{SUBSEC: Rademachers}
Usually, the {\it Rademacher system $\left\{ r_n \right\}_{n=1}^{\infty}$} is understood as the Bernoulli sequence of independent identically
and symmetrically distributed random variables on the unit interval taking the values $\pm 1$. In more familiar form, the {\bf $n^{\text{th}}$ standard Rademacher $r_n$}
with $n \in \mathbb{N}$ is just a function
\begin{equation} \label{EQ: n-th standard Rademachers}
	r_n(t) \coloneqq \sign(\sin 2^n\pi t),
\end{equation}
where $0 < t < 1$ (see \cite[Definition~6.2.1, p.~132]{AK06} and \cite[p.~24]{LT77}; Astashkin's book \cite{Ast20} seems to be an
inexhaustible source of information about Rademacher's functions).

However, for some reasons, which will become clear later, the object we really want to construct here is a family of functions that shares
the most important properties of Rademacher system, but can be defined on {\it any} $\sigma$-finite measure space.
Since on a finite measure space $(\Omega,\Sigma,\mu)$ Rademachers are nothing else but a sequence of independent random variables
each of which takes values $\pm 1$ with probability $\mu(\Omega)/2$, so the construction of Rademachers on any finite and
non-atomic measure space should present no difficulties (cf. \cite[p.~224]{JL0506}). With a little patience while gluing the appropriate
fragments together we can do the same for every $\sigma$-finite measure space. Unfortunately, handling with such a construction turned out
to be quite difficult to us. Therefore, we propose the following, easier-to-work, one.

\begin{construction}[Abstract Rademachers] \label{DEF: Abstract Rademachers}
	Let $(\Omega,\Sigma,\mu)$ be a non-atomic $\sigma$-finite measure space and let $\{\Omega_{N}\}_{N=1}^{\infty}$
	be a decomposition of $\Omega$ into countable family of pairwise disjoint sets of measure not exceeding one. Further,
	for every $N \in \mathbb{N}$, let $T_{N} \colon \Omega_{N} \rightarrow (0,\mu(\Omega_N))$ be a measure-preserving transformation
	the existence of which is guaranteed by Ryff's theorem (see \cite[Theorem~7.5, p.~82]{BS88}). Put
	\begin{equation*}
		r^{\Omega_{N}}_n(\omega) \coloneqq (r_n \circ T_{N})(\omega) \quad \text{ for } \quad \omega \in \Omega_N,
	\end{equation*}
	where, as usual, the symbol $r_n$ stands for the $n^{\text{th}}$ standard Rademacher (like in (\ref{EQ: n-th standard Rademachers}) above).
	Then, for $n \in \mathbb{N}$, by the {\bf $n^{\text{th}}$ abstract Rademacher} on $\Omega$ we will understand the function
	$r^{\Omega}_n \colon \Omega \rightarrow \{\pm 1\}$ defined as
	\begin{equation*}
		r^{\Omega}_n(\omega) \coloneqq \sum_{N=1}^{\infty} r^{\Omega_{N}}_n(\omega) \quad \text{ for } \quad \omega \in \Omega.
	\end{equation*}
	\demo
\end{construction}

With such a definition at hand, it is not difficult to prove the following, essentially technical, lemma which seems to be well-known
in its basic version (that is, when $\Omega$ is just a unit interval or a semi-axis with the Lebesgue measure; see, for example, \cite[Lemma~6.3.2, p.~143]{AK06}
and \cite[Lemma~3.2]{Kub14}).

\begin{lemma} \label{LEMMA: Rademachers}
	{\it Let $(\Omega,\Sigma,\mu)$ be a non-atomic complete and $\sigma$-finite measure space. Then for any function, say $f$,
		from $L_1(\Omega,\Sigma,\mu)$ we have}
	\begin{equation*}
		\lim\limits_{n \rightarrow \infty} \int_{\Omega} f(\omega) r^{\Omega}_n(\omega) d\mu(\omega) = 0.
	\end{equation*}
\end{lemma}
\begin{proof}
	Take $f \in L_1(\Omega,\Sigma,\mu)$ and let $\{\Omega_{N}\}_{N=1}^{\infty}$ be like in Construction~\ref{DEF: Abstract Rademachers}.
	
	First, we claim that for any $N \in \mathbb{N}$ we have 
	\begin{equation} \label{rademacherydozera} 
		\lim\limits_{n \rightarrow \infty} \int_{\Omega} f(\omega) r^{\Omega_N}_n(\omega) d\mu(\omega) = 0.
	\end{equation}
	Indeed,
	\begin{align*}
		\lim\limits_{n \rightarrow \infty} \int_{\Omega} f(\omega) r^{\Omega_N}_n(\omega) d\mu(\omega) &
		=\lim\limits_{n \rightarrow \infty} \int_{\Omega_N} f(\omega) \left( r_n \circ T_{N} \right)(\omega) d\mu(\omega) \\
		& = \lim\limits_{n \rightarrow \infty} \int_{0}^{\mu\left( \Omega_N \right)} \left( f \circ T_N^{-1} \right)(t) r_n(t) dm(t)  = 0,
	\end{align*}
	where the last equality follows from, for example, \cite[Lemma~3.2]{Kub14}. This proves the claim.
	
	Next, fix $\varepsilon>0$. Since the space $L_1(\Omega,\Sigma,\mu)$ is order continuous (see \cite[Theorem~5.5, p.~27]{BS88}),
	so there is a finite sequence of natural numbers, say $\{N_j\}_{j=1}^{M}$, such that
	\begin{equation*}
		\int_{\Omega\setminus \bigcup\limits_{j=0}^M\Omega_{N_j}}\abs{ f(\omega)} d\mu(\omega) < \frac{\varepsilon}{2}.
	\end{equation*}
	Now, using (\ref{rademacherydozera}), for every $j = 1, 2, ..., M$ we can find $n_j \in \mathbb{N}$ such that
	\begin{equation*}
		\abs{\int_{\Omega} f(\omega)r^{\Omega_{N_j}}_n(\omega) d\mu(\omega)} < \frac{\varepsilon}{2M} \quad \text{ for } \quad n \geqslant n_j.
	\end{equation*}
	Finally, for all $n \geqslant \max\{n_1,n_2,... , n_M\}$, we have
	\begin{align*}
		\abs{\int_{\Omega} f(\omega) r^{\Omega}_n(\omega) d\mu(\omega)} 
		&=\abs{\sum\limits_{j=0}^M\int_{\Omega_{N_j}} f(\omega) r^{\Omega}_n(\omega) d\mu(\omega)
			+\int_{\Omega\setminus \bigcup\limits_{j=0}^M\Omega_{N_j}} f(\omega)r^{\Omega}_n(\omega) d\mu(\omega)}\\
		&\leqslant\sum\limits_{j=0}^M\abs{\int_{\Omega_{N_j}} f(\omega) r^{\Omega_{N_j}}_n(\omega) d\mu(\omega)}
		+ \int_{\Omega\setminus \bigcup\limits_{j=0}^M\Omega_{N_j}}\abs{ f(\omega)} d\mu(\omega)
		< \varepsilon.
	\end{align*}
	In view of $\varepsilon$'s arbitrariness this finishes the proof.
\end{proof}

\section{{\bf Main results}} \label{SECTION: ESS NORMS}

This section is a main part of our work. Since the title of this paragraph does not leave much room for guesswork,
so let us stop at the following remark, which we will use repeatedly without making any explicit reference to it.
In view of
\begin{equation*}
	M_{\lambda} = M_{e^{i\varphi} \abs{\lambda}} = M_{e^{i\varphi}}M_{\abs{\lambda}},
\end{equation*}
the multiplication operator $M_{\lambda}$ belongs to some operator ideal, say $\mathscr{I}$, if, and only if, the same can be said about the
multiplication operator $M_{\abs{\lambda}}$.

\subsection{Essential norms} For the sake of completeness, let us remind that an operator acting from one Banach space to another is
called {\bf compact} provided it carries bounded sets to relatively compact one.

What we want to do now is to calculate the essential norm of the multiplication operator $M_{\lambda}$ which acts between two K{\" o}the function spaces $X$ and $Y$.
Roughly speaking, we will show that $M_{\lambda} \colon X \rightarrow Y$ is as far from being compact as it can be.

\begin{theorem}[Essential norm of multiplication operators between K{\" o}the function spaces] \label{THM: essnorm multipliers function}
	{\it Let $X$ and $Y$ be two K{\" o}the function spaces with the Fatou property both defined on the same $\sigma$-finite
		measure space $(\Omega,\Sigma,\mu)$. Suppose that either the space $X$ is separable or the space $Y$ is reflexive.
		Then
		\begin{equation*}
			\norm{M_{\lambda} \colon X \rightarrow Y}_{ess} = \norm{\lambda}_{M(X,Y)} = \norm{M_{\lambda} \colon X \rightarrow Y}.
		\end{equation*}
		In particular, there are no non-trivial compact multiplication operators acting from $X$ into $Y$.}
\end{theorem}
\begin{proof}
	Of course, we can immediately assume that the space $M(X,Y)$ is non-trivial, because otherwise there is nothing to prove. What's more,
	the inequality
	\begin{equation*}
		\norm{M_{\lambda} \colon X \rightarrow Y}_{ess} \leqslant \norm{M_{\lambda}}_{X \rightarrow Y} = \norm{\lambda}_{M(X,Y)}
	\end{equation*}
	is true for obvious reasons, so there is noting else but to turn to the proof of a more interesting inequality.
	
	{\bf First case.}
	Let us start with the case when the space $X$ is separable. Fix $\varepsilon > 0$ and take $\lambda \in M(X,Y)$. It follows from the definition
	of the norm in the space $M(X,Y)$ that there exists $f \in \text{Ball}(X)$ with $\norm{\lambda f}_Y \geqslant \norm{\lambda}_{M(X,Y)} - \varepsilon$.
	Moreover, it is also clear that
	\begin{equation*}
		\norm{\lambda f r^{\Omega}_n}_Y = \norm{\lambda f}_Y \geqslant \norm{\lambda}_{M(X,Y)} - \varepsilon,
	\end{equation*}
	where $r^{\Omega}_n$ denote the $n^{\text{th}}$ abstract Rademacher on $\Omega$ (see Construction~\ref{DEF: Abstract Rademachers} for details).
	Now, for any compact operator, say $K$, acting from $X$ into $Y$, we have
	\begin{align*}
		\norm{M_{\lambda} - K}_{X \rightarrow Y} 
			& \geqslant \norm{(M_{\lambda} - K)(f r^{\Omega}_n)}_Y \\
			& \geqslant \abs{\norm{\lambda f r^{\Omega}_n}_Y - \norm{K(f r^{\Omega}_n)}_Y},
	\end{align*}
	We claim that the sequence $\left\{ f r^{\Omega}_n \right\}_{n=1}^{\infty}$ is weakly null. To see this, take $g \in X^{\times}$
	(due to our assumptions the space $X$ is order continuous, thus we have the equality $X^* = X^{\times}$; see, for example,
	\cite[Theorem~4.1, p.~20]{BS88}). Since $fg \in L_1(\Omega,\Sigma,\mu)$, so it is enough to apply Lemma~\ref{LEMMA: Rademachers} to conclude that 
	\begin{equation*}
		\lim\limits_{n \rightarrow \infty} \abs{ \langle f r^{\Omega}_n, g \rangle } = \lim\limits_{n \rightarrow \infty} \abs{\int_{\Omega} (f g) r^{\Omega}_n d\mu} = 0
	\end{equation*}
	and this proves our claim. Consequently, since every compact operator is {\it a fortiori} completely continuous,
	so $\lim_{n \rightarrow \infty} \norm{K(f r^{\Omega}_n)}_Y = 0$. But now, using the above inequalities, we get
	\begin{equation*}
		\norm{M_{\lambda} - K}_{X \rightarrow Y} \geqslant \norm{\lambda}_{M(X,Y)} - \varepsilon.
	\end{equation*}
	This concludes the proof in the first case.
	
	{\bf Second case.}
	Now, let us move on to the second case where we assume that the space $Y$ is reflexive. It follows from Schauder's theorem that an operator, say $T$, acting
	between Banach spaces is compact if, and only if, its adjoint $T^*$ is compact too (see, for example, \cite[A.4, p.~329]{AK06}).
	Hence, we can conclude\footnote{Note that in general there is no hope for an equality in
	(\ref{inequality: ess norm of T and adjoint T}) (see \cite[Example~1]{AJS80}; cf. \cite[p.~160]{AJS80}).} that
	\begin{equation} \label{inequality: ess norm of T and adjoint T}
		\norm{M_{\lambda} \colon X \rightarrow Y}_{ess} \geqslant \norm{M_{\lambda}^* \colon Y^* \rightarrow X^*}_{ess},
	\end{equation}
	where $M_{\lambda}^*$ is the adjoint of the multiplication operator $M_{\lambda}$.
	Next, we will need the following observation.
	
	\vspace{5pt} \noindent {\bf Claim~$\spadesuit$.}
	{\it Let $W$ and $Z$ be two K{\" o}the function space both defined on the same $\sigma$-finite measure space $(\Omega,\Sigma,\mu)$.
		Further, let $M_{\lambda} \colon W \rightarrow Z$ be the multiplication operator with the symbol $\lambda \in M(W,Z)$.
		Suppose the space $Z$ is separable. Then the following diagram commutes}
		\begin{equation*}
			\xymatrix{
				Z^* \ar[r]^{M_{\lambda}^*} & W^*\\
				Z^{\times} \ar[u]^{L} \ar[r]^{M_{\lambda}} & W^{\times} \ar[u]_{\text{id}}}
		\end{equation*}
	{\it where $L \colon Z^{\times} \rightarrow Z^*$ is the natural isometric isomorphism between the K{\" o}the dual $Z^{\times}$
		and the topological dual $Z^*$ (see below for details), and $M_{\lambda} \colon Z^{\times} \rightarrow W^{\times}$ is the K{\" o}the
		adjoint\footnote{Recall that for a bounded operator $T$ acting between two K{\" o}the spaces $X$ and $Y$ defined on possibly
		different measure spaces its {\bf K{\" o}the adjoint} (or {\bf K{\" o}the dual}) operator $T^{\times} \colon Y^{\times} \rightarrow X^{\times}$
		is defined by
		\begin{equation*}
			\langle x,T^{\times}(y^{*}) \rangle = \langle T(x),y^{*} \rangle \quad \text{ for all } \quad x \in X \quad \text{ and } \quad y^* \in Y^{\times}.
		\end{equation*}
		} of $M_{\lambda} \colon W \rightarrow Z$.}
	
	\vspace{5pt} \noindent
	{\bf Proof of Claim~$\spadesuit$.}
	It is clear that each function $g \in Z^{\times}$ induces a bounded linear functional
	\begin{equation*}
		L_g \colon f \rightsquigarrow L_g(f) \coloneqq \int_{\Omega} f g d\mu
	\end{equation*}
	on $Z$ and we can think of $L_g$ as a member of the space $Z^*$. Moreover, since the space $Z$ is order continuous, it turns out that the mapping
	$L \colon Z^{\times} \rightarrow Z^*$ given by
	\begin{equation*}
		L \colon g \rightsquigarrow \left[ f \rightsquigarrow L_g(f) \right] \quad \text{ for } \quad f \in Z
	\end{equation*}
	is an isometric isomorphism (see, for example, \cite[Theorem~4.1, p.~20]{BS88}).
	Therefore, for $f \in W$ and $g \in Z^{\times}$, we have
	\begin{align*}
		\langle f, M_{\lambda}^{*}(L(g)) \rangle
			& = \langle M_{\lambda}(f), L(g) \rangle \\
			& = \int_{\Omega} M_{\lambda}(f) g d\mu \\
			& = \int_{\Omega} f M_{\lambda}(g) d\mu \\
			& = \langle f, M_{\lambda}(g) \rangle.
	\end{align*}
	This ends the proof of Claim~$\spadesuit$.
	\hfill $\blacksquare$
	
	\vspace{5pt}
	Now, it follows from the preceding Claim~$\spadesuit$ that the adjoint operator $M_{\lambda}^* \colon Y^* \rightarrow X^*$ can be naturally identified
	with the multiplication operator $M_{\lambda} \colon Y^{\times} \rightarrow X^{\times}$. Thus, we have
	\begin{align*}
		\norm{M_{\lambda} \colon X \rightarrow Y}_{ess}
			& \geqslant \norm{M_{\lambda} \colon Y^{\times} \rightarrow X^{\times}}_{ess} \quad (\text{using Claim~$\spadesuit$ together with \eqref{inequality: ess norm of T and adjoint T}}) \\
			& = \norm{\lambda}_{M(Y^{\times},X^{\times})} \quad (\text{using the first part of this proof}) \\
			& = \norm{\lambda}_{M(X^{\times\times},Y^{\times\times})} \\
			& = \norm{\lambda}_{M(X,Y)} \quad (\text{in view of the Fatou property}).
	\end{align*}
	We are done.
\end{proof}

Since it is straightforward to see that $M(X,X) \equiv L_\infty$, so one immediate consequence of Theorem~\ref{THM: essnorm multipliers function} is the following

\begin{corollary} \label{COR: compact function X to X}
	{\it Let $X$ be a separable K{\" o}the function space with the Fatou property. Then
	\begin{equation*}
		\norm{M_{\lambda} \colon X \circlearrowleft}_{ess} = \norm{\lambda}_{L_{\infty}} = \norm{M_{\lambda} \colon X \circlearrowleft}.
	\end{equation*}
	In particular, there are no non-trivial compact multiplication operators acting from $X$ into itself.}
\end{corollary}

Let us turn our attention to the case of K{\" o}the spaces defined over purely atomic measure spaces.
Here the situation is completely different, because, any symbol which is eventually zero induces a compact multiplication operator.

\begin{theorem}[Essential norm of multiplication operators between K{\" o}the sequence spaces] \label{THM: Essential norm Kothe sequence}
	{\it Let $X$ and $Y$ be two K{\" o}the sequence spaces with the Fatou property. Suppose that either the space $X$ is reflexive
		or the space $Y$ is separable. Then
	\begin{equation*}
		\norm{M_{\lambda} \colon X \rightarrow Y}_{ess} = \lim\limits_{n \rightarrow \infty} \norm{ \sum_{i=n}^{\infty} \lambda_i e_i }_{M(X,Y)}.
	\end{equation*}
		In particular, the multiplication operator $M_{\lambda} \colon X \rightarrow Y$ is compact if, and only if, $\lambda \in M(X,Y)_o$.}
\end{theorem}
\begin{proof}
	In order to show the first of the two inequalities, take $\lambda = \{\lambda_n\}_{n=1}^\infty$ from $M(X,Y)$.
	Let $\text{Proj}_n \colon Y \rightarrow Y$ be the projection onto first $n$ coordinates. We have
	\begin{align*}
		\norm{M_{\lambda} \colon X \rightarrow Y}_{ess}
				& = \inf \Big\{ \norm{M_{\lambda} - K}_{X \rightarrow Y} \colon K \in {\mathscr K}(X,Y) \Big\} \\
				& \leqslant \norm{M_{\lambda} - \text{Proj}_n M_{\lambda}}_{X \rightarrow Y} \\
				& = \norm{ \sum_{i=n}^{\infty} \lambda_i e_i }_{M(X,Y)},
	\end{align*}
	where to get the inequality we have used the plain fact that finite rank operators acting between Banach spaces are always compact.
	In consequence, we conclude\footnote{It seems worth emphasizing that inequality \eqref{INEQ : <=} holds without any additional assumptions on $X$ and $Y$.}
	that
	\begin{equation} \label{INEQ : <=}
		\norm{M_{\lambda} \colon X \rightarrow Y}_{ess} \leqslant \lim_{n \rightarrow \infty} \norm{ \sum_{i=n}^{\infty} \lambda_i e_i }_{M(X,Y)}.
	\end{equation}
	
	Now, let us turn to the proof of the second inequality. We will consider both situations mentioned in the formulation of the theorem separately.
	
	{\bf First case}. Suppose that the space $X$ is reflexive. Fix $\varepsilon > 0$. Due to the definition of norm in the space $M(X,Y)$ we infer
	that for any $n \in \mathbb{N}$ we can find a sequence, say $x^{(n)} = \left\{ x^{(n)}_i \right\}_{i=1}^{\infty}$, in $\text{Ball}_X$ such that
	\begin{equation} \label{INEQ : A}
		\norm{ \sum_{i=n}^{\infty} \lambda_i x^{(n)}_i e_i }_Y \geqslant \norm{\sum_{i=n}^{\infty} \lambda_i e_i }_{M(X,Y)} - \varepsilon.
	\end{equation}
	Clearly, without losing generality, we can assume that $\supp\left( x^{(n)} \right) \subset \{n, n+1, ...\}$ for $n \in \mathbb{N}$.
	In view of \eqref{INEQ : A} and the ideal property of the space $Y$ this means that
	\begin{equation} \label{inequality: seq multip}
		\norm{\lambda x^{(n)}}_Y \geqslant \norm{ \sum_{i=n}^{\infty} \lambda_i e_i }_{M(X,Y)} - \varepsilon \quad \text{ for every } \quad n \in \mathbb{N}.
	\end{equation}
	Next, we claim that the sequence $\left\{ x^{(n)} \right\}_{n=1}^{\infty} \subset \text{Ball}(X)$ is weakly null.
	Indeed, since the space $X$ is reflexive, so the space $X^{\times}$ is order continuous and $X^* \equiv X^{\times}$
	(keeping in mind that the space $X$ has the Fatou property we refer to \cite[Corollaries 4.3 and 4.4, p.~23]{BS88}).
	Therefore, for any $x^* = \{ x^*_i \}_{i=1}^{\infty}$ from $X^{\times}$, we have
	\begin{align*}
		\abs{\langle x^{(n)}, x^* \rangle}
			& = \abs{\sum_{i=1}^{\infty} x^{(n)}_i x^*_i} \\
			& = \abs{\sum_{i=n}^{\infty} x^{(n)}_i x^*_i} \\
			& \leqslant \norm{ \sum_{i=1}^{\infty} x^{(n)}_i e_i }_X \norm{ \sum_{i=n}^{\infty} x^*_i e_i }_{X^{\times}} \to 0
	\end{align*}
	as $n$ goes to infinity.
	Now, for any compact operator $K$ acting between $X$ and $Y$, we have
	\begin{align*}
		\norm{M_{\lambda} - K}_{X \rightarrow Y}
			& \geqslant \norm{(M_{\lambda} - K) ( x^{(n)} ) }_Y \\
			& \geqslant \norm{\lambda x^{(n)}}_Y - \norm{K x^{(n)} }_Y,
	\end{align*}
	where the last inequality follows from the reverse triangle inequality. Thus, remembering about inequality (\ref{inequality: seq multip}),
	\begin{equation*}
		\norm{M_{\lambda} - K}_{X \rightarrow Y}
			\geqslant \norm{ \sum_{i=n}^{\infty} \lambda_i e_i }_{M(X,Y)} -  \norm{K x^{(n)} }_Y - \varepsilon.
	\end{equation*}
	Going further, since every compact operator is {\it a fortiori} completely continuous and, as we have shown above, the sequence
	$\left\{ x^{(n)} \right\}_{n=1}^{\infty} \subset \text{Ball}_X$ is weakly null, so $\lim_{n \rightarrow \infty} \norm{K x^{(n)} }_Y = 0$.
	In consequence, we have
	\begin{equation*}
		\norm{M_{\lambda} - K}_{X \rightarrow Y} \geqslant \lim\limits_{n \rightarrow \infty} \norm{ \sum_{i=n}^{\infty} \lambda_i e_i }_{M(X,Y)} - \varepsilon
	\end{equation*}
	and this, due to $\varepsilon$'s arbitrariness, ends the proof in the first case.
	
	{\bf Second case}. Suppose that the space $Y$ is separable. Again, fix $\varepsilon > 0$. Let $\left\{ x^{(n)} \right\}_{n=1}^{\infty} \subset \text{Ball}_X$
	be the sequence constructed above. Further, let $K \colon X \rightarrow Y$ be a compact operator. Since the set $\left\{ K x^{(n)} \right\}_{n=1}^{\infty}$
	is relatively compact in $Y$, so it has an accumulation point, say $a = \{a_n\}_{n=1}^{\infty}$.
	This means that there is a subsequence $\left\{ x^{(n(k))} \right\}_{k=1}^{\infty}$ of the sequence $\left\{ x^{(n)} \right\}_{n=1}^{\infty}$ with
	\begin{equation}\label{inequality: podciag zbiezny}
		\norm{K x^{(n(k))} - y_a}_Y < \varepsilon \quad \text{ for all } \quad k \in \mathbb{N} \quad \text{ big enough}.
	\end{equation}
	Therefore, using again the reverse triangle inequality,
	\begin{align*}
		\norm{M_{\lambda} - K}_{X \rightarrow Y}
			& \geqslant \norm{(M_{\lambda} - K)( x^{(n(k))} )}_Y \\
			& \geqslant \norm{M_{\lambda} x^{(n(k))} - a}_Y - \norm{K x^{(n(k))} - a}_Y.
	\end{align*}
	However, the above inequality in tandem with (\ref{inequality: podciag zbiezny}) implies that for sufficiently big $k \in \mathbb{N}$, we have
	\begin{equation} \label{inequality: nierownosc ess seq}
		\norm{M_{\lambda} - K}_{X \rightarrow Y} \geqslant \norm{M_{\lambda} x^{(n(k))}  - a}_Y - \varepsilon.
	\end{equation}
	On the other hand, recalling that $\supp\left( x^{(n)} \right) \subset \{n, n+1, ...\}$ for $n \in \mathbb{N}$, we get
	\begin{align*}
		\norm{M_{\lambda} x^{(n(k))} - a}_Y
			& = \norm{\lambda x^{(n(k))} - a}_Y \\
			& \geqslant \norm{ \sum_{i = n(k)}^{\infty} \left[ \lambda_i x^{(n(k))}_i - a_i \right] e_i }_Y \\
			& \geqslant \norm{ \sum_{i = n(k)}^{\infty} \lambda_i x^{(n)}_i e_i }_Y - \norm{ \sum_{i = n(k)}^{\infty} a_i e_i }_Y \\
			& \geqslant \norm{ \sum_{i = n(k)}^{\infty} \lambda_i e_i }_{M(X,Y)} - 2\varepsilon,
	\end{align*}
	where the last inequality follows from inequality (\ref{inequality: seq multip}) together with the fact that the space $Y$ is order continuous.
	Finally, by combining the just obtained inequality with (\ref{inequality: nierownosc ess seq}),
	\begin{equation*}
		\norm{M_{\lambda} - K}_{X \rightarrow Y} \geqslant \norm{ \sum_{i=n(k)}^{\infty} \lambda_i e_i }_{M(X,Y)} - 3\varepsilon.
	\end{equation*}
	This, due to the arbitrariness of $K$ and the ideal property of the space $M(X,Y)$, gives
	\begin{equation*}
		\norm{M_{\lambda} \colon X \rightarrow Y}_{ess}
			\geqslant \lim\limits_{n \rightarrow \infty} \norm{ \sum_{i=n}^{\infty} \lambda_i e_i }_{M(X,Y)} - 3\varepsilon.
	\end{equation*}
	This proves what we wanted.
\end{proof}

In the algebraic setting, the above Theorem~\ref{THM: Essential norm Kothe sequence} reduces to the following

\begin{corollary} \label{COR: compact sequence X to X}
	{\it Let $X$ be a separable K{\" o}the sequence space with the Fatou property. Then
		\begin{equation*}
			\norm{M_{\lambda} \colon X \circlearrowleft}_{ess} = \limsup\limits_{n \rightarrow \infty} \abs{\lambda_n}.
		\end{equation*}
		In particular, the multiplication operator $M_{\lambda} \colon X \circlearrowleft$ is compact if, and only if, $\lambda \in c_0$.}
\end{corollary}
\begin{proof}
	Since $M(X,X) \equiv \ell_\infty$, so using Theorem~\ref{THM: Essential norm Kothe sequence} we immediately get
	\begin{equation*}
		\norm{M_{\lambda} \colon X \circlearrowleft}_{ess}
			= \lim\limits_{n \rightarrow \infty} \norm{ \sum_{i=n}^{\infty} \lambda_i e_i }_{\ell_\infty}
			= \lim\limits_{n \rightarrow \infty} \left( \sup\limits_{i \geqslant n} \abs{\lambda_i} \right)
			= \limsup\limits_{n \rightarrow \infty} \abs{\lambda_n}.
	\end{equation*}
\end{proof}

Finally, we can combine an atomic case (Theorem~\ref{THM: Essential norm Kothe sequence}) together with a non-atomic case (Theorem~\ref{THM: essnorm multipliers function})
to describe the essential norm of multiplication operators acting between two K{\" o}the spaces.

\begin{theorem}[Essential norm of multiplication operators between K{\" o}the spaces] \label{THM: ess norm func + seq}
	{\it Let $X$ and $Y$ be two K{\" o}the spaces with the Fatou property both defined over the same $\sigma$-finite measure space $(\Omega,\Sigma,\mu)$.
		Further, let $\Omega_c$ and $\Omega_a$ denote the non-atomic and purely atomic part of $(\Omega,\Sigma,\mu)$, respectively.
		Suppose that one of the spaces $X$ or $Y$ is reflexive. Then}
		\begin{equation*}
			\norm{M_{\lambda} \colon X \rightarrow Y}_{ess} \approx
				\max \left\{ \norm{M_{\lambda} \colon X(\Omega_c) \rightarrow Y(\Omega_c)}_{ess}, \norm{M_{\lambda} \colon X(\Omega_a) \rightarrow Y(\Omega_a)}_{ess} \right\}
		\end{equation*}
	{\it with an equivalence involving universal constants only.}
\end{theorem}
\begin{proof}
	Without loss of generality, we can assume that the space $M(X,Y)$ is non-trivial.
	Moreover, to make things easier, let us assume that every atom in $\Omega_a$ is equivalent to a singleton.
	Then, since $\Omega = \Omega_c \sqcup \Omega_a$ and the measure space $(\Omega,\Sigma,\mu)$ is assumed to be $\sigma$-finite,
	so $\mu$ restricted to $\Omega_a$ is nothing else but the weighted sum of countably many Dirac measures.
	In other words, $\mu(A) = \mu_c(A) + \sum_{n=1}^{\infty} w_n \delta_{a_n}(A)$ for $A \in \Sigma$,
	where $w_n$'s are positive real numbers and $a_n$'s are atoms in $\Omega_a$.
	
	Take $\lambda \in M(X,Y)$.
	To show the first inequality let $\{K_n\}_{n=1}^{\infty}$ be a sequence of compact operators defined as
	\begin{equation*}
		K_n \colon X \rightarrow Y, \quad
		K_n \colon f \rightsquigarrow \left[ \omega \rightsquigarrow \lambda(\omega) f(\omega) \chi_{\{ a_1,a_2,..., a_{n-1} \}}(\omega) \right],
	\end{equation*}
	where $n \geqslant 2$. We have
	\begin{align*}
		\norm{M_{\lambda} \colon X \rightarrow Y}_{ess} & \leqslant \norm{M_{\lambda} - K_n}_{X \rightarrow Y} \\
			& \leqslant \norm{M_{\lambda} - K_n}_{X(\Omega_c) \rightarrow Y(\Omega_c)} + \norm{M_{\lambda} - K_n}_{X(\Omega_a) \rightarrow Y(\Omega_a)} \\
			& = \norm{M_{\lambda}}_{X(\Omega_c) \rightarrow Y(\Omega_c)} + \norm{\text{Proj}_n M_{\lambda}}_{X(\Omega_a) \rightarrow Y(\Omega_a)}.
	\end{align*}
	Now, it follows from Theorem~\ref{THM: essnorm multipliers function} and Theorem~\ref{THM: Essential norm Kothe sequence} that
	\begin{equation*}
		\norm{M_{\lambda} \colon X(\Omega_c) \rightarrow Y(\Omega_c)}_{ess} = \norm{M_{\lambda}}_{X(\Omega_c) \rightarrow Y(\Omega_c)},
	\end{equation*}
	and, respectively,
	\begin{equation*}
		\norm{M_{\lambda} \colon X(\Omega_a) \rightarrow Y(\Omega_a)}_{ess}
			= \norm{\text{Proj}_n M_{\lambda}}_{X(\Omega_a) \rightarrow Y(\Omega_a)}.
	\end{equation*}
	In consequence, we have
	\begin{equation*}
		\norm{M_{\lambda} \colon X \rightarrow Y}_{ess}
			\leqslant 2\max \left\{ \norm{M_{\lambda} \colon X(\Omega_c) \rightarrow Y(\Omega_c)}_{ess}, \norm{M_{\lambda} \colon X(\Omega_a) \rightarrow Y(\Omega_a)}_{ess} \right\}.
	\end{equation*}

	To prove the second inequality, let us first define two operators
	\begin{equation*}
		\text{Proj}_{\bullet}^X \colon X \rightarrow X(\Omega_{\bullet}), \quad
		\text{Proj}_{\bullet}^X \colon f \rightsquigarrow \left[ \omega \rightsquigarrow f(\omega)\chi_{\Omega_{\bullet}}(\omega) \right]
	\end{equation*}
	and
	\begin{equation*}
		\text{Proj}_{\bullet}^Y \colon Y \rightarrow Y(\Omega_{\bullet}), \quad
		\text{Proj}_{\bullet}^Y \colon f \rightsquigarrow \left[ \omega \rightsquigarrow f(\omega)\chi_{\Omega_{\bullet}}(\omega) \right],
	\end{equation*}
	where $\bullet \in \{a, c\}$.
	It is straightforward to see that $\text{Proj}_{\bullet}^X$ and $\text{Proj}_{\bullet}^Y$ are the norm one projections
	onto $X(\Omega_{\bullet})$ and, respectively, $Y(\Omega_{\bullet})$.
	Now, we have
	\begin{align*}
		\norm{M_{\lambda} \colon X(\Omega_{\bullet}) \rightarrow Y(\Omega_{\bullet})}_{ess}
			& = \inf \left\{ \norm{M_{\lambda} - K}_{X(\Omega_{\bullet}) \rightarrow Y(\Omega_{\bullet})} \colon K \in \mathscr{K}(X(\Omega_{\bullet}),Y(\Omega_{\bullet})) \right\} \\
			& = \inf \left\{ \norm{\text{Proj}_{\bullet}^Y(M_{\lambda} - K)\text{Proj}_{\bullet}^X}_{X \rightarrow Y} \colon K \in \mathscr{K}(X,Y) \right\} \\
			& \leqslant \inf \left\{ \norm{M_{\lambda} - K}_{X \rightarrow Y} \colon K \in \mathscr{K}(X,Y) \right\} \\
			& = \norm{M_{\lambda} \colon X \rightarrow Y}_{ess}.
	\end{align*}
	Thus,
	\begin{equation*}
		\max \left\{ \norm{M_{\lambda} \colon X(\Omega_c) \rightarrow Y(\Omega_c)}_{ess}, \norm{M_{\lambda} \colon X(\Omega_a) \rightarrow Y(\Omega_a)}_{ess} \right\}
			\leqslant \norm{M_{\lambda} \colon X \rightarrow Y}_{ess}.
	\end{equation*}
	This ends the proof.
\end{proof}

\subsection{Measures of non-compactness}
Inspired by de Pagter and Schep's approach taken in \cite{PS88}, in this section we want to look at Theorem~\ref{THM: Essential norm Kothe sequence}
from measures of non-compactness's perspective.

Let $D$ be a bounded subset of a Banach space $X$. Recall, that the {\bf Hausdorff measure of non-compactness $\gamma(D)$} of $D$ is defined to be
the infimum over all $\varepsilon$'s such that $D$ can be covered with a finite number of balls with radii $\varepsilon > 0$.
Or, what comes to the same thing,
\begin{equation*}
	\gamma(D) = \inf \left\{ \varepsilon > 0 \colon D \subset K + \varepsilon\text{Ball}_X \text{ with } K \text{ compact} \right\}.
\end{equation*}

The {\bf measure of non-compactness} of an operator $T \colon X \rightarrow Y$ is also denoted by $\gamma$ and is given by
\begin{equation*}
	\gamma(T \colon X \rightarrow Y) \coloneqq \gamma(T\text{Ball}_X).
\end{equation*}
The functional $\gamma$ induces a norm on the quotient $\mathscr{L}(X,Y)/\mathscr{K}(X,Y)$.
However, this norm need not be complete in general (see \cite[Theorem~2.5]{AT87}).
We refer to \cite{BG79}, \cite{Dil87}, \cite{LS71} and \cite{PS88} for more.

We say that $X$ has the {\bf $\delta$-approximation property} (briefly, the {\bf $\delta$-AP}) whenever for each compact subset, say $D$,
of $X$ and for each $\varepsilon > 0$, there is a finite rank operator $T \colon X \rightarrow Y$ with $\norm{T}_{X \rightarrow Y} \leqslant \delta$
so that $\norm{Tx - x}_{X \rightarrow Y} \leqslant \varepsilon$ for every $x \in D$.
Furthermore, $X$ has the {\bf metric approximation property} (briefly, the {\bf MAP}), if it has the {\bf $1$-AP}.
Similarly, we say that $X$ has the {\bf $\delta$-compact approximation property} (briefly, the {\bf $\delta$-CAP}),
if in the definition of {\bf $\delta$-AP} a finite rank operator is replaced with a compact operator.
Moreover, $X$ has the {\bf metric compact approximation property} (briefly, the {\bf MCAP}), if it has the {\bf $1$-CAP}.

\begin{remark}
	Every Banach space with a Schauder basis has the {\bf $\delta$-AP} for some $\delta \geqslant 1$ (that is to say, the {\bf bounded approximation property BAP}).
	Clearly, since a finite rank operators are compact, so the {\bf $\delta$-CAP} is formally weaker that the {\bf $\delta$-AP}.
	Moreover, most Banach spaces are know to have the {\bf $\delta$-CAP} for some $\delta \geqslant 1$ (that is to say, the {\bf bounded compact approximation property BCAP}).
	However, because of Enflo's example, there exists a Banach space without {\bf BCAP} (see \cite{Enf73}; see also \cite[Theorem~1.e.7]{LT77}).
	\demo
\end{remark}

The main result of this section is the following

\begin{theorem}\label{thm: ess norm ri}
	{\it Let $X$ be a K{\" o}the sequence spaces and let $Y$ be a rearrangement invariant sequence space. Then}
	\begin{equation*}
		\norm{M_{\lambda} \colon X \rightarrow Y}_{ess} = \lim\limits_{n \rightarrow \infty} \norm{ \sum_{i = n}^{\infty} \lambda_i e_i }_{M(X,Y)}.
	\end{equation*}
\end{theorem}
\begin{proof}
	It is straightforward to see that
	\begin{equation} \label{INEQ: HMNC < ESS}
		\gamma(M_{\lambda} \colon X \rightarrow Y) \leqslant \norm{M_{\lambda} \colon X \rightarrow Y}_{ess}.
	\end{equation}
	Indeed, for any compact operator $K \colon X \rightarrow Y$, it is crystal clear that
	\begin{equation*}
		\gamma(M_{\lambda}\text{Ball}_X) = \gamma\left( (M_{\lambda}+K)\text{Ball}_X \right) \leqslant \norm{(M_{\lambda}+K)\text{Ball}_X}_Y \leqslant \norm{M_{\lambda}+K}_{X \rightarrow Y}.
	\end{equation*}
	Moreover, it is a particular part of Lebow and Schechter's result \cite[Theorem~3.6]{LS71}
	(later supplemented by Astala and Tylli \cite[Theorems~2.2 and 2.3]{AT87}) that
	\begin{equation} \label{INEQ: ESS < HMNC}
		\norm{M_{\lambda} \colon X \rightarrow Y}_{ess} \leqslant \delta \gamma(M_{\lambda} \colon X \rightarrow Y)
	\end{equation}
	provided $Y$ has the {\bf $\delta$-CAP}.
	Since, due to our assumptions, the space $Y$ is a rearrangement invariant, so it follows from Rao's result (see \cite[Theorem~4.6]{Rao70} and \cite{Rao73})
	that $Y$ has the {\bf MAP}.
	Consequently, $Y$ has the {\bf MCAP}, that is, {\bf $1$-CAP}.
	Then (\ref{INEQ: HMNC < ESS}) together with (\ref{INEQ: ESS < HMNC}) implies that
	\begin{equation} \label{EQ: ESS = HMNC}
		\norm{M_{\lambda} \colon X \rightarrow Y}_{ess} = \gamma(M_{\lambda} \colon X \rightarrow Y).
	\end{equation}
	(It is probably worth mentioning that up this point the above proof goes for any operator $T \colon X \rightarrow Y$,
	where $X$ and $Y$ are Banach spaces such that $Y$ has the {\bf MCAP}.) Putting things together, we have
	\begin{align*}
		\norm{M_{\lambda} \colon X \rightarrow Y}_{ess}
			& = \gamma(M_{\lambda} \colon X \rightarrow Y) \\
			& = \lim_{n \rightarrow \infty} \sup \left\{ \norm{ \sum_{i = n}^{\infty} y_i e_i }_Y \colon y = \{ y_i \}_{i=1}^{\infty} \in M_{\lambda}\text{Ball}_X \right\} \\
			& = \lim_{n \rightarrow \infty} \sup \left\{ \norm{ \sum_{i = n}^{\infty} \lambda_i x_i e_i }_Y \colon x = \{ x_i \}_{i=1}^{\infty} \in \text{Ball}_X \right\} \\
			& = \lim_{n \rightarrow \infty} \norm{ \sum_{i = n}^{\infty} \lambda_i e_i }_{M(X,Y)},
	\end{align*}
	where the first equality follows from (\ref{EQ: ESS = HMNC}),
	while the second equality is a consequence of Bana{\' s} and Goebel's formula \cite[Theorem~6.1.1, p.~23]{BG79}
	along with the fact that the space $Y$ has the ideal property.
\end{proof}

\begin{remark}[About Theorem~\ref{thm: ess norm ri}]
	We belive that Theorem~\ref{thm: ess norm ri} and Theorem~\ref{THM: Essential norm Kothe sequence} are incomparable
	(unfortunately, we know no example of a separable K{\" o}the sequence space without the {\bf MCAP}; see Remark~\ref{REMARK: Szankowski} below).
	Either way, from a pragmatic point of view it is definitely easier to check that a given space is order continuous
	than that it has the {\bf MCAP}.
	\demo
\end{remark}

\begin{remark}[On Szankowski's counter-example] \label{REMARK: Szankowski}
	In his remarkable 1976 paper, Szanowski \cite{Sza76} showed that there exist a Banach lattice which can be isometrically
	embedded into the space $(\bigoplus_{n=1}^{\infty} L_q(0,1))_{\ell_p}$ with $1 \leqslant q < p < \infty$ and which fails
	the {\bf CAP} (see \cite[Proposition~1]{Sza76}; see also a note at the end of the proof of Lemma~2, where he states that his
	argument also works for the weaker version of {\bf AP}, that is, {\bf CAP}). As mentioned later, by taking $1 < q < p < \infty$,
	there exists a uniformly convex (hence, due to Milman and Pettis result, also reflexive) Banach lattice without {\bf CAP}
	(see \cite[Proposition~2]{Sza76}). Since uniformly convex spaces are, in particular, uniformly monotone,
	so the said Banach lattice is order continuous. Actually, this Banach lattice is a K{\" o}the function space on $(0,1)$
	(see the proof of Lemma~1.g.3 in \cite{LT79}).
	\demo
\end{remark}

\begin{remark}[On Djakov and Ramanujan's question] \label{REMARK: Djakov-Ramanujan}
	In \cite{DR00} (precisely, see \cite[p.~318]{DR00}) Djakov and Ramanujan ask the following question
	\begin{quote}
		{\it \enquote{Is it true that every compact multiplier between the spaces $\ell_M$ and $\ell_N$ is a limit of finitely-supported multipliers?}}
	\end{quote}
	Thanks to Proposition~\ref{thm: ess norm ri} we can answer this question in the affirmative.
	To see this, it is enough to note, that the ideal $M(X,Y)_o$ is contained in the closure of finitely-supported
	functions living inside $M(X,Y)$ (see, for example, \cite[Theorem~3.11, p.~18]{BS88}).
	\demo
\end{remark}

Actually, the above remark can be lifted effortlessly to the following setting. 

\begin{corollary}
	{\it Let $X$ and $Y$ be two rearrangement invariant sequence spaces.
	Then every compact multiplier between $X$ and $Y$ is a limit of finitely-supported multipliers.}
\end{corollary}

\subsection{Around Pitt's theorem}

We know from Harry R. Pitt \cite{Pit32} (see also \cite[Proposition~2.c.3, p.~76]{LT77}, \cite[Theorem~2.1.4, p.~32]{AK06} and
\cite[Theorem~A2]{Ros69} for even more general result) that ${\mathscr L}(\ell_p,\ell_q) = {\mathscr K}(\ell_p,\ell_q)$
if, and only if, $1 \leqslant q < p \leqslant \infty$
(with the convention that whenever $p = \infty$, then we are working with $c_0$ instead of $\ell_{\infty}$).

\begin{remark}
	This classical and well-known result has had many generalization over the years. For example,
	\begin{itemize}
		\item[$\bullet$] Let $M$ and $N$ be two Orlicz functions both satisfying the $\Delta_2$-condition\footnote{A Young function $M$ is said to satisfy
			the {\bf $\Delta_2$-condition at zero} (or the {\bf $\Delta_2$-condition for \enquote{small} arguments}) if $\limsup_{t \rightarrow 0^{+}} M(2t)/M(t) < \infty$
			(see \cite[pp.~138--139]{LT77} and \cite[pp.~16--22]{Mal83} for more).} at zero.
		Then ${\mathscr L}(\ell_M,\ell_N) = {\mathscr K}(\ell_M,\ell_N)$ if, and only if, $\alpha_M > \beta_N$,
		where $\alpha_M$ and $\beta_N$ are the so-called Matuszewska--Orlicz indices (see \cite[p.~149]{LT77} and \cite[Chapter~11]{Ma89}).
		\item[$\bullet$] Let $X$ and $Y$ be two Banach spaces. Suppose that $X$ has an unconditional finite-dimensional expansion of the identity.
		Then, as was proved by Kalton, ${\mathscr L}(X,Y) = {\mathscr K}(X,Y)$ if, and only if, the space ${\mathscr L}(X,Y)$ contains no isomorphic copy of $\ell_{\infty}$
		(see \cite[Theorem~6]{Kal74} for details).
	\end{itemize}
	The interested reader is referred to \cite{DLMR00} and \cite{LM02} for more.
	\demo
\end{remark}

In the context of pointwise multipliers, due to Theorem~\ref{THM: Essential norm Kothe sequence}, Pitt's compactness theorem
reduces to the simple fact that the space $M(\ell_p,\ell_q) = \ell_r$ with $1/r = 1/q - 1/p$ is order continuous, while the space
$M(\ell_q,\ell_p) = \ell_{\infty}$ is not.
In general, we have the following result (which should probably be compared with \cite[Proposition~3.1]{Cro69}).
	
	\begin{proposition}[Pitt's compactness theorem for multipliers] \label{PROP: Pitt's multiplication}
		{\it Let $E$ and $F$ be two K{\" o}the sequence spaces with the Fatou property. Then $M(E,F) = \mathscr{K} \cap M(E,F)$
		if, and only if, the space $M(X,Y)$ contains no isomorphic copy of $\ell_{\infty}$.}
	\end{proposition}
	\begin{proof}
		Suppose that the space $M(X,Y)$ contains no isomorphic copy of $\ell_{\infty}$. Then, due to Lozanovski{\u \i}'s result \cite{Loz69},
		the space $M(X,Y)$ is order continuous.
		Now, this is a direct consequence of the first part of the proof of Theorem~\ref{THM: Essential norm Kothe sequence}
		(note that we need no additional assumptions on $E$ and $F$).
		
		Next, suppose the space $M(X,Y)$ contains an isomorphic copy of $\ell_{\infty}$. This means that $M(E,F)_o \neq M(E,F)$.
		Take $\xi = \{ \xi_i \}_{i=1}^{\infty}$ from $M(X,Y) \setminus M(X,Y)_o$. Then, due to \cite[Lemma~5]{HN04},
		there exists $\delta > 0$ and a strictly increasing sequence, say $\{n(k)\}_{k=1}^{\infty}$, of natural numbers such that
		\begin{equation} \label{eq: 1}
			\norm{ \sum_{i = n(k)}^{n(k+1)} \xi_i e_i }_{M(X,Y)} > \delta \quad \text{ for } \quad k \in \mathbb{N}.
		\end{equation}
		Moreover, since the space $M(X,Y)$ is furnished with the operator norm, so there is a sequence, say $\left\{ x^{(k)} \right\}_{k=1}^{\infty}$,
		in $\text{Ball}_X$ with $\text{supp}( x^{(k)} ) \subset \left\{ n(k), n(k) +1, ..., n(k+1) \right\}$ and such that
		\begin{equation} \label{eq: 2}
			\norm{ \sum_{i = n(k)}^{n(k+1)} \xi_i x^{(k)}_i e_i }_Y
			\geqslant \norm{ \sum_{i = n(k)}^{n(k+1)} \xi_i e_i }_{M(X,Y)} - \frac{\delta}{2} \quad \text{ for } \quad k \in \mathbb{N}.
		\end{equation}
		Now, combining (\ref{eq: 1}) together with (\ref{eq: 2}), we get
		\begin{equation} \label{EQ: not relatively compact}
			\norm{ \sum_{i=1}^{\infty} \xi_i x^{(k)}_i e_i }_Y
				= \norm{ \sum_{i = n(k)}^{n(k+1)} \xi_i x^{(k)}_i e_i }_Y > \frac{\delta}{2}
					\quad \text{ for } \quad k \in \mathbb{N}.
		\end{equation}
		However, since the image of a bounded sequence $\left\{ x^{(k)} \right\}_{k=1}^{\infty}$ through the multiplication operator $M_{\xi}$
		is just $\left\{ \xi x^{(k)} \right\}_{k=1}^{\infty}$, so \eqref{EQ: not relatively compact} means that the set $\left\{ \xi x^{(k)} \right\}_{k=1}^{\infty}$
		is not relatively compact in $Y$.
		In consequence, the operator $M_{\xi} \colon  X \rightarrow Y$ is not compact. This completes the proof.
	\end{proof}
	
	\begin{remark} \label{REMARK: general Pitt}
		Let $1 \leqslant p \neq q \leqslant \infty$.
		It is straightforward to see that the following eight conditions are equivalent:
		\begin{itemize}
			\item[$\bullet$] $p < q$,
			\item[$\bullet$] $\ell_p \hookrightarrow \ell_q$,
			\item[$\bullet$] $M(\ell_p,\ell_q) = \ell_{\infty}$,
			\item[$\bullet$] $M(\ell_q,\ell_p) = \ell_{r}$ with $1/r = 1/p - 1/q$,
			\item[$\bullet$] the space $M(\ell_p,\ell_q)$ contains an isomorphic copy of $\ell_{\infty}$,
			\item[$\bullet$] the space $M(\ell_q,\ell_p)$ is separable,
			\item[$\bullet$] there exists a non-compact multiplier between $\ell_p$ and $\ell_q$,
			\item[$\bullet$] every multiplier between $\ell_q$ and $\ell_p$ is compact.
		\end{itemize}
		
		In general, however, the situation is more delicate.
		For example, there is a pair of rearrangement invariant spaces, say $E$ and $F$, such that $E \subsetneqq F$,
		but at the same time both spaces $M(E,F)$ and $M(F,E)$ contain an isomorphic copy of $\ell_{\infty}$ (a rather non-obvious example illustrating
		this situation is given in Example~\ref{PROPOSITION: Appendix A - counterexample factorization}).
		This imposes important limitations on the general variant of Pitt's compactness theorem, namely,
		we need stronger types of separation between spaces than just set-theoretic inclusion (like, for example,
		strict inequality between lower and upper $p$-estimates or lower and upper Boyd indices;
		see \cite{DLMR00} and \cite[pp.~138--139]{LT77} where these ideas are further developed).
		\demo
	\end{remark}

\subsection{Weak essential norms}
Recall that an operator acting between two Banach spaces is {\bf weakly compact} if it takes bounded subset to relatively weakly compact one.
Equivalently, thanks to the Eberlein--{\v S}mulian theorem, an operator is weakly compact if, and only if, every bounded sequence admits a subsequence
whose image is weakly convergent.

As we saw earlier (see Theorem~\ref{THM: essnorm multipliers function}), there are basically no compact multiplication operators between K{\" o}the
function spaces. It is natural, then, to weaken this notion and ask, for example, about weakly compact multipliers.

But, before we proceed to the formulation of the next result, let us remind of a concept that, in a sense, will play one of the main role here. 
Precisely, a K{\" o}the space (in general, nothing prevents considering Banach lattices) has the {\bf positive Schur property} if every weakly null
sequence\footnote{As it turns out, it is enough to consider disjoint sequences only (see~\cite[Corollary~2.3.5]{MN91}).} with positive terms is norm
convergent. This property was investigated by Wnuk (see \cite{Wnu93} and \cite{Wnu89}). Due to Rosenthal's $\ell_1$-theorem, K{\" o}the spaces having
the positive Schur property are exactly those spaces which are {\bf 1-disjointly homogeneous} (briefly, {\bf 1-DH}), that is, in which every normalized
disjoint sequence has a subsequence equivalent to unit vector basis of $\ell_1$ (see~\cite[Theorem~7]{Wnu93}; cf.~\cite[Proposition~4.9]{FHSTT14}).
This class of spaces was introduced by Flores, Tradacete and Troitsky in \cite{FTT09} and attracted much attention in recent times (see, e.g., \cite{FHSTT14} and
\cite{LTM22} and references therein).

\begin{remark}
	The class of 1-disjointly homogeneous spaces significantly exceeds the family of $L_1$-spaces, because, for example,
	\begin{itemize}
		\item[$\bullet$] any Orlicz space $L_M$ generated by Young function $M$ such that {\bf (a)} $M$ satisfies the $\Delta_2$-condition for \enquote{large} arguments;
		{\bf (b)} the Young conjugate $M^*$ of $M$ satisfies the $\Delta_0$-condition; and {\bf (c)} $M$ is equivalent to the linear function for \enquote{small} arguments
		is 1-DH (clearly, if we narrow our considerations only to Orlicz spaces defined on $(0,1)$, then the last condition may be ignored; see \cite{LTM22} for more details);
		\item[$\bullet$] any Lorentz space $\Lambda_{\varphi}(0,1)$ is 1-DH provided the following three conditions are satisfied, namely,
		{\bf (a)} $\lim_{t \rightarrow 0^{+}} \varphi(t) = 0$; {\bf (b)} $\varphi(1) > 0$; and {\bf (c)} $\int_0^1 \varphi(t)dt = 1$
		(in particular, Lorentz spaces $L_{p,1}(0,1)$ are 1-DH for all $1 < p < \infty$);
		\item[$\bullet$] for each $1 < p < \infty$ the space $L_{p,1}(0,\infty) \cap L_1(0,\infty)$ is 1-DH (see~\cite[Theorem~8]{LTM22}).
	\end{itemize}
	Note, however, that spaces with the positive Schur property, or, which is one thing, 1-DH spaces, are in general not stable under isomorphism.
	\demo
\end{remark}

We are now ready to present the main result here.

\begin{theorem}[Weak essential norm of multiplication operators acting between K{\" o}the function spaces] \label{THM: weak ess norm}
	{\it Let $X$ and $Y$ be two K{\" o}the function spaces with the Fatou property both defined over the same $\sigma$-finite measure space
		$(\Omega,\Sigma,\mu)$. Suppose that the space $Y$ is} $1$-DH {\it and that every simple function\footnote{By the {\bf simple function}
		we understand here any function defined on $\Omega$ which is a finite linear combination of characteristic functions of sets of
		finite measure.} in the space $M(X,Y)$ has an absolutely continuous norm. Then}
		\begin{equation*}
			\norm{M_{\lambda} \colon X \rightarrow Y}_{w} = \lim_{n \rightarrow \infty} \norm{\lambda\chi_{\mathcal{U}_n(\lambda)}}_{M(X,Y)},
		\end{equation*}
		{\it where\footnote{It is clear that if the measure space $(\Omega,\Sigma,\mu)$ is finite, then it is enough to put
			$\mathcal{U}_n(\lambda) \coloneqq \left\{ \omega \in \supp(\lambda) \colon \abs{\lambda(\omega)} > n \right\}$ for $n \in \mathbb{N}$.}}
		\begin{equation*}
			\mathcal{U}_n(\lambda) \coloneqq \bigcup\limits_{N\geqslant n}\Omega_N
				\cup \left\{ \omega \in \supp(\lambda) \colon \abs{\lambda(\omega)} > n \right\} \quad \textit{ for } \quad n \in \mathbb{N}
		\end{equation*}
		{\it and $\{\Omega_N\}_{N=1}^{\infty}$ is a decomposition of $\Omega$ into a pairwise disjoint family of sets of finite measure.
		In particular, the multiplication operator $M_{\lambda} \colon X \rightarrow Y$ is weakly compact if, and only if, $\lambda \in M(X,Y)_o$.}
\end{theorem}
\begin{proof}
	Take $\lambda \in M(X,Y)$ and note that $\{\mathcal{U}_n(\lambda)\}_{n=1}^{\infty}$ is a decreasing family of measurable sets with
	$\mu\left( \bigcap_{n=1}^{\infty} \mathcal{U}_n(\lambda) \right) = 0$.
	
	To demonstrate the first inequality observe that all sets
	\begin{equation*}
		\mathcal{U}_n(\lambda)^{\complement} \coloneqq \supp(\lambda) \setminus \mathcal{U}_n(\lambda)
	\end{equation*}
	have a finite measure and that the symbol $\lambda$ is a bounded function when restricted to the set $\mathcal{U}_n(\lambda)^{\complement}$ for all $n \in \mathbb{N}$.
	Therefore, thanks to our assumptions imposed on the space $M(X,Y)$, we immediately see that $\lambda\chi_{\mathcal{U}_n(\lambda)^{\complement}} \in M(X,Y)_o$.
	To make our life easier, set $\widetilde{\lambda}_n \coloneqq \lambda\chi_{\mathcal{U}_n(\lambda)^{\complement}}$. Next, for each $n \in \mathbb{N}$, let
	$\left\{ A^n_k \right\}_{k=1}^{\infty} \subset \Sigma$ be a decreasing family of measurable subsets with $\mu\left( \bigcap_{k=1}^{\infty} A^n_k \right) = 0$.
	Since $\widetilde{\lambda}_n \in M(X,Y)_o$, so it follows that for all $\varepsilon > 0$ there is $K \in \mathbb{N}$ such that for all $k > K$ we have
	\begin{equation*}
		\sup\limits_{g \in M_{\widetilde{\lambda}_n}\text{Ball}_X} \norm{g\chi_{A^n_k}}_Y
			= \sup\limits_{h \in \text{Ball}_X} \norm{\widetilde{\lambda}_n h\chi_{A^n_k}}_Y
			= \norm{\widetilde{\lambda}_n \chi_{A^n_k}}_{M(X,Y)} < \varepsilon.
	\end{equation*}
	However, this exactly means that the set $M_{\widetilde{\lambda}_n} \text{Ball}_X$ is $Y$-equicontinuous (or, using a different nomenclature, $L$-weakly compact;
	see \cite[Definition~3.6.1, p.~212]{MN91}), so applying \cite[Proposition~3.6.5, p.~212]{MN91}, also relatively weakly compact. Consequently, all operators
	$M_{\widetilde{\lambda}_n}$ for $n \in \mathbb{N}$ are weakly compact, and we have the inequalities
	\begin{align*}
		\norm{M_{\lambda} \colon X \rightarrow Y}_{w}
			& \leqslant \inf\{\norm{M_{\lambda} - W}_{X \rightarrow Y} \colon W \in \mathscr{W}(X,Y)\} \\
			& \leqslant \norm{M_{\lambda} - M_{\widetilde{\lambda}_n}}_{X \rightarrow Y} \\
			& = \norm{\lambda - \widetilde{\lambda}_n}_{M(X,Y)} \\
			& = \norm{\lambda\chi_{\mathcal{U}_n}}_{M(X,Y)}.
	\end{align*}
	This clearly implies that
	\begin{equation*}
		\norm{M_{\lambda} \colon X \rightarrow Y}_{w} \leqslant \lim\limits_{n \rightarrow \infty} \norm{\lambda\chi_{\mathcal{U}_n}}_{M(X,Y)}
	\end{equation*}
	and we are done.
	
	Let us pass to the proof of the second inequality. Fix $\varepsilon > 0$ and take $\lambda \in M(X,Y)$. Still, let $\left\{ \mathcal{U}_n(\lambda) \right\}_{n=1}^{\infty}$
	be as in the formulation of the theorem. For any $n \in \mathbb{N}$ we can find $f_n \in \text{Ball}_X$ with $\supp(f_n) \subset \mathcal{U}_n(\lambda)$
	such that
	\begin{equation} \label{inequality: wydobywanie normy weakly compact}
		\norm{\lambda f_n}_Y = \norm{\lambda \chi_{\mathcal{U}_n(\lambda)} f_n}_Y \geqslant \norm{\lambda \chi_{\mathcal{U}_n(\lambda)}}_{M(X,Y)} - \varepsilon.
	\end{equation}
	Take any weakly compact operator, say $W$, acting between $X$ and $Y$. Since the space $Y$ is $1$-disjointly homogeneous, so due to
	\cite[Theorem~3.6]{LTM22} the space $Y$ satisfies\footnote{In fact, in the class of K{\" o}the function spaces, both conditions are equivalent.}
	the so-called Dunford--Pettis criterion, that is, every bounded and relatively weakly compact subset of $Y$ is $Y$-equicontinuous. However, this means
	that the set $\{W(f_n)\}_{n=1}^{\infty}$ is actually $Y$-equicontinuous and
	\begin{equation} \label{inequality: equi-continuity}
		\norm{W(f_n)\chi_{\mathcal{U}_n(\lambda)}}_Y < \varepsilon
	\end{equation}
	for all sufficiently large $n$. Therefore,
	\begin{align*}
		\norm{M_{\lambda} - W}_{X \rightarrow Y}
			& \geqslant \norm{(M_{\lambda} - W)(f_n)}_Y \\
			& = \norm{(\lambda f_n - W(f_n))\chi_{\mathcal{U}_n(\lambda)}}_Y \\
			& \geqslant \norm{\lambda f_n}_Y - \norm{W(f_n)\chi_{\mathcal{U}_n(\lambda)}}_Y,
	\end{align*}
	and, using (\ref{inequality: wydobywanie normy weakly compact}) along with (\ref{inequality: equi-continuity}), we get that for all
	sufficiently big $n$ we have
	\begin{equation*}
		\norm{M_{\lambda} - W}_{X \rightarrow Y} \geqslant \norm{\lambda \chi_{\mathcal{U}_n(\lambda)}}_{M(X,Y)} - 2 \varepsilon.
	\end{equation*}
	This is enough to conclude the proof.
\end{proof}

Next two examples explain the role of the assumption imposed on the space $Y$ in the above result.
However, before we dive deeper into details, we need to define the objects of our considerations.

Let $\varphi$ be an increasing concave function on $(0,1)$ with $\varphi(0^+) = 0$. Recall that the {\bf Lorentz space
$\Lambda_{{\varphi},q}(0,1)$} with $1 \leqslant q < \infty$ and the {\bf Marcinkiewicz space $M_{\varphi}(0,1)$} consists
of all equivalence classes, modulo equality almost everywhere, of real-valued Lebesgue measurable functions defined on
the unit interval $(0,1)$ with
\begin{equation} \label{EQ: Lorentz function}
	\norm{f}_{\Lambda_{{\varphi},q}} \coloneqq \left( \int_0^1 \left( \varphi(t)f^{\star}(t) \right)^q \frac{dt}{t} \right)^{1/q} < \infty
\end{equation}
and, respectively,
\begin{equation} \label{EQ: Marcinkiewicz function}
	\norm{f}_{M_{\varphi}} \coloneqq \sup\limits_{0 < t < 1} \frac{\varphi(t)}{t} \int_0^t f^{\star}(s)ds  < \infty.
\end{equation}
In particular, if the Lorentz space  $\Lambda_{{\varphi},q}$ and the Marcinkiewicz space $M_{\varphi}$ is generated by
the power function $\varphi(t) = t^{1/p}$, where $1 < p < \infty$, then it is denoted by $L_{p,q}$ and, respectively, $L_{p,\infty}$.
Of course, $L_{p,p} \equiv L_p$. Moreover, the spaces $L_{p,\infty}$ are sometimes also called the {\bf weak $L_p$-spaces}.

\begin{example}[Multipliers between weak $L_p$-spaces] \label{EXP: not weakly compact between Marcinkiewicz}
	Let $1 \leqslant p < q < \infty$. Recall that
	\begin{equation*}
		M\left( L_{p,\infty}(0,1), L_{q,\infty}(0,1) \right) = L_{r,\infty}(0,1),
	\end{equation*}
	where $1/r = 1/p - 1/q$ (see, e.g., \cite[Theorem~4]{KLM19} and references given there). Note also that the space $L_{q,\infty}$ is not\footnote{This space
	is not even {\it disjointly homogeneous} (in particular, it is not $p$-DH for any $1 \leqslant p \leqslant \infty$), because inside $L_{q,\infty}$ we can find a disjoint
	sequence with a subsequence equivalent to unit vector basis of $c_0$, as well as another disjoint sequence spanning $\ell_q$ (see, e.g., \cite[Example, p.~659]{FTT09}
	for more details). On the other hand, the space $(L_{q,\infty})_o$ is $\infty$-DH.}	$1$-DH.
	
	Take the function $\lambda_0(t) = t^{-1/r}$ from $M(L_{p,\infty}(0,1), L_{q,\infty}(0,1))$. Since for all $\varepsilon > 0$ we have that
	\begin{equation*}
		\norm{\lambda_0\chi_{(0,\varepsilon)}}_{L_{r,\infty}} = 1,
	\end{equation*}
	so $\lambda_0$ is not order continuous.
	
	Next, we will show that the multiplication operator $M_{\lambda_0} \colon L_{p,\infty}(0,1) \rightarrow L_{q,\infty}(0,1)$ is not weakly compact.
	For this, take $g(t) = t^{-1/p} \in L_{p,\infty}(0,1)$. Observe that to achieve the intended goal it is enough to show that
	\begin{equation} \label{functional yes}
		\lim\limits_{n \rightarrow \infty} \langle \lambda_0 g r_n, \varphi \rangle = 0 \quad \text{ for every } \quad \varphi \in (L_{q,\infty}(0,1))^{\times}
	\end{equation}
	and, moreover, to construct a linear functional, say $\psi$, in $(L_{q,\infty}(0,1))^*$ such that
	\begin{equation} \label{functional no}
		\langle \lambda_0 g r_n, \psi \rangle = \frac{q}{q-1} \quad \text{ for all } \quad n \in \mathbb{N},
	\end{equation}
	where $\{r_n\}_{n=1}^{\infty}$ is the sequence of Rademacher functions.
	In fact, if this can be done, then the image of the sequence $\{g r_n\}_{n=1}^{\infty} \subset L_{p,\infty}$ through the operator $M_{\lambda_0}$, that is,
	the sequence $\{\lambda_0 g r_n\}_{n=1}^{\infty} \subset L_{q,\infty}$, does not possess a weakly convergent subsequence (if it did, then because the intersection
	of kernels of all functionals from $(L_{q,\infty}(0,1))^{\times}$ is zero, this hypothetical subsequence would have to converge weakly to zero, which in view of
	(\ref{functional no}) is just impossible). But this would exactly mean that the operator $M_{\lambda_0} \colon L_{p,\infty}(0,1) \rightarrow L_{q,\infty}(0,1)$
	is not weakly compact.
	
	First, take $h \in (L_{q,\infty}(0,1))^{\times}$ and put
	\begin{equation*}
		\langle f, \varphi \rangle \coloneqq \int_0^1 f(t)h(t)dt \quad \text{ for } \quad f \in L_{q,\infty}(0,1).
	\end{equation*}
	Because of Lemma~\ref{LEMMA: Rademachers} (from Section~\ref{SUBSEC: Rademachers}) and the fact that $\lambda_0(t)g(t) = t^{-1/q} \in L_{q,\infty}(0,1)$ we infer that
	\begin{equation*}
		\langle \lambda_0 g r_n, \varphi \rangle = \int_0^1 \underbrace{\lambda_0(t)g(t)}_{\in L_{q,\infty}} r_n(t) \underbrace{h(t)}_{\in (L_{q,\infty})^{\times}}dt
			\rightarrow 0 \quad (\text{as $n$ goes to infinity}).
	\end{equation*}
	This proves our first claim (\ref{functional yes}).

	Second, let $\mathscr{U}$ be a free ultrafilter on $\mathbb{N}$. For $f \in L_{q,\infty}(0,1)$ put
	\begin{equation*}
		\langle f, \psi \rangle \coloneqq \lim\limits_{\mathscr{U}} \left[ 2^{n\left( 1-\frac{1}{q} \right)} \int_0^{2^{-n}} f(t)dt \right].
	\end{equation*}
	It is not hard to see that $\psi$ is bounded linear functional on $L_{q,\infty}(0,1)$. Indeed, it follows from Hardy--Littlewood's inequality
	(see, e.g., \cite[Theorem~2.2, p.~44]{BS88}) that
	\begin{align*}
		\abs{2^{n\left( 1-\frac{1}{q} \right)} \int_0^{2^{-n}} f(t)dt}
			& \leqslant 2^{n\left( 1-\frac{1}{q} \right)} \int_0^{2^{-n}} f^{\star}(t)dt \\
			& \leqslant \norm{f}_{L_{q,\infty}} 2^{n\left( 1-\frac{1}{q} \right)} \int_0^{2^{-n}} \frac{dt}{t^{\frac{1}{q}}} \\
			& \leqslant \frac{\norm{f}_{L_{q,\infty}}}{1 - \frac{1}{q}},
	\end{align*}
	where the second inequality is due to the fact that $f^{\star}(t) \leqslant \norm{f}_{L_{q,\infty}} t^{-1/q}$ for all $f \in L_{q,\infty}(0,1)$.
	Now, it remains only to note that for all $n \in \mathbb{N}$ there exists a big enough $N \in \mathbb{N}$ with
	\begin{align*}
		2^{N\left( 1-\frac{1}{q} \right)} \int_0^{2^{-N}} \lambda_0(t)g(t)r_n(t)dt
			& = 2^{N\left( 1-\frac{1}{q} \right)} \int_0^{2^{-N}} t^{-\frac{1}{q}}r_n(t)dt \\
			& = 2^{N\left( 1-\frac{1}{q} \right)} \int_0^{2^{-N}} t^{-\frac{1}{q}}dt \\
			& = \frac{q}{q-1}.
	\end{align*}
	Hence, as a consequence of the definition of functional $\psi$, we have
	\begin{equation*}
		\langle \lambda_0 g r_n, \psi \rangle = \frac{q}{q-1} \quad \text{ for all } \quad n \in \mathbb{N}.
	\end{equation*}
	That is all we needed.
	\demo
\end{example}

\begin{example}[Multipliers between Lorentz spaces $L_{p,q}$] \label{EXP: Lorentz symbol not OC but WC}
	Let $1 < p < q < \infty$. Again, in view of \cite[Theorem~4]{KLM19}, we have
	\begin{equation*}
		M(L_p(0,1),L_{p,q}(0,1)) = L_{\infty}(0,1).
	\end{equation*}
	Moreover, it follows from \cite[Proposition~5.1]{FJT75} that the Lorentz function spaces $L_{p,q}$ are $q$-DH, but since $q > 1$, they are not $1$-DH.
	Of course, the ideal $M(L_p(0,1),L_{p,q}(0,1))_o$ is trivial. On the other hand, due to the reflexivity of the space $L_p(0,1)$, each multiplication
	operator $M_{\lambda} \colon L_p(0,1) \rightarrow L_{p,q}(0,1)$ is weakly compact (see, e.g., \cite[p.~348]{AK06}). Therefore, in general, the multiplication
	operator can be weakly compact even though its symbol is not order continuous.
	\demo
\end{example}

\begin{remark}[About disjointly homogeneous K{\" o}the sequence spaces]
	It should also be added that the class of disjointly homogeneous K{\" o}the sequence spaces is not particularly rich, because, for example (and in
	opposition to the function case presented above), Orlicz spaces $\ell_M$ and Lorentz spaces $\ell_{p,q}$ are, excluding trivial cases,
	not disjointly homogeneous at all. This is because basic sequences in K{\" o}the sequence spaces are usually equivalent to disjoint sequences.
	The culprit for this are the echoes of the Bessaga--Pe{\l}czy{\' n}ski selection principle (cf. \cite[Proposition~3.1.10, p.~14]{AK06}).
	\demo
\end{remark}

\subsection{Two sides of the same coin}
Typically, (some) operator theoretic aspects of the multiplication operator $M_{\lambda} \colon X \rightarrow Y$ (that is, treating $M_{\lambda} \colon X \rightarrow Y$
as a member of $\mathscr{L}(X,Y)$) are reflected in function theoretic properties of its symbol $\lambda$ (which, of course,
live a fully-fledged life in $M(X,Y)$). Here, we aim to study this correspondence in the context specified by the next

\begin{theorem} \label{Proposition: odleglosci = odleglosci}
	{\it Let $X$ and $Y$ be two K{\" o}the spaces with the Fatou property both defined on the same $\sigma$-finite measure
		space $(\Omega,\Sigma,\mu)$.}
	\begin{enumerate}
		\item [(i)] {\it Suppose $(\Omega,\Sigma,\mu)$ is purely atomic. Assume further that either the space $X$ is reflexive, the space $Y$ is order continuous or the space $Y$ is rearrangement invariant. Then}
		\begin{equation*}
			\text{dist}_{\mathscr{L}(X,Y)}(M_{\lambda} \colon X \rightarrow Y, {\mathscr K}(X,Y)) = \text{dist}_{M(X,Y)}(\lambda, M(X,Y)_o).
		\end{equation*}
		\item [(ii)] {\it Suppose $(\Omega,\Sigma,\mu)$ is non-atomic. Assume further that the space $Y$ is} $1$-DH {\it and every simple
			function in $M(X,Y)$ has an absolutely continuous norm. Then}
	\end{enumerate}
	\begin{equation*}
		\text{dist}_{\mathscr{L}(X,Y)}(M_{\lambda} \colon X \rightarrow Y, {\mathscr W}(X,Y)) = \text{dist}_{M(X,Y)}(\lambda, M(X,Y)_o).
	\end{equation*}
\end{theorem}
\begin{proof}
	{\bf (i)} First of all, there is no doubt in assuming that $M(X,Y)_o \neq M(X,Y)$, because if both spaces $M(X,Y)_o$
	and $M(X,Y)$ are equal, then it follows from Theorem~\ref{THM: Essential norm Kothe sequence} and Theorem~\ref{thm: ess norm ri},
	that every multiplication operator $M_{\lambda}$ acting between $X$ and $Y$ is compact. Thus,
	\begin{equation*}
		\text{dist}_{\mathscr{L}(X,Y)}(M_{\lambda} \colon X \rightarrow Y, {\mathscr K}(X,Y)) = 0 = \text{dist}_{M(X,Y)}(\lambda, M(X,Y)_o).
	\end{equation*}
	
	Now, let us begin the proof with the following observation\footnote{This result can be seen as a spin-off from Theorem~3.2 in \cite{KKT22}.
	However, in the mentioned article we were more interested in the rearrangement invariant structure. Morally speaking, it may be instructive
	to compare our Theorem~\ref{THM: Essential norm Kothe sequence} and Theorem~\ref{THM: weak ess norm} with Theorem~3.4 from \cite{KKT22}.}
	which seems quite interesting in itself.
	
	\vspace{5pt} \noindent {\bf Claim~$\clubsuit$.}
	{\it Let $Z$ be a K{\" o}the sequence space. Suppose that the ideal $Z_o$ is non-trivial.
	Then, for $z = \{z_i\}_{i=1}^{\infty} \in Z$, we have}
	\begin{equation}
		\text{dist}_Z(z,Z_o) = \lim\limits_{n \rightarrow \infty} \norm{ \sum_{i=n}^{\infty} z_i e_i }_Z.
	\end{equation}
	\vspace{5pt}
	\noindent
	{\bf Proof of Claim $\clubsuit$.}
		Using the ideal property of the space $Z$ it is straightforward to see that 
		\begin{align*}
			\norm{z - y}_Z
				& \geqslant \lim\limits_{n \rightarrow \infty} \norm{ \sum_{i=n}^{\infty} (z_i - y_i) e_i }_Z \\
				& \geqslant \lim\limits_{n \rightarrow \infty} \abs{\norm{ \sum_{i=n}^{\infty} z_i e_i }_Z - \norm{ \sum_{i=n}^{\infty} y_i e_i }_Z},
		\end{align*}
		where the second inequality follows from the reverse triangle inequality. Consequently, for $y \in Z_o$,
		we have $\lim_{n \rightarrow \infty} \norm{ \sum_{i=n}^{\infty} y_i e_i }_Z = 0$, which means that
		\begin{equation*}
			\text{dist}_Z(z,Z_o)
				= \inf\{\norm{z - y}_Z \colon y \in Z_o\}
				\geqslant \lim\limits_{n \rightarrow \infty} \norm{ \sum_{i=n}^{\infty} z_i e_i }_Z.
		\end{equation*}
		
		To prove the opposite inequality just note that
		\begin{align*}
			\text{dist}_Z(z,Z_o)
				& = \inf\{\norm{z - y}_Z \colon y \in Z_o\} \\
				& \leqslant \lim\limits_{n \rightarrow \infty} \norm{z - \sum_{i=1}^{n-1} z_i e_i }_Z \\
				& = \lim\limits_{n \rightarrow \infty} \norm{ \sum_{i=n}^{\infty} z_i e_i }_Z,
		\end{align*}
		where the first inequality is due to the fact\footnote{Even a slightly more general statement is true, namely, if $Z$ is a K{\" o}the sequence space
		on $(\Omega,\Sigma,\mu)$ and $A \in \Sigma$ has finite measure, then $z\chi_{A} \in Z_o$ for every $z \in Z$. In other words, $c_{00} \subset Z_o$.}
		that $\{ x_i \}_{i=1}^{n-1} \in Z_o$. Thus we have completed the proof of Claim $\clubsuit$.
	\hfill $\blacksquare$
	\vspace{5pt}
	\noindent
	
	Let's go back to the proof of the theorem. Using Theorems~\ref{THM: Essential norm Kothe sequence} and \ref{thm: ess norm ri}
	once again together with the above Claim $\clubsuit$, but for the space  $M(X,Y)$ instead of $Z$, we have
	\begin{align*}
		\text{dist}_{\mathscr{L}(X,Y)}({M_{\lambda} \colon X \rightarrow Y, {\mathscr K}(X,Y)}) & = \norm{M_{\lambda} \colon X \rightarrow Y}_{ess} \\
		& = \lim\limits_{n \rightarrow \infty} \norm{ \sum_{i=n}^{\infty} \lambda_i e_i }_{M(X,Y)} \\
		& = \text{dist}_{M(X,Y)}({\lambda,(M(X,Y))_o}).
	\end{align*}
	That's all.
	
	{\bf (ii)} The proof of the second part is very similar. It requires two ingredients. First, if $M(X,Y)_o = M(X,Y)$,
	then, as Theorem~\ref{THM: weak ess norm} teaches us, every multiplication operator $M_{\lambda}$ between $X$ and $Y$ is weakly compact,
	so we have the equality
	\begin{equation*}
		\text{dist}_{\mathscr{L}(X,Y)}(M_{\lambda} \colon X \rightarrow Y, {\mathscr W}(X,Y)) = 0 = \text{dist}_{M(X,Y)}(\lambda, M(X,Y)_o).
	\end{equation*}
	Therefore, without loss of the generality we can assume that $M(X,Y)_o \neq M(X,Y)$. Next, we need an analogue of Claim~$\clubsuit$
	in the following form
	
	\vspace{5pt} \noindent {\bf Claim~$\heartsuit$.}
	{\it Let $Z$ be a K{\" o}the function space defined on a $\sigma$-finite measure space $(\Omega,\Sigma,\mu)$.
		Suppose that every simple function in $Z$ has an absolutely continuous norm. Then}
		\begin{equation}
			\text{dist}_Z(z,Z_o) = \lim\limits_{n \rightarrow \infty} \norm{z \chi_{\mathcal{U}_n(\lambda)}}_Z,
		\end{equation}
	{\it where, as in Theorem~\ref{THM: weak ess norm},}
		\begin{equation*}
			\mathcal{U}_n(\lambda) \coloneqq \bigcup\limits_{N\geqslant n}\Omega_N
				\cup \left\{ \omega \in \supp(\lambda) \colon \abs{\lambda(\omega)} > n \right\} \text{\it for } n \in \mathbb{N}
		\end{equation*}
	{\it and $\{\Omega_N\}_{N=1}^{\infty}$ is a decomposition of $\Omega$ into a pairwise disjoint family of sets of finite measure.}
	
	\vspace{5pt} We will omit the proof of this fact, as it does not differ significantly from that of Claim~$\clubsuit$. Finally, using
	Theorem~\ref{THM: weak ess norm} in combination with Claim $\heartsuit$ we complete the proof.
\end{proof}

\begin{corollary}[Weak essential norm of multiplication operator between rearrangement invariant function spaces] \label{COR: ess r.i. function}
	{\it Let $X$ and $Y$ be two rearrangement invariant function spaces with the Fatou property.
	Suppose that the space $Y$ is} $1$-DH {\it and that the ideal $M(X,Y)_o$ is non-trivial. Then}
	\begin{equation} \label{EQ: weak ess norm in r.i.}
		\norm{M_{\lambda} \colon X \rightarrow Y}_{w} = \lim\limits_{n \rightarrow \infty} \norm{\lambda^{\star} \chi_{\left( 0,\frac{1}{n} \right) \cup (n, \infty)}}_{M(X,Y)}.
	\end{equation}
\end{corollary}
\begin{proof}
	First of all, in the class of rearrangement invariant spaces we essentially have only two possibilities, namely, either the ideal of order continuous
	functions is trivial or it is equal to the closure of simple functions (see \cite[Theorem~5.5, pp.~67--68]{BS88} and \cite[Theorem~B]{KT17}). Hence,
	since the space $M(X,Y)$ is also rearrangement invariant (see \cite[Theorem~2.2(i)]{KLM12}) and, thanks to our assumptions, the ideal $M(X,Y)_o$ is non-trivial,
	so every simple function in $M(X,Y)$ has an absolutely continuous norm. Consequently, the use of Proposition~\ref{Proposition: odleglosci = odleglosci}(ii)
	is justified, and we have the following identifications
	\begin{equation*}
		\norm{M_{\lambda} \colon X \rightarrow Y}_{w} = \text{dist}_{\mathscr{L}(X,Y)}(M_{\lambda} \colon X \rightarrow Y, {\mathscr W}(X,Y))
			= \text{dist}_{M(X,Y)}(\lambda, M(X,Y)_o).
	\end{equation*}
	On the other hand, due \cite[Corollary~3.7]{KKT22}, we have
	\begin{equation*}
		\text{dist}_{M(X,Y)}(\lambda, M(X,Y)_o) = \lim\limits_{n \rightarrow \infty} \norm{\lambda^{\star} \chi_{\left( 0,\frac{1}{n} \right) \cup (n, \infty)}}_{M(X,Y)}.
	\end{equation*}
	Combining the above two equalities completes the proof.
\end{proof}

\section{{\bf Some applications}} \label{SECTION: Applications}

In \cite{RFRSSB19} Ramos-Fern{\' a}ndez, Rivera-Sarmiento and Salas-Brown computed the essential norm of multiplier $M_{\lambda}$
acting on the classical Ces{\` a}ro sequence spaces $ces_p$ into itself. They showed that
\begin{equation*}
	\norm{M_{\lambda} \colon ces_p \rightarrow ces_p}_{ess} = \limsup\limits_{n \rightarrow \infty} \abs{\lambda_n},
\end{equation*}
where the symbol $\lambda = \{\lambda_n\}_{n=1}^{\infty}$ belongs to $\ell_{\infty}$. In addition, they write the following
\begin{quote}
	{\it \enquote{Thus, it is natural to ask if the result obtained in this article is also valid for multiplication operator acting on any K{\" o}the sequence space $X$.}}
\end{quote}
In this section, we are interested not only in answering the above question, but we want to consider the non-algebraic version
of the above problem. Moreover, we will compute the essential norm of Fourier multipliers acting from a Banach space of analytic
functions into a separable K{\" o}the sequence space.

\subsection{Banach sequence spaces related to decreasing functions}

The {\bf abstract\footnote{They owe the nickname \enquote{abstract} due to Le{\' s}nik and Maligranda, who proposed this nomenclature in \cite{LM15a}.}
Ces{\' a}ro sequence space} $CX$ is most often defined as the {\it optimal domain} for Hardy's operator
$\mathscr{H} \colon x \rightsquigarrow \left[ n \rightsquigarrow \frac{1}{n} \sum_{i=1}^n \abs{x_n} \right]$,
that is, as the largest (in the sense of inclusion) Banach sequence space with the property that the operator $\mathscr{H}$
is bounded when acting into $X$. Explicitly, the space $CX$ is the collection of all sequences $x = \{x_n\}_{n=1}^{\infty}$
such that the norm
$\norm{x}_{CX} \coloneqq \norm{ \left\{ \frac{1}{n} \sum_{i=1}^n \abs{x_n} \right\}_{n=1}^{\infty} }_X$
is finite.

Abstract Ces{\' a}ro sequence spaces can be viewed as a direct generalization of the {\bf classical Ces{\` a}ro sequence spaces $ces_p \coloneqq C\ell_p$}
(see \cite{AM14} and references given there; see also Bennett's memoir \cite{Ben96} and Grosse-Erdmann book \cite{GE98}).
Moreover, they are closely related to Sinnamon's construction of the so-called {\it down spaces $X^{\downarrow}$} (see \cite{KMS07}, \cite{LM15a}, \cite{MS06} and \cite{Si94} for more).

Nowadays, the structure of these spaces, especially in its most general form, enjoys unflagging interest
(see, for example, \cite{ALM19}, \cite{CR16}, \cite{DS07}, \cite{KT24}, \cite{KKT22}, \cite{LM15a}, \cite{LM16}, \cite{MS06}, \cite{NP10} and their references).
More studies on the structure of the optimal domains can be found, for example, in \cite{Mas91} and
Okada, Ricker and S{\' a}nchez P{\' e}rez's monograph \cite{ORS08} and their references.

In a sense, a dual construction to the construction of Ces{\` a}ro sequence spaces are the Tandori sequence spaces.
For a K{\" o}the sequence space $X$, by the corresponding {\bf abstract Tandori sequence space $\widetilde{X}$} we understand a vector space
\begin{equation*}
	\widetilde{X} \coloneqq \left\{ x \colon \widetilde{x} \in X \right\}
\end{equation*}
furnished with the norm $\norm{x}_{\widetilde{X}} \coloneqq \norm{\widetilde{x}}_X$, where $\widetilde{x} = \left\{ \widetilde{x}_n \right\}_{n=1}^{\infty}$
is the {\bf least decreasing majorant} of a given sequence $x = \{x_n\}_{n=1}^{\infty}$ defined as
\begin{equation*}
	\widetilde{x} \coloneqq \left\{ \sup\limits_{m \geqslant n} \abs{x_m} \right\}_{n=1}^{\infty}.
\end{equation*}

The name of these spaces was proposed in \cite[p.~767]{LM16}, where the authors refer to K{\' a}roly Tandori's paper \cite{Tan55} from 1954.
However, despite being nameless, these spaces were investigated even earlier (see \cite{AM14} and \cite{LM15a}).
Their significance stems mostly from the following reasons, which we will only indicate here in the form of slogans with
appropriate references to the literature:
\begin{itemize}
	\item[$\bullet$] Halperin's $D$-type H{\" o}lder--Rogers inequality (see \cite{Hal53}),
	\item[$\bullet$] Sawyer's duality formula (see \cite{Saw90}),
	\item[$\bullet$] Halperin's and Sinnamon's level functions (see \cite{Hal53} and \cite{Si94}),
	\item[$\bullet$] Hardy's inequality (see \cite{Ste93}).
\end{itemize}

Let us start with the following preparatory result.

\begin{lemma}[Separable part of Tandori sequence spaces] \label{LEMMA: OC w Tandori}
	{\it Let $X$ be a K{\" o}the sequence space. Then}
	\begin{equation*}
		( \widetilde{X} )_o = \widetilde{X_o}.
	\end{equation*}
\end{lemma}
\begin{proof}
	Let us start by recalling the well-known fact: {\it Let $Z$ be a K{\" o}the sequence space. Then}
	\begin{equation} \tag{$\spadesuit$}
		z = \{z_n\}_{n=1}^{\infty} \in Z \textit{ belongs to the ideal } Z_o
			\textit{ if, and only if, } \lim\limits_{n \rightarrow \infty} \norm{ z\chi_{\{n,n+1,...\}} }_Z = 0.
	\end{equation}
	If we keep in mind ($\spadesuit$), the rest of the proof basically follows from the following simple-to-see inequality
	\begin{equation} \label{INEQ: tandori X o = X o tandori}
		\stackon[-8pt]{$x\chi_{\{n, n+1, ...\}}$}{\vstretch{1.5}{\hstretch{2.4}{\widetilde{\phantom{\;\;\;\;\;\;\;\;}}}}} = \widetilde{x}\chi_{\{n, n+1, ...\}}
			\quad \text{ for all } \quad n \in \mathbb{N}.
	\end{equation}
	In fact, after imposing the norm $\norm{\bullet}_X$ on both sides of (\ref{INEQ: tandori X o = X o tandori}), we get
	\begin{equation} \label{EQ: 1}
		\norm{x\chi_{\{n, n+1, ...\}}}_{\widetilde{X}} = \norm{\stackon[-8pt]{$x\chi_{\{n, n+1, ...\}}$}{\vstretch{1.5}{\hstretch{2.4}{\widetilde{\phantom{\;\;\;\;\;\;\;\;}}}}}}_X
			= \norm{\widetilde{x}\chi_{\{n, n+1, ...\}}}_X \quad \text{ for all } \quad n \in \mathbb{N}.
	\end{equation}
	Now, if $x \in ( \widetilde{X} )_o$, then it follows from ($\spadesuit$) that $\lim_{n \rightarrow \infty} \norm{x\chi_{\{n, n+1, ...\}}}_{\widetilde{X}} = 0$.
	Therefore, in view of (\ref{EQ: 1}), also $\lim_{n \rightarrow \infty} \norm{\widetilde{x}\chi_{\{n, n+1, ...\}}}_X = 0$.
	But this exactly means that $\widetilde{x} \in X_o$. Consequently, $x \in \widetilde{X_o}$ and we have the inclusion $( \widetilde{X} )_o \subset \widetilde{X_o}$.
	On the other hand, take $x \in \widetilde{X_o}$. Then $\widetilde{x} \in X_o$, so  $\lim_{n \rightarrow \infty} \norm{\widetilde{x}\chi_{\{n, n+1, ...\}}}_X = 0$.
	However, using again (\ref{EQ: 1}), we see that $\lim_{n \rightarrow \infty} \norm{x\chi_{\{n, n+1, ...\}}}_{\widetilde{X}} = 0$ as well.
	This gives us that $x \in ( \widetilde{X} )_o$ and, in consequence, $( \widetilde{X} )_o \supset \widetilde{X_o}$.
	The proof is complete.
\end{proof}

Later, it will be necessary to use some of the results obtained by the authors in \cite{KT24}.
For this reason, let us briefly recall the following

\begin{construction}
	Let $X$ be a rearrangement invariant sequence space.
	With $X$ we can associate a K{\" o}the sequence space $\mathbf{E}(X)$ defined as a vector space
	\begin{equation*} \label{Def: E(X) function}
		\mathbf{E}(X) \coloneqq \left\{ x = \{x_n\}_{n=0}^{\infty} \colon \sum_{n=0}^{\infty} x_n \chi_{\Delta_n} \in X \right\}
	\end{equation*}
	and equipped with the norm  $\norm{x}_{\mathbf{E}(X)} \coloneqq \norm{ \sum_{n=0}^{\infty} x_n \chi_{\Delta_n} }_X$.
	Here, $\Delta_j$ is the $n^{\text{th}}$ dyadic interval, that is, $\Delta_n \coloneqq \{ 2^{n}, 2^{n}+1, 2^{n+1}-1 \}$ for $n \in \mathbb{N}_0$.
	\demo
\end{construction}

Interestingly, the above construction goes back to Kalton \cite{Ka93a}.
However, in \cite{KT24} it was used to obtain an abstract variant of the so-called Grosse-Erdmann's blocking technique (see \cite{GE98}).

Now we are ready to present and also prove the main result of this section.

\begin{theorem}[Essential norm of multipliers acting between abstract Ces{\` a}ro and Tandori sequence spaces] \label{THM: multipliers Cezaro i Tandori}
	{\it Let $X$ and $Y$ be two rearrangement invariant sequence spaces with the Fatou property.
		Suppose that either the space $X$ is reflexive or the space $Y$ is separable.
		Further, suppose that either the space ${\bf E}(Y)$ factorizes though ${\bf E}(X)$, that is,
		\begin{equation*}
			{\bf E}(Y) = {\bf E}(X) \odot M({\bf E}(X),{\bf E}(Y)),
		\end{equation*}
		or that both spaces $X$ and $Y$ are strongly separated, that is,
		\begin{equation*}
			\alpha_X > \beta_Y.
		\end{equation*}
		Then, we have
		\begin{equation*}
			\norm{M_{\lambda} \colon \widetilde{X} \rightarrow \widetilde{Y}}_{ess} \approx \lim\limits_{n \rightarrow \infty} \norm{\widetilde{\lambda}\chi_{\{n, n+1, ...\}}}_{M(X,Y)}.
		\end{equation*}
		Additionally, if Hardy's operator $\mathscr{H}$ is bounded on both spaces $X$ and $Y$, then also
		\begin{equation*}
			\norm{M_{\lambda} \colon CX \rightarrow CY}_{ess} \approx \lim\limits_{n \rightarrow \infty} \norm{\widetilde{\lambda}\chi_{\{n, n+1, ...\}}}_{M(X,Y)}.
		\end{equation*}
		Moreover, both of the above estimates depend on $X$ and $Y$ only.
		In particular, the multiplication operator $M_{\lambda} \colon \begin{array}{l} CX \rightarrow CY \\ \;\;\: \widetilde{X} \rightarrow \widetilde{Y} \end{array}$
		is compact if, and only if, the function $\widetilde{\lambda} = \left\{ \widetilde{\lambda_n} \right\}_{n=1}^{\infty}$ belongs to the ideal $M(X,Y)_o$.}
\end{theorem}

\begin{proof}
	As a warm-up, let us explain the following, not particularly difficult issues, that will be very useful later.

	(a) Due to our assumption about the boundedness of Hardy's operator $\mathscr{H}$ on $X$ and $Y$, both spaces $CX$ and $CY$ are non-trivial
	(this easily follows from the definition of the Ces{\` a}ro space construction, however, this statement is also a form of Hardy's inequality;
	see~\cite[p.~1]{Ben96} and \cite[pp.~19--21 and 147--159]{KMP07}).
			
	(b) Let $Z$ be a K{\" o}the sequence space with the Fatou property. Then both spaces $CZ$ and $\widetilde{Z}$ have the Fatou property
	(see, for example, \cite[Theorem~1(d)]{LM15a}; although, formally, the proof given in \cite{LM15a} covers only K{\" o}the function spaces,
	it works just as well for their sequence counterparts).
			
	(c) Both spaces $\widetilde{Z}$ and $CZ$ are order continuous provided $Z$ is an order continuous K{\" o}the sequence space
	(this is a direct consequence of Lemma~\ref{LEMMA: OC w Tandori} and a nearly direct consequence of Lemma~11 in \cite{KT17};
	cf. the proof of Proposition~5.1 in \cite{KKT22}).
			
	(d) Lastly, we claim that
		
	\vspace{5pt} \noindent {\bf Claim $\diamondsuit$.}
	{\it Let $Z$ be a reflexive rearrangement invariant sequence space with the Fatou property.
		Suppose that the Hardy's operator $\mathscr{H}$ is bounded on $Z$.
		Then both spaces $\widetilde{Z}$ and $CZ$ are reflexive.}
	
	\vspace{5pt}
	\noindent
	{\bf Proof of Claim $\diamondsuit$}
		Both spaces are related to each other through duality, so we will limit ourselves to the proof in the case of the Ces{\` a}ro space construction.
		
		Since the space $Z$ is reflexive and has the Fatou property, so it follows from \cite[Corollary~4.4, p.~23]{BS88} that both spaces $Z$ and $Z^{\times}$
		are order continuous. Because of (b) and, one more time, \cite[Corollary~4.4, p.~23]{BS88}, to complete the proof it is enough to show that the spaces
		$CZ$ and $(CZ)^{\times}$ are order continuous. Now, in the first case, that is, when we consider $CZ$, it results immediately from (c),
		while in the second case, that is, when we consider $(CZ)^{\times}$, things are a bit more complicated.
		Here, using the Le{\' s}nik and Maligranda duality results from \cite{LM15a}, we have
		\begin{equation*}
			(CX)^{\times} = \widetilde{X^{\times}}
		\end{equation*}
		(see \cite[Theorem~6]{LM15a}). But, with Lemma~\ref{LEMMA: OC w Tandori} in mind, we also have
		\begin{equation*}
			\left( (CX)^{\times} \right)_o = \left( \widetilde{X^{\times}} \right)_o = \widetilde{(X^{\times})_o} = \widetilde{X^{\times}} = (CX)^{\times}.
		\end{equation*}
		This means that also the space $(CX)^{\times}$ is order continuous. That is all we needed to finish the proof of Claim~$\diamondsuit$.
	\hfill $\blacksquare$
	\vspace{5pt}
	\noindent

	Now, the essential part of the proof is based on the following three steps.
	
	{\bf First step.} Due to Theorem~6.a.3 from \cite{KT24} (see also \cite[Theorem~6]{KLM19} and \cite[pp. 67--77]{Ben96}), we have
	\begin{equation} \label{EQ: multiplikatory ktorych nie ma}
		M(CX,CY) = \widetilde{M(X,Y)}
	\end{equation}
	and
	\begin{equation} \label{EQ: multiplikatory tandori}
		M(\widetilde{X},\widetilde{Y}) = \widetilde{M(X,Y)}.
	\end{equation}
	In consequence,
	\begin{equation*}
		\norm{M_{\lambda} \colon \begin{array}{l} CX \rightarrow CY \\ \;\;\: \widetilde{X} \rightarrow \widetilde{Y} \end{array}}
			\approx  \norm{\lambda}_{\widetilde{M(X,Y)}}.
	\end{equation*}

	{\bf Second step.} In view of first step, (a), (b), (c) and (d) there is noting to prevent us from using Theorem~\ref{THM: Essential norm Kothe sequence}.
	Thus, we get the formula
	\begin{equation*}
		\norm{M_{\lambda} \colon  CX \rightarrow CY }_{ess}
			\approx \lim\limits_{n \rightarrow \infty} \norm{\lambda\chi_{\{n, n+1, ...\}}}_{\widetilde{M(X,Y)}}.
	\end{equation*}
	Moreover, this means that the multiplication operator $M_{\lambda} \colon \begin{array}{l} CX \rightarrow CY \\ \;\;\: \widetilde{X} \rightarrow \widetilde{Y} \end{array}$
	is compact if, and only if, $\lambda \in \left( \widetilde{M(X,Y)} \right)_o$.
	However, as shown in Lemma~\ref{LEMMA: OC w Tandori},
	\begin{equation*}
		\left( \widetilde{M(X,Y)} \right)_o = \widetilde{M(X,Y)_o}.
	\end{equation*}
	Therefore,
	\begin{equation*}
		\lambda \in \left( \widetilde{M(X,Y)} \right)_o \quad \text{ if, and only if, } \quad \lambda \in \widetilde{M(X,Y)_o}.
	\end{equation*}
	The latter clearly means that $\widetilde{\lambda} \in M(X,Y)_o$.
	
	{\bf Third step.} Finally, using the formula (\ref{EQ: 1}) form the proof of Lemma~\ref{LEMMA: OC w Tandori}, we have
	\begin{equation*}
		\norm{\lambda\chi_{\{n, n+1, ...\}}}_{\widetilde{M(X,Y)}}
			= \norm{\stackon[-8pt]{$\lambda \chi_{\{n, n+1, ...\}}$}{\vstretch{1.5}{\hstretch{2.4}{\widetilde{\phantom{\;\;\;\;\;\;\;\;}}}}}}_{M(X,Y)}
			= \norm{\widetilde{\lambda} \chi_{\{n, n+1, ...\}}}_{M(X,Y)}.
	\end{equation*}
\end{proof}

Algebraic situation is, of course, much simpler.

\begin{corollary}
	{\it Let $X$ be a separable K{\" o}the sequence space with the Fatou property. Then}
	\begin{equation*}
		\norm{M_{\lambda} \colon \begin{array}{l} CX \circlearrowleft \\ \;\, \widetilde{X} \circlearrowleft \end{array}}_{ess}
			= \limsup\limits_{n \rightarrow \infty} \abs{\lambda_n}.
	\end{equation*}
	{\it In particular, the multiplication operator $M_{\lambda} \colon \begin{array}{l} CX \circlearrowleft \\ \;\, \widetilde{X} \circlearrowleft \end{array}$
		is compact if, and only if, $\lambda \in c_0$.}
\end{corollary}
\begin{proof}
	This result is an easy consequence of Corollary~\ref{COR: compact sequence X to X} and (c) from the proof of Theorem~\ref{THM: multipliers Cezaro i Tandori}.
	Indeed, since $M(CX,CX) \equiv \ell_{\infty}$ and the space $CX$ is order continuous (this follows from our assumptions together with \cite[Lemma~1]{LM16};
	cf. \cite[Proposition~5.1]{KKT22}), so
	\begin{equation*}
		\norm{M_{\lambda} \colon CX \circlearrowleft }_{ess}
			= \lim\limits_{n \rightarrow \infty} \norm{\lambda \chi_{\{n, n+1, ...\}}}_{\ell_{\infty}}
			= \lim\limits_{n \rightarrow \infty} \left( \sup\limits_{m \geqslant n} \abs{\lambda_m} \right)
			= \limsup\limits_{n \rightarrow \infty} \abs{\lambda_n}.
	\end{equation*}
	The part concerning Tandori spaces is proved in a completely analogous way.
\end{proof}

Let us close this section with two examples concerning the compact multipliers between Ces{\' a}ro--Orlicz sequence spaces.
For a Young function $M$, by the {\bf Ces{\' a}ro--Orlicz sequence space $ces_M$} we understand here the space $C\ell_M$,
where $\ell_M$ is the well-known Orlicz sequence space.

\begin{example} \label{EXAMPLE: compact cesp -> cesq faktoryzacja}
	Let $M$ and $N$ be two Young functions such that either $M$ or $N$ and its Young conjugate $N^*$
	satisfy the $\Delta_2$-condition for small arguments (see \cite{Mal83} for details).
	Suppose that $\alpha_{\ell_M} > 1$ and $\alpha_{\ell_N} > 1$ (equivalently, we can use the so-called
	Matuszewska--Orlicz indices $\alpha_M$ and $\alpha_N$ here; see \cite[Proposition~2.b.5, p.~139]{LT79}).
	Further, suppose that $M(\ell_M,\ell_N) \neq \ell_{\infty}$.
	Then, due to \cite[Example~6.b.2]{KT24}, we have
	\begin{equation} \label{first situation}
		M(ces_M,ces_N) = \widetilde{\ell_{N \ominus M}}.
	\end{equation}
	Here, the function $N \ominus M$ is the so-called generalized Young conjugate of $M$ with respect to $N$, that is,
	$(N \ominus M)(t) \coloneqq \sup_{t > 0} \left\{ N(ts) - M(s) \right\}$ (see \cite{DR00} and \cite{LT21}).
	In consequence, using Theorem~\ref{THM: multipliers Cezaro i Tandori}, we see that the multiplication operator
	$M_{\lambda} \colon ces_M \rightarrow ces_N$ is compact if, and only if, $\widetilde{\lambda} \in (\ell_{N \ominus M})_o$.
	Moreover, it follows from \cite[Theorem~3.3(a), p.~17]{Mal83}, that the ideal $(\ell_{N \ominus M})_o$ coincide
	with the space of finite elements in $\ell_{N \ominus M}$, that is,
	\begin{equation*}
		(\ell_{N \ominus M})_o
			= \left\{ x = \{x_n\}_{n=1}^{\infty} \colon \sum_{n=1}^{\infty} (N \ominus M)(\eta x_n) < \infty \text{ for all } \eta > 0 \right\}.
	\end{equation*}
	A more classical situation is as follows. Let $1 < p < q < \infty$ with $1/r = 1/p - 1/q$.
	Then, using \eqref{first situation} (see also \cite[p.~69]{Ben96} and \cite[Corollary~6.b.4]{KT24}), it is straightforward to see that
	\begin{equation*}
		M(ces_q, ces_p) = \widetilde{\ell_r}.
	\end{equation*}
	Therefore, the multiplication operator $M_{\lambda} \colon ces_q \rightarrow ces_p$ is compact if, and only if, $\widetilde{\lambda} \in \ell_r$.
	\demo
\end{example}

	\begin{example} \label{Remark: ces_p -> ces_q [ABR18]}
		Let $M$ and $N$ be two Young functions exactly as in Example~\ref{EXAMPLE: compact cesp -> cesq faktoryzacja}
		except that now we assume that $M(\ell_M,\ell_N) = \ell_{\infty}$ (or, which is one thing, $\ell_M \hookrightarrow \ell_N$).
		Again, using \cite[Example~6.b.1]{KT24}, we have
		\begin{equation} \label{second situation}
			M(ces_M,ces_N) = \ell_{\infty}(w_{M \to N}),
		\end{equation}
		where $w_{M \to N}(n) \coloneqq M^{-1}(1/n)/N^{-1}(1/n)$ for $n \in \mathbb{N}$.
		Therefore, this time, the multiplication operator $M_{\lambda} \colon ces_M \rightarrow ces_N$ is compact if, and only if,
		\begin{equation*}
			\left\{ \frac{M^{-1}(1/n)}{N^{-1}(1/n)} \lambda_n \right\}_{n=1}^{\infty} \in c_0.
		\end{equation*}
		In particular, for $1 < p < q < \infty$, we have
		\begin{equation}
			M(ces_p, ces_q) = \ell_{\infty}\left( w_{p \to q} \right),
		\end{equation}
		where $w_{p \to q} = n^{1/q - 1/p}$ for $n \in \mathbb{N}$, so the multiplication operator $M_{\lambda} \colon ces_p \rightarrow ces_q$
		is compact if, and only if, $\left\{ n^{1/q-1/p} \lambda_n \right\}_{n=1}^{\infty} \in c_0$.
		In this way, we recreated Albanese, Bonet and Ricker's result from \cite{ABR18} (precisely, see \cite[Proposition~2.2]{ABR18}).
		\demo
\end{example}

\begin{remark} \label{REMARK: jak wynioskowac Tandori lp -> lq}
	Let $1 < p < q < \infty$ with $1/r = 1/p - 1/q$.
	Using once again \eqref{EQ: multiplikatory tandori} (that is, Theorem~6.a.3 from \cite{KT24}) in tandem with
	Examples~\ref{EXAMPLE: compact cesp -> cesq faktoryzacja} and \ref{Remark: ces_p -> ces_q [ABR18]},
	it is straightforward to see that the multiplication operators
	$M_{\lambda} \colon \widetilde{\ell_q} \rightarrow \widetilde{\ell_p}$ and $M_{\xi} \colon \widetilde{\ell_p} \rightarrow \widetilde{\ell_q}$
	are compact if, and only if, $\widetilde{\lambda} \in \ell_r$ and $\left\{ n^{1/q-1/p} \xi_n \right\}_{n=1}^{\infty} \in c_0$, respectively.
	This reproduces Bonet and Ricker's results from \cite{BR20} (precisely, see Proposition~4.1 and Proposition~4.4 therein).
	\demo
\end{remark}

\subsection{Abstract Ces{\` a}ro function spaces}

Let $X$ is a K{\" o}the function space defined on the unit interval $(0,1)$ or the semi-axis $(0,\infty)$
both equipped with the usual Lebesgue measure.
Then, the {\bf continuous Hardy operator} $\mathscr{H}$ is defined by
$\mathscr{H} \colon f \rightsquigarrow \left[ x \rightsquigarrow \frac{1}{x}\int_0^x f(t)dt \right]$
and the corresponding {\bf abstract Ces{\` a}ro function space $CX$}, is defined as the vector space of all real-valued measurable
functions, say $f$, such that $\mathscr{H}(\abs{f}) \in X$, equipped with the norm $\norm{f}_{CX} \coloneqq \norm{\mathscr{H}(\abs{f})}_X$.

Here, we do not really need more than the above definition, but the isomorphic and isometric structure of the Ces{\` a}ro function spaces
has been the subject of extensive research and has many non-trivial results to offer (see, for example, \cite{ALM19}, \cite{CR16}, \cite{DS07},
\cite{KKM21}, \cite{KKT22}, \cite{KMS07}, \cite{KT17}, \cite{KT24}, \cite{LM15a}, \cite{LM16} and \cite{MS06}; see also a survey \cite{AM14} and references therein).

Let us get to the point (for unexplained notation, we refer directly to \cite{KT24}).

\begin{theorem}[Essential norm of multiplication operators between abstract Ces{\` a}ro function spaces]\label{THM: ess norm Cezaro and Tandori function}
	{\it Let $X$ and $Y$ be two K{\" o}the function spaces with the Fatou property both defined on the semi-axis $(0,\infty)$ equipped with the Lebesgue measure.
		Suppose that either the space $X$ is separable or the space $Y$ is reflexive.
		Further, suppose that either the space ${\bf E}(Y)$ factorizes though ${\bf E}(X)$, that is,
		\begin{equation*}
			{\bf E}(Y) = {\bf E}(X) \odot M({\bf E}(X),{\bf E}(Y)),
		\end{equation*}
		or that both spaces $X$ and $Y$ are strongly separated, that is,
		\begin{equation*}
			\alpha_X > \beta_Y.
		\end{equation*}
		Then, we have
		\begin{equation*}
				\norm{M_{\lambda} \colon CX \rightarrow CY}_{ess}
					\approx \norm{\lambda}_{\widetilde{M(X,Y)}}
					= \norm{M_{\lambda} \colon CX \rightarrow CY}
		\end{equation*}
		and the above estimate depend on $X$ and $Y$ only.
		Therefore, in particular, there is a lack of compact multiplication operators acting between two abstract Ces{\` a}ro function spaces.}
\end{theorem}
\begin{proof}
	The proof of this fact is to mimic the argument used in the proof of Theorem~\ref{THM: multipliers Cezaro i Tandori}. And so:
	
		($\alpha$) Since Hardy's operator $\mathscr{H}$ is bounded on both spaces $X$ and $Y$, so it follows from the optimal domain's point of view
		that the spaces $CX$ and $CY$ are both non-trivial. Note, however, that the boundedness of Hardy's operator $\mathscr{H}$ is in general not
		necessary for a non-triviality of the Ces{\` a}ro function space (see Example~3.3(d) and Example~~3.3(f) in \cite{KKM21}).
		
		($\beta$) As in the sequence case, the Fatou property and order continuity property are inherited by the corresponding Ces{\` a}ro function space
		(here we need \cite[Lemma~1(d)]{LM15a} and \cite[Lemma~11]{KT17}, respectively). Note also that the abstract Ces{\` a}ro function spaces
		are {\it never} reflexive, because, for example, they contain a complemented copy of $L_1(0,1)$ (see \cite[Lemma~4.1]{KKM21}).
		
		($\gamma$) In view of the two previous steps, we can now use Theorem~\ref{THM: essnorm multipliers function} to conclude that
		\begin{equation*}
			\norm{M_{\lambda} \colon CX \rightarrow CY}_{ess} \approx \norm{\lambda}_{M(CX,CY)} = \norm{M_{\lambda} \colon CX \rightarrow CY}.
		\end{equation*}
	
		($\delta$) Finally, due to our assumptions and \cite[Theorem~6.a.3]{KT24} (see also \cite[Theorem~6]{KLM19}), we have
		\begin{equation*}
			M(CX,CY) = \widetilde{M(X,Y)}.
		\end{equation*}
	
		All this together with Theorem~\ref{THM: essnorm multipliers function} completes the proof.
\end{proof}

Let $1 < p \leqslant \infty$.
Recall, that the {\bf classical Ces{\` a}ro function spaces $Ces_p$} is defined as $Ces_p \coloneqq CL_p$.
It was proved by Al Alam, Gaillard, Habib, Lef{\` e}vre and Maalouf in \cite{AGHLM18} that the essential norm
of the multiplier $M_{\lambda} \colon Ces_p(0,1) \rightarrow Ces_p(0,1)$ is precisely equal to sup-norm of the symbol $\lambda$ (see \cite[Theorem~5.2]{AGHLM18}).
Thus, for the sake of completeness, let us also formulate the following result.

\begin{corollary}
	{\it Let $X$ be a separable K{\" o}the function spaces with the Fatou property defined either on the semi-axis $(0,\infty)$
	or the unit interval $(0,1)$ equipped with the Lebesgue measure. Then}
		\begin{equation*}
			\norm{M_{\lambda} \colon CX \circlearrowleft}_{ess} = \norm{\lambda}_{L_{\infty}} = \norm{M_{\lambda} \colon CX \circlearrowleft}.
		\end{equation*}
		{\it In particular, there are no non-trivial compact multiplication operators acting from $CX$ into itself.}
\end{corollary}
\begin{proof}
	To begin with, recall that $M(CX,CX) \equiv L_{\infty}$. Moreover, it follows from \cite[Lemma~11]{KT17} and, respectively, \cite[Lemma~1(d)]{LM15a}
	that if $X$ is order continuous or, respectively, has the Fatou property, then the same can be said about the space $CX$. Now, this result is an immediate
	consequence of Corollary~\ref{COR: compact function X to X}.
\end{proof}

\subsection{Banach spaces of analytic functions}
In this subsection, we will improve and complete the main result from Mleczko's paper \cite{Mle09} (precisely, see \cite[Theorem~2.5]{Mle09}).
However, before we do that, we need short preliminaries to clarify our notation.

Let $\mathscr{H}(\mathbb{D})$ be the space of all analytic functions on the unit disc $\mathbb{D} \coloneqq \{z \in \mathbb{C} \colon \abs{z} < 1\}$.
Further, let $E$ be a K{\" o}the sequence space. With a linear subspace, say $H(\mathbb{D})$, of $\mathscr{H}(\mathbb{D})$ we can associate
the space
\begin{equation*}
	\widehat{H}(\mathbb{D}) \coloneqq \left\{ \left\{ \widehat{f}(n) \right\}_{n=0}^{\infty} \colon \sum_{n=0}^{\infty} \widehat{f}(n)\chi_n \in H(\mathbb{D}) \right\}
\end{equation*}
of Taylor's coefficients of functions from $H(\mathbb{D})$, where $\chi_n(z) = z^n$ for $z \in \mathbb{D}$ and $n \in \mathbb{N}_0$.
Then with any sequence $\lambda = \{\lambda_n\}_{n=0}^{\infty}$ from $M(\widehat{H}(\mathbb{D}),E)$ we can associate its {\bf Fourier multiplier operator}
$\mathscr{M}_{\lambda} \colon H(\mathbb{D}) \rightarrow E$ defined by
\begin{equation*}
	\mathscr{M}_{\lambda} \colon f \rightsquigarrow \left\{ \lambda_n \widehat{f}(n) \right\}_{n=0}^{\infty},
\end{equation*}
where $f = \sum_{n=0}^{\infty} \widehat{f}(n)\chi_n \in H(\mathbb{D})$ for $n \in \mathbb{N}_0$.

Let $X$ be a K{\" o}the space defined on the unit circle $\mathbb{T} \coloneqq \{z \in \mathbb{C} \colon \abs{z} = 1\}$
equipped with the normalized Lebesgue measure $dm(t) = \abs{dt}/2\pi$.
Then the {\bf abstract Hardy space $H[X]$} is defined as a vector space
\begin{equation*}
	H[X] \coloneqq \left\{ f \in X \colon \widehat{f}(n) = 0 \text{ for all } n < 0 \right\}
\end{equation*}
furnished with the norm inherited from $X$. Equivalently, via convolution with the Poisson kernel, we can treat the Hardy space $H[X]$
as a space of analytic functions on the unit disc $\mathbb{D}$. In particular, the space $H[L_p]$ with $1 \leqslant p \leqslant \infty$
is nothing else, but the {\bf classical Hardy space} $H_p$. For much more information about the theory of Hardy spaces we refer to Duren's monograph \cite{Dur76}.

We are ready to prove the promised result.

\begin{theorem}[Essential norms of Fourier multipliers] \label{THM: ess norm holomorphic}
	{\it Let $H(\mathbb{D})$ be a Banach space of analytic function on the unit disc $\mathbb{D}$ intermediate between $H_{\infty}(\mathbb{D})$
		and $H_2(\mathbb{D})$. Moreover, let $E$ be a separable K{\" o}the sequence space with the Fatou property. Then}
		\begin{equation*}
			\norm{\mathscr{M}_{\lambda} \colon H(\mathbb{D}) \rightarrow E}_{ess} = \lim\limits_{n \rightarrow \infty} \norm{ \sum_{i=n}^{\infty} \lambda_i e_i }_{M(\ell_2,E)}.
		\end{equation*}
		{\it In particular, the Fourier multiplier $\mathscr{M}_{\lambda} \colon H(\mathbb{D}) \rightarrow E$ is compact if, and only if, $\lambda \in M(\ell_2,E)_o$.}
\end{theorem}
\begin{proof}
	The proof basically consists of two steps.
	
	{\bf First step}. First, we need to transfer the situation to our \enquote{ideal} word. This can be done using the so-called solid envelopes
	together with Anderson and Shield's results from \cite{AS76}. More precisely, thanks to our assumption that
	\begin{equation*}
		H_{\infty}(\mathbb{D}) \hookrightarrow H(\mathbb{D}) \hookrightarrow H_2(\mathbb{D})
	\end{equation*}
	and, moreover, knowing that the solid envelope of $H_{\infty}(\mathbb{D})$ and $H_2(\mathbb{D})$ is $\ell_2$
	(the harder part of this result is essentially due to S. Kisliakov; see \cite[Proposition~1.1]{Mle09}), it is not difficult to see that the following diagram commutes
	\begin{equation*}
		\xymatrix{
			H(\mathbb{D}) \ar[d]_{\text{id}} \ar[r]^{\mathscr{M}_{\lambda}} & E\\
			H_2 \ar[r]^{\mathscr{F}} & \ell_2 \ar[u]_{M_{\lambda}}}
	\end{equation*}
	Here $\lambda \in \widehat{H}(\mathbb{D})$ and $\mathscr{F}(f) \coloneqq \left\{ \widehat{f}(n) \right\}_{n=0}^{\infty}$ for $f \in H(\mathbb{D})$
	(see \cite[Corollary~1.2]{Mle09}). But this clearly means that
	\begin{equation*}
		\mathscr{M}(H(\mathbb{D}),E) = M(\ell_2,E).
	\end{equation*}

	{\bf Second step}. Since we assumed that the space $E$ is separable, so to finish the proof it is enough to apply Theorem~\ref{THM: Essential norm Kothe sequence}.
\end{proof}

\begin{remark}[About Theorem~\ref{THM: ess norm holomorphic}]
	Let $X$ be a K{\" o}the space defined on the unit circle $\mathbb{T}$.
	Suppose that $X$ is an intermediate space between $L_{\infty}(\mathbb{T})$ and $L_2(\mathbb{T})$.
	Then, it is straightforward to see, that the corresponding Hardy space $H[X]$ is an intermediate space between $H_{\infty}(\mathbb{D})$ and $H_2(\mathbb{D})$.
	For this reason, the abstract Hardy spaces may serve here as a main illustration for Theorem~\ref{THM: ess norm holomorphic}.
	\demo
\end{remark}

One immediate conclusions from Theorem~\ref{THM: ess norm holomorphic} is the following Mleczko's result from \cite{Mle09} (precisely, see Theorem~2.5 ibidem).

\begin{corollary}[P. Mleczko, 2009]
	{\it Let $H(\mathbb{D})$ be a Banach space of analytic function on the unit disc $\mathbb{D}$ intermediate between $H_{\infty}(\mathbb{D})$
		and $H_2(\mathbb{D})$. Moreover, let $E$ be a separable K{\" o}the sequence space with the Fatou property intermediate between
		$\ell_2$ and $\ell_{\infty}$. Then the Fourier multiplier $\mathscr{M}_{\lambda} \colon H(\mathbb{D}) \rightarrow E$ is compact if, and
		only if,
		\begin{equation*}
			\limsup_{n \rightarrow \infty} \abs{\lambda_n} = 0.
		\end{equation*}
	}
\end{corollary}
\begin{proof}
	With the extra assumption that $\ell_2 \hookrightarrow E$, we have $M(\ell_2,E) = \ell_{\infty}$. Thus,
	\begin{align*}
		\lim\limits_{n \rightarrow \infty} \norm{ \sum_{i=n}^{\infty} \lambda_i e_i }_{M(\ell_2,E)}
			& = \lim\limits_{n \rightarrow \infty} \norm{ \sum_{i=n}^{\infty} \lambda_i e_i }_{\ell_{\infty}} \\
			& = \lim\limits_{n \rightarrow \infty} \left( \sup_{m \geqslant n} \abs{\lambda_m} \right) \\
			& = \limsup_{n \rightarrow \infty} \abs{\lambda_n}.
	\end{align*}
\end{proof}

Let us close this section with the following simple example.

\begin{example}
	Let $1 \leqslant q < 2$ with $1/r = 1/q - 1/2$. Then it is straightforward to see that $M(\ell_2,\ell_q) = \ell_r$.
	Next, let $2 \leqslant q < \infty$. In this case $\ell_2 \hookrightarrow \ell_q$, so $M(\ell_2, \ell_q) = \ell_{\infty}$.
	Therefore, due to Theorem~\ref{THM: ess norm holomorphic}, the Fourier multiplier $\mathscr{M}_{\lambda} \colon H_p \rightarrow \ell_q$
	is compact if, and only if, its symbol $\lambda = \{\lambda_n\}_{n=1}^{\infty}$ belongs either to $\ell_r$, provided
	$1 \leqslant q < 2$, or to $c_0$, provided $2 \leqslant q < \infty$.
	\demo
\end{example}

\section{{\bf Arithmetic, remarks and examples}} \label{SECTION: examples}

Theorem~\ref{THM: Essential norm Kothe sequence} provides a simple recipe to check whether
the pointwise multiplier $M_{\lambda} \colon X \rightarrow Y$ is compact.
Namely, it is enough to first calculate the space of pointwise multipliers $M(X,Y)$ and, then,
find its separable part, that is, the ideal $M(X,Y)_o$.
The functions belonging to $M(X,Y)_o$ are exactly those symbols that generate compact multipliers.

In general, however, the description of $M(X,Y)_o$ is impossible.
The main reason for this is the observation that the space $M(X,Y)$ depends on the \enquote{relative}
position of the spaces $X$ and $Y$. For instance, $M(\ell_1,\ell_{\infty}) \equiv \ell_{\infty}$
and $M(\ell_{\infty},\ell_{1}) \equiv \ell_{1}$.
Anyhow, even when $X$ and $Y$ belong to some class of rearrangement invariant spaces, like, for example,
Orlicz spaces, the description of the space $M(X,Y)$ in the language of $X$ and $Y$ is a difficult task
with only piecemeal solutions scattered throughout the existing literature.
Actually, the example of Orlicz spaces is an exceptionally appropriate example here.
After many partial results and almost half a century of struggle, this problem has only recently been
put to rest by Djakov and Ramanujan in \cite{DR00}, and Le{\' s}nik and the second-named author in \cite{LT17}
(we refer to \cite{LT17} for a more detailed discussion).

That said, from this point on, we will focus on some important classes of rearrangement invariant sequence
spaces like, for example, Orlicz sequence spaces, Lorentz sequence spaces and Marcinkiewicz sequence spaces.
After the discussion on Musielak--Orlicz sequence spaces, our goal is to prove Theorem~\ref{Prop: computing multipliers from m and to lambda}.
{\it En route} to this result we will show a series of lemmas (see Lemma~\ref{LEMMA: M(X,Y) is r.i.}, Lemma~\ref{LEMMA: a'la KLM12/14}
and Lemma~\ref{LEMMA: a'la Kolwicz OC}) that are of independent interest.

\subsection{A few thoughts about Musielak--Orlicz sequence spaces} \label{SUBSEC: Musielak-Orlicz sequence}
	Let $\{M_n\}_{n=1}^{\infty}$ be a sequence of Orlicz functions\footnote{By an {\bf Orlicz function $M \colon [0,\infty) \rightarrow [0,\infty)$}
	we understand a convex and strictly increasing function vanishing at zero. We decided to use only the Orlicz functions here instead of the more
	general Young functions (which can \enquote{jump} to infinity) to avoid some unnecessary complications.}.
	Recall, that the {\bf Musielak--Orlicz sequence space} $\ell_{\{M_n\}}$ is a vector space of all sequences $x = \{x_n\}_{n=1}^{\infty}$
	with $\sum_{n=1}^{\infty} M_n(\abs{x_n}/\varrho) < \infty$ for some $\varrho > 0$, equipped with the {\bf Luxemburg--Nakano norm}
	\begin{equation*}
		\norm{x}_{\ell_{\{M_n\}}} \coloneqq \inf \left\{ \varrho > 0 \colon \sum_{n=1}^{\infty} M_n\left( \frac{x_n}{\varrho} \right) \leqslant 1 \right\}.
	\end{equation*}
	
		This general construction includes two very important types of spaces.
		On the one hand, the {\bf Nakano sequence spaces $\ell_{\{p_n\}}$} (or {\bf variable exponent Lebesgue--Riesz sequence spaces}),
		which are nothing else but Musielak--Orlicz sequence spaces $\ell_{\{M_n\}}$ with all generating functions $M_n$ being just power functions,
		that is, $M_n(t) = t^{p_n}$ for some $p_n \geqslant 1$.
		On the other, the {\bf Orlicz sequence spaces $\ell_M$} are Musielak--Orlicz sequence spaces $\ell_{\{M_n\}}$ with \enquote{constant}
		sequence of Orlicz functions $\{M_n\}_{n=1}^{\infty}$, that is, with $M_n = M$ for all $n \in \mathbb{N}$.

	We refer to the book of Cruz-Uribe and Fiorenza \cite{CUF13}, the monograph of Lindenstrauss and Tzafriri \cite[Chapter~4]{LT77}
	and the book of Maligranda \cite{Ma89} for much more information about variable Lebesgue spaces, Musielak--Orlicz sequence spaces and Orlicz spaces.
	
	\begin{remark}[About compact multipliers between Musielak--Orlicz sequence spaces] \label{REMARK: compact multipliers MO}
		Here we want to roughly explain the general technology needed to describe compact multipliers
		acting between $\ell_{\{M_n\}}$ and $\ell_{\{N_n\}}$.
		
		It follows from \cite{LT21} that the space of pointwise multipliers between two Musielak--Orlicz
		sequence spaces is another Musielak--Orlicz sequence space.
		In our setting, using Theorem~9 from \cite{LT21}, we have
		\begin{equation*}
			M\left( \ell_{\{M_n\}}, \ell_{\{N_n\}} \right) = \ell_{\{N_n \ominus M_n\}},
		\end{equation*}
		where
		\begin{equation*}
			(N_n \ominus M_n)(t) \coloneqq \sup\limits_{0 \leqslant s \leqslant 1} [N_n(ts) - M_n(s)] \quad \text{ for } \quad n \in \mathbb{N}
		\end{equation*}
		is the {\bf generalized Young conjugate} of $M_n$ with respect to $N_n$.
		
		Next, Kolwicz and Panfil in \cite{KP12} gave a complete description of a separable part of the Musielak--Orlicz space
		(in fact, they allowed for an even more general construction).
		Precisely, it can be deduced from Theorem~11 in \cite{KP12} that a sequence $x = \{x_n\}_{n=1}^{\infty}$
		has an order continuous norm in $\ell_{\{M_n\}}$ if, and only if, the sequence of Orlicz functions $\{M_n\}_{n=1}^{\infty}$
		satisfies the {\bf $\Delta_2(x)$-condition} (see \cite[Definition~1]{KP12}).
		
		The rest is essentially a matter of collecting existing results. Thanks to Theorem~\ref{THM: Essential norm Kothe sequence},
		the multiplication operator $M_{\lambda} \colon \ell_{\{M_n\}} \rightarrow \ell_{\{N_n\}}$ is compact if, and only if,
		the sequence $\{ N_n \ominus M_n \}_{n=1}^{\infty}$ satisfies the $\Delta_2(\lambda)$-condition.
		\demo
	\end{remark}

 	Even though the necessary calculations from Remark~\ref{REMARK: compact multipliers MO} can be done quite easily in a concrete situation,
 	the entire method can hardly be called effective (let alone the case when we allow the more general Young functions in place of Orlicz functions).
 	Anyway, sacrificing a little generality on the altar of clarity, we have the following
	
	\begin{theorem}[Compact multipliers between Nakano sequence spaces] \label{PROPOSITION: Compact multipliers between Nakano}
		{\it Let $\{p_n\}_{n=1}^{\infty}$ and $\{q_n\}_{n=1}^{\infty}$ be two sequences of real numbers with $1 \leqslant p_n \leqslant q_n < \infty$
			for all $n \in \mathbb{N}$. Suppose that either
			\begin{equation*}
				1 < \inf\limits_{n \in \mathbb{N}} p_n \leqslant \sup\limits_{n \in \mathbb{N}} p_n < \infty
			\end{equation*}
			or
			\begin{equation*}
				\sup\limits_{n \in \mathbb{N}} q_n < \infty.
			\end{equation*}
			Then the multiplication operator $M_{\lambda} \colon \ell_{\{q_n\}} \rightarrow \ell_{\{p_n\}}$ is compact if, and only if,}
		\begin{equation*}
			\forall \ell > 1 \colon \lim\limits_{k \rightarrow \infty} \sum_{\left\{ n \in \supp(\lambda) \colon r_n > \frac{\log k}{\log \ell} \right\}}
				\left( \abs{\lambda_n} \ell \right)^{r_n} = 0,
		\end{equation*}
		{\it where $1/r_n = 1/q_n - 1/p_n$ for $n \in \mathbb{N}$.}
	\end{theorem}
	\begin{proof}
		To make things clear, let us divide the proof into a few short steps.
		
		{\bf First step.} It is part of folklore that the Nakano space $\ell_{\{s_n\}}$ is order continuous if, and only if, $\sup_{n \in \mathbb{N}} s_n < \infty$.
		In consequence, using duality and \cite[Corollary~4.4, p.~23]{BS88}, the space $\ell_{\{s_n\}}$ is reflexive if, and only if, $\inf_{n \in \mathbb{N}} s_n > 1$
		and $\sup_{n \in \mathbb{N}} s_n < \infty$.
		
		{\bf Second step.} Suppose that $1 \leqslant p_n \leqslant q_n < \infty$ for all $n \in \mathbb{N}$.
		Let $1/r_n = 1/q_n - 1/p_n$ for $n \in \mathbb{N}$. Then, it follows from \cite[Theorem~9]{LT21}
		(see also Corollary~10 in \cite{LT21}) that
		\begin{equation*}
			M\left( \ell_{\{p_n\}}, \ell_{\{q_n\}} \right) = \ell_{\{q_n \ominus p_n\}} = \ell_{\{r_n\}}.
		\end{equation*}
		
		{\bf Third step.} Next, due to \cite[Theorem~11]{KP12}, the sequence $x = \{x_n\}_{n=1}^{\infty}$ belongs to the separable part of the
		Musielak--Orlicz sequence space $\ell_{\{M_n\}}$ if, and only if, for each $\ell > 1$, we have
		\begin{equation*}
			\lim_{k \rightarrow \infty} \norm{  \sum_{n \in A_k^{\ell}} M_n\left( \ell x_n \right) e_n }_{\ell_1} = 0.
		\end{equation*}
		Here, $A_k^{\ell} \coloneqq \left\{ n \in \supp(x) \colon M_n(\ell x_n) > k M_n(x_n) \right\}$.
		However, in our situation, we need to apply the above criterion to the Nakano sequence space $\ell_{\{r_n\}}$.
		Thus, after some simple calculations, it is easy to see that the sequence $x = \{x_n\}_{n=1}^{\infty}$
		belongs to the ideal $(\ell_{\{r_n\}})_o$ if, and only if, 
		\begin{equation*}
			\lim_{k \rightarrow \infty} \norm{ \sum_{n \in A_k^{\ell}} \left( \ell x_n \right)^{r_n} e_n }_{\ell_1} = 0.
		\end{equation*}
		Here,
		\begin{equation*}
			A_k^{\ell} = \left\{ n \in \supp(x) \colon r_n > \frac{\log k}{\log \ell} \right\}.
		\end{equation*}
	
		{\bf Fourth step.} Now, to complete the proof, it is enough to use the previous steps and Theorem~\ref{THM: Essential norm Kothe sequence}.
	\end{proof}

\subsection{Lorentz and Marcinkiewicz's setting}
	Let $\varphi$ be a positive, increasing and concave function on $(0,\infty)$.
	By the {\bf Lorentz sequence space $\lambda_{\varphi}$} and the {\bf Marcinkiewicz sequence space $m_{\varphi}$},
	we understand the spaces of all sequences $x = \{x_n\}_{n=1}^{\infty}$ for which the following norms
	\begin{equation*}
		\norm{x}_{\lambda_{\varphi}} \coloneqq \sum_{n=1}^{\infty} x_n^{\star} (\varphi(n+1) - \varphi(n))
	\end{equation*}
	and, respectively,
	\begin{equation*}
		\norm{x}_{m_{\varphi}} \coloneqq \sup\limits_{n \in \mathbb{N}} \frac{\varphi(n)}{n} \sum_{k=1}^n x_k^{\star}
	\end{equation*}
	are finite. Here, $x^{\star} = \{x_n^{\star}\}_{n=1}^{\infty}$ is the decreasing rearrangement of the sequence $x = \{x_n\}_{n=1}^{\infty}$.
	In the case when $\varphi(n) = n^{1/p}$ with $1 < p < \infty$, the corresponding Marcinkiewicz sequence space $m_{\varphi}$
	is denoted by $\ell_{p,\infty}$ and is sometimes called the {\bf weak $\ell_p$-space}.
	
	\begin{remark}[Separable part of $\lambda_{\varphi}$ and $m_{\varphi}$] \label{REMARK: separable part lorentz marcinkiewicz}
		It is well-known that the Lorentz sequence space $\lambda_{\varphi}$ is separable if, and only if, $\varphi(\infty) = \infty$
		(see, for example, \cite[Lemma~5.1, p.~110]{KPS82}; cf. \cite[Corollary~4.20]{Kol16} together with the note on p.~820 about the sequence case).
		Note also that if $\varphi(\infty) < \infty$, then $\lambda_{\varphi} = \ell_{\infty}$. Thus, in this case, $(\lambda_{\varphi})_o = c_0$.
		On the other hand, the Marcinkiewicz sequence spaces $m_{\varphi}$ are never separable (see, for example, \cite[Remark~4.4(ii)]{Kol16}).
		However, the separable part $(m_{\varphi})_o$ of $m_{\varphi}$ can be described precisely in the following way
		\begin{equation*}
			(m_{\varphi})_o = \left\{ x = \{x_n\}_{n=1}^{\infty} \in m_{\varphi} \colon \lim\limits_{n \rightarrow \infty} \frac{\varphi(n)}{n} \sum_{k=1}^n x_k^{\star} = 0 \right\}
		\end{equation*}
		(see \cite[pp.~115--116]{KPS82}). Similarly, if $\varphi(\infty) < \infty$, then $m_{\varphi} = \ell_{\infty}$ and $(m_{\varphi})_o = c_0$.
		\demo
	\end{remark}

	After this short reminder, the main result of this section is as follows.
	
	\begin{theorem}[Compact multipliers between Lorentz and Marcinkiewicz sequence spaces] \label{Prop: computing multipliers from m and to lambda}
		{\it Let $\varphi$ and $\psi$ be two positive, increasing and concave functions on $(0,\infty)$.
			Further, let $X$ and $Y$  be two rearrangement invariant sequence spaces with the Fatou property.}
		\begin{enumerate}
			\item {\it Suppose that $\delta_{\varphi} < \infty$. Then the multiplier
				$M_{\lambda} \colon X \rightarrow \lambda_{\varphi}$ is compact if, and only if,
				either $\left\{ \frac{\varphi(n)}{n}\lambda_n^{\star} \right\}_{n=1}^{\infty} \in \left(X^{\times}\right)_o$ provided
				$\left\{ \varphi(n) / n \right\}_{n=1}^{\infty} \notin X^{\times}$, or $\{\lambda_n\}_{n=1}^{\infty} \in c_0$
				provided $\left\{ \varphi(n) / n \right\}_{n=1}^{\infty} \in \left( X^{\times} \right)_o$.}
			\item {\it Suppose that $\gamma_{\varphi} > 1$. Then the multiplier
				$M_{\lambda} \colon m_{\varphi} \rightarrow Y$ is compact if, and only if, either
				$\left\{ \lambda_n^{\star} / \varphi(n) \right\}_{n=1}^{\infty} \in Y_o$, provided
				$\left\{ 1 / \varphi(n) \right\}_{n=1}^{\infty} \notin Y$, or $\{\lambda_n\}_{n=1}^{\infty} \in c_0$,
				provided $\left\{ 1 / \varphi(n) \right\}_{n=1}^{\infty} \notin Y$.}
			
		\end{enumerate}
		{\it In particular,}
		\begin{enumerate}
			\item[(iii)] {\it if $\limsup_{n \rightarrow \infty} \psi(n) / \varphi(n) = \infty$
				and $\gamma_{\varphi} > 1$, then the multiplier $M_{\lambda} \colon m_{\varphi} \rightarrow \lambda_{\psi}$
				is compact if, and only if, $\left\{ \frac{\psi(n+1) - \psi(n)}{\varphi(n)} \lambda_n^{\star} \right\}_{n=1}^{\infty} \in \ell_1$,}
			\item[(iv)] {\it if $\limsup_{n \rightarrow \infty} \varphi(n) / \psi(n) = \infty$ and $\delta_{\varphi} < \infty$,
				then the multiplier $M_{\lambda} \colon \lambda_{\psi} \rightarrow \lambda_{\varphi}$
				is compact if, and only if, $\left\{ \frac{\varphi(n)}{\psi(n)} \lambda_n^{\star} \right\}_{n=1}^{\infty} \in c_0$.}
		\end{enumerate}
	\end{theorem}

	Before we give the proof of Theorem~\ref{Prop: computing multipliers from m and to lambda},
	we need some preparatory lemmas that seem interesting in themselves.
	
	\begin{lemma} \label{LEMMA: M(X,Y) is r.i.}
		{\it Let $X$ and $Y$ be two rearrangement invariant sequence spaces with the Fatou property.
		Then, up to the equivalence of norms, the space $M(X,Y)$ is rearrangement invariant as well.}
	\end{lemma}
	\begin{proof}
		To see this, take the function $x = \{x_n\}_{n=1}^{\infty}$ from $M(X,Y)$. For convenience, let us consider two cases separately.
		
		{\bf First case.} Suppose that $x^{\star}(\infty) > 0$.
		This clearly implies that $\ell_{\infty} \hookrightarrow M(X,Y)$. Moreover, if $F \hookrightarrow G$, then it is straightforward to
		see that $M(E,F) \hookrightarrow M(E,G)$. Note also that $M(E,\ell_{\infty}) \hookrightarrow \ell_{\infty}$. Indeed, if an unbounded
		function, say $z = \{z_n\}_{n=1}^{\infty}$, belonged to $M(E,\ell_{\infty})$, then $z_n = z e_n \in \ell_{\infty}$ for all $n \in \mathbb{N}$,
		which is impossible. Thus, combining all these three facts together and remembering that $Y \hookrightarrow \ell_{\infty}$
		(see \cite[Corollary~6.7, p.~78]{BS88}), we get
		\begin{equation*}
			\ell_{\infty} \hookrightarrow M(X,Y) \hookrightarrow M(X,\ell_{\infty}) \hookrightarrow \ell_{\infty}.
		\end{equation*}
		In consequence, $M(X,Y) = \ell_{\infty}$.
		
		{\bf Second case.} Suppose that $x^{\star}(\infty) = 0$. Here we will, more or less, follow the argument given in \cite[Theorem~2.2]{KLM12}.
		Let $y = \{y_n\}_{n=1}^{\infty}$ be any function that is equi-measurable with $x$. Then, due to \cite[Lemma~2.1, p.~60]{KPS82},
		for any $\varepsilon > 0$ there exists a measure-preserving transformation $T_{\varepsilon}$ such that
		\begin{equation*}
			\norm{x - y \circ T_{\varepsilon} }_{\ell_1} < \varepsilon.
		\end{equation*}
		Thus, since $\ell_1 \overset{c}{\hookrightarrow} Y$ with $c = \varphi_Y(1)$ and $X \overset{C}{\hookrightarrow} \ell_{\infty}$ with $C = 1/\varphi_X(1)$
		(see \cite[pp.~77--78]{BS88}), so for all functions $z = \{z_n\}_{n=1}^{\infty}$ from $\text{Ball}_X$, we have
		\begin{align*}
			\norm{yz}_Y
				& = \norm{(y \circ T_{\varepsilon})(z \circ T_{\varepsilon})}_Y \\
				& \leqslant \norm{x(z \circ T_{\varepsilon})}_Y + \norm{\left( y \circ T_{\varepsilon} - x \right)(z \circ T_{\varepsilon})}_Y \\
				& \leqslant \norm{x}_{M(X,Y)} + \norm{\left( y \circ T_{\varepsilon} - x \right)(z \circ T_{\varepsilon})}_F \\
				& \leqslant \norm{x}_{M(X,Y)} + \norm{z \circ T_{\varepsilon}}_{\ell_{\infty}} \norm{y \circ T_{\varepsilon} - x}_Y \\
				& \leqslant \norm{x}_{M(X,Y)} + \frac{\varphi_Y(1)}{\varphi_X(1)}\norm{z}_{X} \norm{y \circ T_{\varepsilon} - x}_{\ell_1} \\
				& < \norm{x}_{M(X,Y)} + \frac{\varphi_Y(1)}{\varphi_X(1)} \varepsilon.
		\end{align*}
		Due to $\varepsilon$'s arbitrariness we conclude that
		\begin{equation*}
			\norm{y}_{M(X,Y)} \leqslant \norm{x}_{M(X,Y)}.
		\end{equation*}
		To prove the reverse inequality, it is enough to swap the roles of $x$ and $y$.
	\end{proof}
		
		\begin{lemma} \label{LEMMA: a'la KLM12/14}
			{\it Let $X$ and $Y$ be two rearrangement invariant sequence spaces with the Fatou property.
				Further, let $\varphi$ be a positive, increasing and concave function on $(0,\infty)$. Then}
			\begin{enumerate}
				\item {\it $M(m_{\varphi},Y) = \left[ Y\left( 1 / \varphi \right) \right]^{\bigstar}$ provided $\gamma_{\varphi} > 1$,}
				\item {\it $M(X,\lambda_{\varphi}) = \left[ X^{\times}\left( \varphi(t) / t \right) \right]^{\bigstar}$ provided $\delta_{\varphi} < \infty$.}
			\end{enumerate}
		\end{lemma}
		\begin{proof}
			{\bf (i)}. Without loss of generality, we can assume that $m_{\varphi} \neq \ell_{\infty}$ or, which is the same thing,
			that the function $\varphi$ is unbounded. Otherwise, the sequence $\{\varphi(n)\}_{n=1}^{\infty}$ is equivalent to a constant
			sequence and we have the following identifications
			\begin{equation*}
				M(m_{\varphi},Y) = M(\ell_{\infty},Y) = Y = Y\left( \frac{1}{\varphi} \right) = \left[ Y\left( \frac{1}{\varphi} \right) \right]^{\bigstar}.
			\end{equation*}
			We will now consider two cases.
			
			{\bf First case.} Suppose that $\left\{ 1 / \varphi(n) \right\}_{n=1}^{\infty} \in Y$.
			Then, it is clear, that $\ell_{\infty} \hookrightarrow Y\left( 1 /\varphi \right)$.
			Moreover, for all $n \in \mathbb{N}$, we have
			\begin{equation*}
				 \frac{\varphi_Y(n)}{\varphi(n)}x_n^{\star}
				 	= \varphi_Y(n) \left[ \frac{x_n^{\star}}{\varphi(n)} \right]^{\star}
					\leqslant \frac{\varphi_Y(n)}{n} \sum_{k=1}^{n} \frac{x_k^{\star}}{\varphi(k)},
			\end{equation*}
			where the first equality follows from a simple fact that the product of two decreasing functions is decreasing,
			and the second inequality is due to \cite[Proposition~3.2, p.~52]{BS88}.
			Thus, in view of the embedding $Y \hookrightarrow m_{\varphi_Y}$ (see \cite[Proposition~5.1, p.~70]{BS88}), we have
			\begin{equation*}
				\ell_{\infty} \equiv \ell_{\infty}^{\bigstar} \hookrightarrow \left[ Y\left( \frac{1}{\varphi} \right) \right]^{\bigstar}
					\hookrightarrow \left[ m_{\varphi_Y}\left( \frac{1}{\varphi} \right) \right]^{\bigstar}
					\hookrightarrow \left[ \ell_{\infty}\left( \frac{\varphi_Y}{\varphi} \right) \right]^{\bigstar}.
			\end{equation*}
			Now, since $\sup_{n \in \mathbb{N}} \varphi_Y(n) / \varphi(n) < \infty$ (otherwise, the sequence $\left\{ 1 / \varphi(n) \right\}_{n=1}^{\infty}$
			does not belongs to $\left[ \ell_{\infty}(\varphi_Y) \right]^{\bigstar}$, which is impossible), so
			\begin{equation*}
				\left[ \ell_{\infty}\left( \frac{\varphi_Y}{\varphi} \right) \right]^{\bigstar} \hookrightarrow [\ell_{\infty}]^{\bigstar} \equiv \ell_{\infty}.
			\end{equation*}
			This means that
			\begin{equation} \label{EQ: Y(w)^s = l infty}
				\left[ Y\left( \frac{1}{\varphi} \right) \right]^{ \bigstar } = \ell_{\infty}.
			\end{equation}
			On the other hand, since $\left\{ 1 / \varphi(n) \right\}_{n=1}^{\infty} \in Y$, so $\ell_{\infty}(\varphi) \hookrightarrow Y$.
			Thus, remembering that $\gamma_{\varphi} > 1$ and using \cite[Theorem~5.3, p.~115]{KPS82}, we have
			\begin{equation*}
				m_{\varphi} = \left[ \ell_{\infty}(\varphi) \right]^{\bigstar} \hookrightarrow Y^{\bigstar} \equiv Y.
			\end{equation*}
			This imply that $\ell_{\infty} \hookrightarrow M(m_{\varphi},Y)$.
			Moreover, due to Lemma~\ref{LEMMA: M(X,Y) is r.i.}, the space $M(m_{\varphi},Y)$ is rearrangement invariant,
			so $M(m_{\varphi},Y) \hookrightarrow \ell_{\infty}$. In consequence,
			\begin{equation} \label{EQ M = linfty}
				M(m_{\varphi},Y) = \ell_{\infty}.
			\end{equation}
			Putting (\ref{EQ: Y(w)^s = l infty}) together with \eqref{EQ M = linfty}, we get
			\begin{equation*}
				M(m_{\varphi},Y) = \left[ Y\left( \frac{1}{\varphi} \right) \right]^{\bigstar}.
			\end{equation*}
			This is what we wanted.
			
			{\bf Second case.} This time, let us assume that $\left\{ 1 / \varphi(n) \right\}_{n=1}^{\infty} \notin Y$.
			Again, our plan is to show that
			\begin{equation} \label{EQ: idea KLM14}
				M(m_{\varphi},Y) = \left[ Y\left( \frac{1}{\varphi} \right) \right]^{\bigstar}.
			\end{equation}
			To do this, we will need the following result whose proof we essentially borrow from \cite[p.~653]{KLM14}.
			
			\vspace{5pt} \noindent {\bf Claim~$\blacktriangle$.}
			{\it Let $F$ be a rearrangement invariant sequence space.
			Further, let $\varphi$ be a positive, increasing and concave function on $(0,\infty)$ with $\gamma_{\varphi} > 1$.
			Then, we have}
			\begin{equation*} \label{EQ: equality}
				\norm{x}_{M(m_{\varphi},F)} = \sup\limits_{T} \norm{\frac{x}{\varphi \circ T}}_F,
			\end{equation*}
			{\it where the supremum is taken over all measure-preserving transformations $T$.}
			
			\vspace{5pt}
			\noindent
			{\bf Proof of Claim~$\blacktriangle$}
			Take $x \in M(m_{\varphi},F)$ and let $T$ be a measure preserving transformation.
			Since, $\left( 1 / (\varphi \circ T) \right)^{\star} = 1 / \varphi \in \ell_{\infty}(\varphi)$
			and $m_{\varphi} = [\ell_{\infty}(\varphi)]^{\bigstar}$,
			so $1 / (\varphi \circ T) \in m_{\varphi}$.
			In consequence, $x / (\varphi \circ T) \in F$ and
			\begin{equation*}
				\norm{\frac{x}{\varphi \circ T}}_F \leqslant \norm{x}_{M(m_{\varphi},F)}.
			\end{equation*}
			Taking the supremum over all measure-preserving transformations $T$, we have
			\begin{equation*}
				\sup\limits_{T} \norm{\frac{x}{\varphi \circ T}}_F \leqslant \norm{x}_{M(m_{\varphi},F)}.
			\end{equation*}
		
			To prove the reverse inequality, suppose that $x / (\varphi \circ T) \in Y$ for any measure-preserving transformation $T$.
			Take $z \in \text{Ball}_{m_{\varphi}}$. As we already explained above, we can assume that $m_{\varphi} \neq \ell_{\infty}$.
			Thus, $m_{\varphi} \hookrightarrow c_0$ and, consequently, $z^*(\infty) = 0$.
			Moreover, Ryff's theorem guarantee the existence of a measure-preserving transformation, say $T_0$,
			such that $\abs{z} = z^{\star} \circ T_0$ (see \cite[Theorem~7.5, p.~82]{BS88}).
			Putting these facts together, we get
			\begin{equation*}
				\abs{xz} = \abs{x}(z^{\star} \circ T_0) \leqslant \frac{\abs{x}}{\varphi \circ T_0} \in F.
			\end{equation*}
			This means that $x \in M(m_{\varphi},Y)$ and, moreover,
			\begin{equation*}
				\norm{x}_{M(m_{\varphi},F)} \leqslant \sup\limits_{T} \norm{\frac{x}{\varphi \circ T}}_F.
			\end{equation*}
			The proof is finished.
			\hfill $\blacksquare$
			\vspace{5pt} \noindent
			
			Now we are ready to show (\ref{EQ: idea KLM14}).
			
			{\bf Embedding $M(m_{\varphi},Y) \hookrightarrow \left[ Y\left( 1 / \varphi \right) \right]^{\bigstar}$.}
			Take $x \in M(m_{\varphi},Y)$. In view of Claim~$\blacktriangle$, it is enough to show that
			\begin{equation*}
				\norm{\frac{x^{\star}}{\varphi}}_{Y}
					= \norm{\frac{x^{\star} \circ T}{\varphi \circ T}}_Y
					\leqslant \sup\limits_{T} \norm{\frac{x}{\varphi \circ T}}_Y.
			\end{equation*}
			However, once we know that there is a measure-preserving transformation, say $T_0$, such that $\abs{x} = x^{\star} \circ T_0$,
			the above inequality is basically obvious. So our goal for now is to show that $ x^{\star}(\infty) = 0$. Then, a simple application
			of Ryff's theorem will complete the argument. To do this, recall that we assumed that $\left\{ 1 / \varphi(n) \right\}_{n=1}^{\infty} \notin Y$.
			Thus
			\begin{equation*}
				\ell_{\infty}(\varphi) \not\hookrightarrow Y
			\end{equation*}
			and, consequently,
			\begin{equation} \label{EQ: 876}
				m_{\varphi} = \left[ \ell_{\infty}(\varphi) \right]^{\bigstar} \not\hookrightarrow Y^{\bigstar} = Y.
			\end{equation}
			We claim that
			\begin{equation} \label{EQ: 94949}
				\ell_{\infty} \not\hookrightarrow M(m_{\varphi}, Y).
			\end{equation}
			Indeed, otherwise, due to the fact that the space $M(m_{\varphi}, Y)$ is rearrangement invariant (see Lemma~\ref{LEMMA: M(X,Y) is r.i.}),
			we would have that $M(m_{\varphi}, Y) = \ell_{\infty}$. This in turn, however, would mean that $m_{\varphi} \hookrightarrow Y$.
			This clearly contradicts \eqref{EQ: 876}, so \eqref{EQ: 94949} follows.
			Thus, using once again the fact that the space $M(m_{\varphi}, Y)$ is rearrangement invariant, we have
			\begin{equation*}
				M(m_{\varphi}, Y) \hookrightarrow c_0.
			\end{equation*}
			This is exactly what we needed.
			
			{\bf Embedding $\left[ Y\left( 1 / \varphi \right) \right]^{\bigstar} \hookrightarrow M(m_{\varphi},Y)$.}
			Take $x \in \left[ Y\left( 1 / \varphi \right) \right]^{\bigstar}$ and let $T$ be a measure-preserving transformation.
			Applying \cite[Property~$10^{\circ}$, p.~67]{KPS82} and remembering that the function $t \rightsquigarrow 1 / \varphi(t)$
			is decreasing, we have
			\begin{align*}
				\left[ \frac{x}{\varphi \circ T} \right]^{\star}
					& \leqslant D_2 x^{\star} \, D_2 \left[ \frac{1}{\varphi \circ T} \right]^{\star}  \\
					& = D_2 x^{\star} \, D_2\left[ \frac{1}{\varphi} \right] \\
					& = D_2\left[ \frac{x^*}{\varphi} \right].
			\end{align*}
			Moreover, since the space $Y$ is rearrangement invariant, so
			\begin{align*}
				\norm{\frac{x}{\varphi \circ T}}_Y 
					& = \norm{ \frac{x}{\varphi \circ T} }_{Y^{\bigstar}} \\
					& \leqslant \norm{D_2}_{Y \rightarrow Y} \norm{\frac{x^{\star}}{\varphi}}_Y \\
					& \leqslant 2 \norm{x}_{\left[ Y\left( 1 / \varphi \right) \right]^{\bigstar}},
			\end{align*}
			where the last inequality is essentially due to the Calder{\' o}n--Mityagin interpolation theorem (see \cite[p.~130]{LT79}; cf. \cite[pp.~96--98]{KPS82}).
			Taking the supremum over all measure-preserving mappings ends the proof of (\ref{EQ: idea KLM14}).
			
			{\bf (ii)} This follows by duality. Indeed, remembering about our assumptions, we have
			\begin{align*}
				M\left( X, \lambda_{\varphi} \right) & = M\left( (\lambda_{\varphi})^{\times}, X^{\times} \right) \quad \text{(using \cite[Property~(vii), p.~879]{KLM12})} \\
					& = M\left( m_{t/\varphi(t)}, X^{\times} \right) \quad \text{(using \cite[Theorem~5.2, p.~112]{KPS82})} \\
					& = \left[ X^{\times}\left( \frac{\varphi(t)}{t} \right) \right]^{\bigstar} \quad \text{(by (i) proved above)}.
			\end{align*}
		\end{proof}
	
		Next, we need a lemma in the spirit of Kolwicz's results from \cite{Kol16} (see, for example, Theorem~3.9 ibidem).
	
		\begin{lemma} \label{LEMMA: a'la Kolwicz OC}
			{\it Let $X$ be a rearrangement invariant sequence space with the Fatou property.
			Further, let $w$ be a positive and decreasing function on $(0,\infty)$ such that $w \notin X$.
			Suppose that up to an equivalent norm the space $[X(w)]^{\bigstar}$ is a Banach space. Then}
			\begin{equation*}
				[ X(w)^{\bigstar} ]_o = \left[ X_o(w) \right]^{\bigstar}.
			\end{equation*}
		\end{lemma}
		\begin{proof}
			In order not to overload the notation, we will omit some parentheses here and there. This should not lead to misunderstandings.
			Moreover, it is more or less obvious that $X(w)_o = X_o(w)$, so we only need to show that
			\begin{equation}
				[ X(w)_o ]^{\bigstar} = [ X(w)^{\bigstar} ]_o.
			\end{equation}
			We will justify both inclusions separately.
			
			{\bf First inclusion}. Fix $\varepsilon > 0$. To show that
			\begin{equation} \label{OCgwiazdkazawieranie 1}
				[ X(w)_o ]^{\bigstar} \subset [ X(w)^{\bigstar} ]_o,
			\end{equation}
			take the sequence $x = \{x_n\}_{n=1}^{\infty}$ from $[ X(w)_o ]^{\bigstar}$.
			This means that the sequence $x^{\star} = \{x^{\star}_n\}_{n=1}^{\infty}$ belongs to $X(w)_o$.
			Since the space $X(w)^{\bigstar}$ is rearrangement invariant, so in order to prove the inclusion \eqref{OCgwiazdkazawieranie 1},
			it is enough to show that $x^{\star} \in [ X(w)^{\bigstar} ]_o$ (see \cite[Theorem~3.6]{KKT22}).
			Clearly,
			\begin{equation*}
				\lim_{n \rightarrow \infty} \norm{ \sum_{i=n}^{\infty} x^{\star}_i e_i }_{X(w)} = 0,
			\end{equation*}
			so there is $N\in\mathbb{N}$ with
			\begin{equation*}
				\norm{ \sum_{i=n}^{\infty} x^{\star}_i e_i }_{X(w)} < \varepsilon \quad \text{ for } \quad n \geqslant N.
			\end{equation*}
			Next, let us note that $x^{\star}(\infty) = 0$. In fact, if this were not the case, then $\ell_\infty \hookrightarrow X(w)$.
			However, since the function $w$ is decreasing, this would mean that $w \in X$, which is impossible.
			Therefore, we can find $M \in \mathbb{N}$ such that
			\begin{equation*}
				x^{\star}_{m} \leqslant \frac{\varepsilon}{\norm{ \sum_{i=1}^N e_i }_{X(w)}} \quad \text{ for } \quad m \geqslant M.
			\end{equation*}
			In consequence, for $n \geqslant \max\{M,N\}$, we have
			\begin{align*}
				\norm{ \sum_{i=n}^{\infty} x^{\star}_i e_i }_{X(w)^{\bigstar}}
					& \leqslant \norm{\left( \sum_{i=M}^{\infty} x^{\star}_i e_i \right)^{\star} }_{X(w)} \\
					& \leqslant\norm{ \sum_{i = 1}^N x^{\star}_{M} e_i }_{X(w)} + \norm{ \sum_{i = N+1}^{\infty} x^{\star}_i e_i }_{X(w)}
					< 2\varepsilon.
			\end{align*}
			Thus $x^{\star} \in [ X(w)^{\bigstar} ]_o$ and this completes the proof of \eqref{OCgwiazdkazawieranie 1}.
			
			{\bf Second inclusion}. Take $x = \{x_n\}_{n=1}^{\infty}$ from $[ X(w)^{\bigstar} ]_o$. We will show that
			\begin{equation} \label{OCgwiazdkazawieranie 2}
				[ X(w)^{\bigstar} ]_o \subset [ X(w)_o ]^{\bigstar}.
			\end{equation}
			Once again, since the space $X(w)^{\bigstar}$ is rearrangement invariant, so in view of \cite[Theorem~3.6]{KKT22}
			the fact that $x \in [ X(w)^{\bigstar} ]_o$ is equivalent to $x^* \in [ X(w)^{\bigstar} ]_o$.
			Thus, in order to prove \eqref{OCgwiazdkazawieranie 2}, we need to show that $x^*\in X(w)_o$, that is,
			that for any sequence $\left\{ x^{(n)} \right\}_{n=1}^{\infty} \subset X(w)$ with $0 \leqslant x^{(n)} \leqslant x^{\star}$
			and $x^{(n)} \downarrow 0$ we have $\lim_{n \rightarrow \infty} \norm{x^{(n)}}_{X(w)} = 0$. Take such a sequence $\left\{ x^{(n)} \right\}_{n=1}^{\infty}$.
			Since $[X(w)^{\bigstar}]_o \subset [\ell_{\infty}]_o = c_0$, so $x^{\star}(\infty) = 0$.
			Thus $\left( x^{(n)} \right)^{\star} \downarrow 0$ and, trivially, $0 \leqslant \left( x^{(n)} \right)^{\star} \leqslant x^{\star}$.
			This means that $\lim_{n \rightarrow \infty} \norm{\left( x^{(n)} \right)^{\star}}_{X(w)} = 0$.
			In consequence, we get
			\begin{align*}
				\norm{x^{(n)}}_{X(w)} & = \norm{ ( x^{(n)} w )^{\star}}_{X} \quad \text{(because $X$ is rearrangement invariant)} \\
					& \leqslant \norm{D_2 ( x^{(n)} )^{\star} \, D_2 w^{\star} }_{X} \quad \text{(using \cite[Property~$10^{\circ}$, p.~67]{KPS82})} \\
					& = \norm{D_2 \left[ ( x^{(n)} )^{\star} w^{\star} \right]}_{X} \\
					& \leqslant \norm{D_2}_{X \rightarrow X} \norm{( x^{(n)} )^{\star} w^{\star}}_{X} \quad \text{(because $D_2$ is bounded on $X$)} \\
					& = \norm{D_2}_{X \rightarrow X} \norm{( x^{(n)} )^{\star} w}_{X} \quad \text{(since the function $w$ is decreasing)} \\
					& = \norm{D_2}_{X \rightarrow X} \norm{x^{(n)}}_{[X(w)]^{\bigstar}} \to 0
			\end{align*}
			as $n$ goes to infinity. The proof is finished.
		\end{proof}
	
		\begin{scholium}[About~Lemma~\ref{LEMMA: a'la Kolwicz OC}]
			What seems quite surprising is that our proof is rather elementary, in the sense, that it use neither the concept of the so-called $OC^{\bigstar}$-point
			or the machinery developed in \cite{Kol16}. Actually, the first-named author learned from Kolwicz another, slightly more general,
			proof which directly uses some ideas from \cite{Kol16}. However, let us put it in the drawer for now.
			\demo
		\end{scholium}
		
		\begin{remark}
			As we mentioned earlier, the symmetrization of a K{\" o}the space is in general only a quasi-Banach space (see \cite[Lemma~1.4]{KR09}; cf. \cite[Corollary~1]{KLM19}).
			However, there is a rather simple criterion to determine whether the space $[X(1/\phi)]^{\bigstar}$, where $X$ is a rearrangement invariant sequence space
			and $\phi$ is a non-decreasing and concave function on $(0,\infty)$, is a Banach space up to an equivalent norm. Indeed, suppose that
			\begin{equation*}
				\frac{1}{\beta_X} < 1 - \frac{1}{\gamma_{\phi}}.
			\end{equation*}
			Then the discrete Hardy operator $\mathscr{H}$ is bounded on the space $X(1/\phi)$ (see \cite[Theorem~1]{Mal83};
			see also a note on Pavlov's result on p.~127 in \cite{KMP07}).
			In consequence, the functional $\vertiii{\bullet} \colon [X(1/\phi)]^{\bigstar} \rightarrow \mathbb{R}$ defined as
			\begin{equation*}
				x \rightsquigarrow \vertiii{x}\coloneqq \norm{\frac{x^{\star\star}}{\phi}}_X,
			\end{equation*}
			where $x^{\star \star}$ denote the {\bf maximal function} of $x^{\star}$, that is,
			$x^{\star\star} \coloneqq \left\{ \frac{1}{n} \sum_{k=1}^{n} x^{\star}_k \right\}_{n=1}^{\infty}$
			(see \cite[Definition~3.1, p.~52]{BS88}), defines on the space $[X(1/\phi)]^{\bigstar}$ an equivalent norm.
			In fact, in view of \cite[p.~54]{BS88} the triangle inequality is clear.
			Moreover, for any $x \in [X(1/\phi)]^{\bigstar}$ we have the following inequalities
			\begin{equation*}
				\norm{\frac{x^{\star}}{\phi}}_X
					\leqslant \norm{\frac{x^{\star\star}}{\phi}}_{X}
					= \vertiii{x}
					\leqslant \norm{\mathscr{H}}_{[X(1/\phi)]^{\bigstar} \rightarrow [X(1/\phi)]^{\bigstar}} \norm{\frac{x^{\star}}{\phi}}_X,
			\end{equation*}
			which means that
			\begin{equation*}
				\norm{x}_{[X(1/\phi)]^{\bigstar}} \approx \vertiii{x}.
			\end{equation*}
			\demo
		\end{remark}
		
		After all this preparations, we are finally in the position to provide the promised proof of Theorem~\ref{Prop: computing multipliers from m and to lambda}.
		
		\vspace{5pt}
		\noindent
		{\bf Proof of Theorem~\ref{Prop: computing multipliers from m and to lambda}.}
		{\bf (i)} Under the assumption that $\left\{ \varphi(n)/n \right\}_{n=1}^{\infty} \notin X^{\times}$,
		we have the following equivalences
		\begin{align*}
			M_{\lambda} \colon X \rightarrow \lambda_{\varphi} \quad \text{ is compact } \quad
			& \Longleftrightarrow \quad \lambda \in M(X,\lambda_{\varphi})_o \quad \text{(using Theorem~\ref{thm: ess norm ri})} \\
			& \Longleftrightarrow \quad \lambda \in \left[ \left[ X^{\times}\left( \frac{\varphi(t)}{t} \right) \right]^{\bigstar} \right]_o \quad \text{(by Lemma~\ref{LEMMA: a'la KLM12/14}(ii))} \\
			& \Longleftrightarrow \quad \lambda \in \left[ \left( X^{\times} \right)_o \left( \frac{\varphi(t)}{t} \right) \right]^{\bigstar} \quad \text{(by Lemma~\ref{LEMMA: a'la Kolwicz OC})} \\
			& \Longleftrightarrow \quad \left\{ \frac{\varphi(n)}{n}\lambda_n^{\star} \right\}_{n=1}^{\infty} \in \left( X^{\times} \right)_o.
		\end{align*}
		
		Now, suppose that $\left\{ \varphi(n)/n \right\}_{n=1}^{\infty} \in X^{\times}$.
		Then the first part of the proof of Lemma~\ref{LEMMA: a'la KLM12/14} together with duality results from
		\cite[Theorem~5.2, p.~112]{KPS82} teaches us that
		\begin{equation*}
			M(X, \lambda_{\varphi}) = M\left( (\lambda_{\varphi})^{\times},X^{\times} \right) = M(m_{t/\varphi(t)},X^{\times}) = \ell_{\infty}.
		\end{equation*}
		Therefore, applying Theorem~\ref{thm: ess norm ri} one more time, we see immediately that
		\begin{equation*}
			M_{\lambda} \colon X \rightarrow \lambda_{\varphi} \quad \text{ is compact if, and only if, } \quad \lambda \in (\ell_{\infty})_o = c_0.
		\end{equation*}
		
		{\bf (ii)} Suppose that $\left\{ 1/\varphi(n) \right\}_{n=1}^{\infty} \notin Y$. We have
		\begin{align*}
			M_{\lambda} \colon m_{\varphi} \rightarrow Y \quad \text{ is compact } \quad
			& \Longleftrightarrow \quad \lambda \in M(m_{\varphi},Y)_o \quad \text{(using Theorem~\ref{thm: ess norm ri})} \\
			& \Longleftrightarrow \quad \lambda \in \left[ \left[ Y\left( \frac{1}{\varphi} \right) \right]^{\bigstar} \right]_o \quad \text{(by Lemma~\ref{LEMMA: a'la KLM12/14}(i))} \\
			& \Longleftrightarrow \quad \lambda \in \left[ Y_o\left( \frac{1}{\varphi} \right) \right]^{\bigstar} \quad \text{(by Lemma~\ref{LEMMA: a'la Kolwicz OC})} \\
			& \Longleftrightarrow \quad \left\{ \frac{1}{\varphi(n)}\lambda_n^{\star} \right\}_{n=1}^{\infty} \in Y_o.
		\end{align*}
		
		Now, suppose that $\left\{ 1/\varphi(n) \right\}_{n=1}^{\infty} \in Y$. Then, due to the first part of the proof of Lemma~\ref{LEMMA: a'la KLM12/14},
		\begin{equation*}
			M(m_{\varphi},Y) = \ell_{\infty}.
		\end{equation*}
		Thus, we have
		\begin{equation*}
			M_{\lambda} \colon m_{\varphi} \rightarrow Y \quad \text{ is compact if, and only if, } \quad \lambda \in (\ell_{\infty})_o = c_0.
		\end{equation*}
		
		{\bf (iii)} Suppose that $\limsup_{n \rightarrow \infty} \psi(n)/\varphi(n) = \infty$. Then, in view of the embedding $\lambda_{\psi} \hookrightarrow m_{\psi}$
		(see \cite[Proposition~5.9, p.~70]{BS88}), it is straightforward to see that if $\limsup_{n \rightarrow \infty} \psi(n)/\varphi(n) = \infty$,
		then $\left\{ 1/\varphi(n) \right\}_{n=1}^{\infty} \notin m_{\psi}$. In consequence, $\left\{ 1/\varphi(n) \right\}_{n=1}^{\infty} \notin \lambda_{\psi}$.
		We have
		\begin{align*}
			M(m_{\varphi}, \lambda_{\psi})
				& = \left[ \lambda_{\psi}\left( \frac{1}{\varphi} \right) \right]^{\bigstar} \quad \text{(by Lemma~\ref{LEMMA: a'la KLM12/14}(i))}\\
				& = \left[ \left[ \ell_1(\Delta \psi) \right]^{\bigstar} \left( \frac{1}{\varphi} \right) \right]^{\bigstar}
					\quad \text{(because $\lambda_{\phi} \equiv [\ell_1(\Delta \phi)]^{\bigstar}$, $\Delta \phi(n) \coloneqq \phi(n+1) - \phi(n)$)} \\
				& = \left[ \ell_1(\Delta \psi) \left( \frac{1}{\varphi} \right) \right]^{\bigstar}
					\quad \text{(in view of the equality $\left( \frac{x^{\star}}{\varphi} \right)^{\star} = \frac{x^{\star}}{\varphi}$)} \\
				& = \left[ \ell_1 \left( \frac{\Delta \psi}{\varphi} \right) \right]^{\bigstar},
		\end{align*}
		where $\Delta \psi(n) \coloneqq \psi(n+1) - \psi(n)$ for $n \in \mathbb{N}$. Since the sequence $\left\{ \frac{\psi(n+1)-\psi(n)}{\varphi(n)} \right\}_{n=1}^{\infty}$
		is decreasing and $\ell_1$ is separable, so applying Lemma~\ref{LEMMA: a'la Kolwicz OC}, we get
		\begin{equation*}
			\left[ \left[ \ell_1 \left( \frac{\Delta \psi}{\varphi} \right) \right]^{\bigstar} \right]_o
				= \left[ (\ell_1)_o \left( \frac{\Delta \psi}{\varphi} \right) \right]^{\bigstar}
				= \left[ \ell_1 \left( \frac{\Delta \psi}{\varphi} \right) \right]^{\bigstar}.
		\end{equation*}
		Let us make it clear that the use of Lemma~\ref{LEMMA: a'la Kolwicz OC} was completely justified in this situation, because
		$\left\{ \Delta \psi(n) / \varphi(n) \right\}_{n=1}^{\infty} \notin \ell_1$. In fact, if this were not the case, then
		$\left\{ 1/\varphi(n) \right\}_{n=1}^{\infty} \in \lambda_{\psi}$ which, due to our assumptions, is impossible.
		All this means is that the sequence $\lambda = \{\lambda_n\}_{n=1}^{\infty}$ belongs to $M(m_{\varphi}, \lambda_{\psi})_o$ if,
		and only if, the sequence $\left\{ \frac{\psi(n+1)-\psi(n)}{\varphi(n)} \lambda_n^* \right\}_{n=1}^{\infty}$ belongs to $\ell_1$.
		Since the space $\lambda_{\psi}$ is order continuous (see Remark~\ref{REMARK: separable part lorentz marcinkiewicz}),
		so the rest follows from Theorem~\ref{THM: Essential norm Kothe sequence}.
		
		{\bf (iv)} Suppose that $\limsup_{n \rightarrow \infty} \varphi(n)/\psi(n) = \infty$. Using the known equality between the fundamental
		functions of the space $X$ and the fundamental function of its K{\" o}the dual $X^{\times}$ (see \cite[Theorem~5.2, p.~66]{BS88}), that is,
		\begin{equation*}
			\varphi_X(n)\varphi_{X^{\times}}(n) = n \quad \text{ for } \quad n \in \mathbb{N},
		\end{equation*}
		and arguing as at the beginning of the proof of (iii), it is clear that if $\limsup_{n \rightarrow \infty} \varphi(n)/\psi(n) = \infty$,
		then $\left\{ \varphi(n)/n \right\}_{n=1}^{\infty} \notin (\lambda_{\psi})^{\times} = m_{t/\psi(t)}$.
		In consequence, we have
		\begin{align*}
			M(\lambda_{\psi}, \lambda_{\varphi})
				& = \left[ \left( \lambda_{\psi} \right)^{\times}\left( \frac{\varphi(t)}{t} \right) \right]^{\bigstar}
					\quad \text{(by Lemma~\ref{LEMMA: a'la KLM12/14}(ii))} \\
				& = \left[ m_{t/\psi(t)} \left( \frac{\varphi(t)}{t} \right) \right]^{\bigstar} \quad \text{(using \cite[Theorem~5.2, p.~112]{KPS82})}\\
				& = \left[ \left[ \ell_{\infty}\left( \frac{t}{\psi(t)} \right) \right]^{\bigstar} \left( \frac{\varphi(t)}{t} \right) \right]^{\bigstar}
					\quad \text{(because $m_{\phi} = [\ell_{\infty}(\phi)]^{\bigstar}$ provided $\gamma_{\phi} > 1$)}.
		\end{align*}
		Next, keeping in mind that the function $t \rightsquigarrow \varphi(t)/t$ is decreasing and using the obvious fact that the product of
		two decreasing functions is still decreasing, we get
		\begin{equation*}
			\left[ \left[ \ell_{\infty}\left( \frac{t}{\psi(t)} \right) \right]^{\bigstar} \left( \frac{\varphi(t)}{t} \right) \right]^{\bigstar}
				= \left[ \ell_{\infty}\left( \frac{t}{\psi(t)} \right) \left( \frac{\varphi(t)}{t} \right) \right]^{\bigstar}
				= \left[ \ell_{\infty}\left( \frac{\varphi}{\psi} \right) \right]^{\bigstar}.
		\end{equation*}
		Now, since $\left\{ \varphi(n)/\psi(n) \right\}_{n=1}^{\infty} \notin \ell_{\infty}$
		(otherwise, $\left\{ \varphi(n)/n \right\}_{n=1}^{\infty} \in m_{t/\psi(t)}$ and this is contrary to our assumptions),
		so we can use Lemma~\ref{LEMMA: a'la Kolwicz OC} and, consequently, we get
		\begin{equation*}
			\left[ \left[ \ell_{\infty}\left( \frac{\varphi}{\psi} \right) \right]^{\bigstar} \right]_o
				= \left[ (\ell_{\infty})_o \left( \frac{\varphi}{\psi} \right) \right]^{\bigstar}
				= \left[ c_0 \left( \frac{\varphi}{\psi} \right) \right]^{\bigstar}.
		\end{equation*}
		Hence it is clear that the sequence $\lambda = \{\lambda_n\}_{n=1}^{\infty}$ belongs to the space $M(\lambda_{\psi}, \lambda_{\varphi})_o$ if, and only if,
		the sequence $\left\{ \frac{\varphi(n)}{\psi(n)} \lambda_n^{\star} \right\}_{n=1}^{\infty}$ belongs to $c_0$. The use of Theorem~\ref{THM: Essential norm Kothe sequence}
		together with the information that the space $\lambda_{\varphi}$ is order continuous (again, see Remark~\ref{REMARK: separable part lorentz marcinkiewicz}) completes the proof.
		\hfill $\blacksquare$
	
\section{{\bf Open ends}} \label{SECTION: Open ends}

Let us list and discuss some problems which arise from this paper.

\subsection{Things that should not be}

We believe that the assumptions about separability and reflexivity present in Theorem~\ref{THM: essnorm multipliers function}
and Theorem~\ref{THM: Essential norm Kothe sequence} are unnecessary or, at least, can be significantly weakened. This belief
is clearly supported by Theorem~\ref{thm: ess norm ri} (which, however, relies on some heavy machinery) and Schep's results from \cite{Sch22}.
Moreover, although we have not checked all the details, our global assumption of $\sigma$-finiteness of the measure space does
not seem to be essential (of course, some modifications, such as the use of limits along Fr{\' e}chet filters, are necessary,
but here again we can recycle the ideas from \cite{Sch22}).

\subsection{The general case}

Although Theorem~\ref{THM: essnorm multipliers function} together with Theorem~\ref{THM: Essential norm Kothe sequence} provide an exact estimate of the essential norms,
unfortunately, their conglomeration in the form of Theorem~\ref{THM: ess norm func + seq} loses this value. Therefore, in this context, it is natural to ask the following question.

\begin{problem}
	{\it Let $X$ and $Y$ be two K{\" o}the spaces with the Fatou property both defined over the same $\sigma$-finite measure space $(\Omega,\Sigma,\mu)$.
		Suppose they can be represented, up to the equality of norms, as a direct sums of their atomic and non-atomic parts, that is, there are two two-dimensional
		K{\" o}the spaces, say $E$ and $F$, such that}
	\begin{equation*}
		X \equiv X(\Omega_a) \oplus_E X(\Omega_c) \quad \textit{ and } \quad Y \equiv Y(\Omega_a) \oplus_F Y(\Omega_c).
	\end{equation*}
	{\it Do we then have the equality}
	\begin{equation*}
		\norm{M_{\lambda} \colon X \rightarrow Y}_{ess}
		= \norm{ \left\{ \norm{M_{\lambda} \colon X(\Omega_{\bullet}) \rightarrow Y(\Omega_{\bullet})}_{ess} \right\}_{\bullet \in \{a, c\}} }_{M(E,F)} \text{\it ?}
	\end{equation*}
\end{problem}

It is worth recalling here that Schep in \cite{Sch22} (precisely, see Theorem~5.2 ibidem) was able to show that
for a complex K{\" o}the sequence space $X$ the following formula holds
\begin{equation*}
	\norm{M_{\lambda} \colon X \rightarrow X}_{ess}
		= \max \left\{ \norm{M_{\lambda} \colon X(\Omega_a) \circlearrowleft }_{ess}, \norm{M_{\lambda} \colon X(\Omega_c) \circlearrowleft }_{ess} \right\}.
\end{equation*}
His proof, however, rely heavily on the spectral theory and does not work in the non-algebraic setting.

\subsection{On Pitt's compactness theorem once again}

It follows from Example~\ref{PROPOSITION: Appendix A - counterexample factorization} that a direct generalization of Pitt's theorem
to the scale of rearrangement invariant spaces is not possible (see Remark~\ref{REMARK: general Pitt}).
With \cite{DLMR00} and \cite{LM02} in mind, the following problem seems tempting.

\begin{problem}
	{\it Let $X$ and $Y$ be two rearrangement invariant sequence spaces.
	Suppose that $X$ and $Y$ are strongly separated, that is, $\alpha_X > \beta_Y$.
	Is it then that $\mathscr{L}(X,Y) = \mathscr{K}(X,Y)$?}
\end{problem}

\subsection{The secret live of approximation numbers}

	Let $X$ and $Y$ be two Banach spaces. Recall, that the {\bf $n^{\text{th}}$ approximation number $a_n(T \colon X \rightarrow Y)$}
	of an operator $T$ acting between $X$ and $Y$ is given by
	\begin{equation*}
		a_n(T \colon X \rightarrow Y) \coloneqq \inf \Big\{ \norm{T - T_n}_{X \rightarrow Y} \colon \text{rank}(T_n) \leqslant n \Big\}.
	\end{equation*}
	Roughly speaking, the asymptotic behavior of $a_n(T \colon X \rightarrow Y)$ describes the degree of compactness of $T$.
	In particular, as was proved by Pietsch, the operator $T$ is approximable if, and only if, $\lim_{n \rightarrow \infty} a_n(T \colon X \rightarrow Y) = 0$.
	Moreover, $\lim_{n \rightarrow \infty} a_n(T \colon X \rightarrow Y) = \norm{T \colon X \rightarrow Y}_{ess}$.

	Now, it is not very hard to show (see \cite[Theorem~2.6]{HMR76} and references therein) that for $1 < p < q < \infty$ with $1/r = 1/q - 1/p$
	and $n \in \mathbb{N}$ we have the following formula
	\begin{equation*}
		a_n(M_{\lambda} \colon \ell_q \rightarrow \ell_p)
			= \left( \sum_{k=n+1}^{\infty} {\lambda_k^{\star}}^r \right)^{1/r}
			= \norm{ \sum_{i=n+1}^{\infty} \lambda^{\star}_i }_{M(\ell_q,\ell_p)}.
	\end{equation*}
	On the other hand, the most obvious guess that
	\begin{equation*}
		a_n(M_{\lambda} \colon \ell_p \rightarrow \ell_q)
			= \norm{ \sum_{i=n+1}^{\infty} \lambda^{\star}_i }_{M(\ell_p,\ell_q)}
			= \sup_{k = n+1, n+2, ...} \lambda_k^*
			= \lambda_{n+1}^{\star}
	\end{equation*}
	does not seem to work\footnote{However, it can be shown that $a_n(M_{\lambda} \colon \ell_p \rightarrow \ell_p) = \lambda_{n+1}$ (see \cite[Theorem~7.1]{Pie74} and \cite[p.~58]{HMR76}).}
	(see \cite[Example~2.13]{HMR76}).
	
	\begin{problem}
		{\it Describe the pairs of rearrangement invariant sequence spaces $E$ and $F$ for which we have the equality}
		\begin{equation*}
			a_n(M_{\lambda} \colon E \rightarrow F) = \norm{ \sum_{i=n+1}^{\infty} \lambda^{\star}_i }_{M(E,F)}.
		\end{equation*}
		{\it What then can be said about $a_n(M_{\lambda} \colon F \rightarrow E)$?}
	\end{problem}
	
	Of course, approximation numbers are only a special members of a large family of the so-called {\it $s$-numbers}
	(including, for example, Bernstein, Gelfand, Hilbert, Kolmogorov and Weyl numbers; see \cite[Section~7]{Pie74} and \cite{Pie80}).
	Note also that the so-called {\it entropy numbers} of multiplication operators have been found in \cite{Kuh05} and \cite{Sch84}.
	
\subsection{Some other operator ideals}

Let $\mathscr{I}$ be a closed operator ideal in the sense of Pietsch. Further, let $X$ and $Y$ be two K{\" o}the spaces
both defined over the same $\sigma$-finite measure space $(\Omega,\Sigma,\mu)$.
By $M_{\mathscr{I}}(X,Y)$ we will understand the space of all symbols $\lambda$ from $M(X,Y)$ with the property that induced
multiplication operator $M_{\lambda} \colon X \rightarrow Y$ belongs to the ideal $\mathscr{I}(X,Y)$.
In brief, $M_{\mathscr{I}}(X,Y) = \mathscr{I} \cap M(X,Y)$.
Clearly, $M_{\mathscr{I}}(X,Y)$ is a closed subspace of $M(X,Y)$.
Moreover, with this definition in hand, we can easily estimate the essential norm of the multiplication operator
$M_{\lambda} \colon X \rightarrow Y$ with respect to the ideal $\mathscr{I}$ in the following way
\begin{align*}
	\text{dist}_{\mathscr{L}(X,Y)}(M_{\lambda} \colon X \rightarrow Y,\mathscr{I}(X,Y))
	& = \inf\limits_{J \in \mathscr{I}(X,Y)} \norm{M_{\lambda} - J}_{X \rightarrow Y} \\
	& \leqslant \inf\limits_{M_{\xi} \in \mathscr{I}(X,Y)} \norm{M_{\lambda} - M_{\xi}}_{X \rightarrow Y} \\
	& = \inf\limits_{\xi \in M_{\mathscr{I}}(X,Y)} \norm{\lambda - \xi}_{M(X,Y)} \\
	& = \text{dist}_{M(X,Y)}(\lambda, M_{\mathscr{I}}(X,Y)).
\end{align*}
Thus, looking at Theorem~\ref{Proposition: odleglosci = odleglosci} from this broader perspective, it is hard not to formulate the following

\begin{problem}
	{\it Under what assumptions on $X$ and $Y$ do we have the equality}
	\begin{equation*}
		\text{dist}_{\mathscr{L}(X,Y)}(M_{\lambda} \colon X \rightarrow Y,\mathscr{I}(X,Y)) = \text{dist}_{M(X,Y)}(\lambda, M_{\mathscr{I}}(X,Y)) \textit{?}
	\end{equation*}
\end{problem}

Aside from the general theory, perhaps the most interesting question concerns the description of symbols that generate strictly singular and nuclear multipliers.

\subsection{Compact multipliers between Tandori function spaces}

For a given K{\" o}the function space $X$, by the {\bf abstract Tandori function space $\widetilde{X}$} we mean a vector space of all real-valued measurable functions,
say $f$, with $\widetilde{f} \in X$, where $\widetilde{f}$ is the {\bf least decreasing majorant} of a given function $f$,
that is, $\widetilde{f}(x) \coloneqq \esssup_{t \geqslant x} \abs{f(t)}$.
We equip $\widetilde{X}$ with the norm $\norm{f}_{\widetilde{X}} \coloneqq \Vert \widetilde{f} \Vert_X$.
It is straightforward to see that the ideal $(\widetilde{X})_o$ is trivial (see, for example, \cite[Theorem~1(e)]{LM15a};
in fact, due to \cite[Proposition~4.8]{KKM21}, the space $\widetilde{X}$ even contains a lattice isometric copy of $\ell_{\infty}$).
Thus, in light of Theorem~\ref{THM: essnorm multipliers function} and Theorem~\ref{THM: weak ess norm}, we want to pose the following

\begin{problem}
	{\it What can we say about (weakly) compact multipliers between $\widetilde{X}$ and $\widetilde{Y}$?}
\end{problem}

\subsection{Missing multipliers}

The following problem, if solved, will complete Lemma~\ref{LEMMA: a'la KLM12/14}.

\begin{problem}
	{\it Describe the space $M(\lambda_{\varphi}, m_{\psi})$.}
\end{problem}
	
\appendix

\section{{\bf One counter-example to the factorization problem}} \label{Appendix B}

Let $(\Omega,\Sigma,\mu)$ be a $\sigma$-finite measure space.
Celebrated Lozanovski{\u \i}'s result says that each function $f$ from $L_1(\Omega)$ can be written as a pointwise product of two functions,
say $g$ and $h$, where the first one belongs to a K{\" o}the space $X$ and the second one comes from its K{\" o}the dual $X^{\times}$,
in such a way that
\begin{equation*}
	\norm{f}_{L_1(\Omega)} \approx \norm{g}_X \norm{h}_{X^{\times}}
\end{equation*}
(see \cite[Theorem~6]{Loz69}; cf. \cite[p.~185]{Ma89}). Using the language of the arithmetic of K{\" o}the spaces we can
express this result in compact, but equivalent, form as
\begin{equation*}
	X \odot X^{\times} = X \odot M(X, L_1(\Omega)) = L_1(\Omega).
\end{equation*}
From here, however, the short road leads to the much more general question if for a given pair of K{\" o}the spaces, say $X$ and $Y$,
the following equality holds
\begin{equation*}
	X \odot M(X,Y) = Y?
\end{equation*}
We will refer to this kind of question as the {\bf factorization problem}.

\begin{remark}
	It is not particularly difficult to find a pair of K{\" o}the sequence spaces for which the factorization problem has a negative solution.
	For example, let $X$ and $Y$ be two K{\" o}the spaces such that $X \neq Y$, but $M(X,Y) = L_{\infty}(\Omega)$. Then
	\begin{equation*}
		X \odot M(X,Y) = X \odot L_{\infty}(\Omega) = X \neq Y,
	\end{equation*}
	that is, $X$ does not factorize through $Y$. As a more concrete illustration, take $1 < p < \infty$ and notice that
	\begin{equation*}
		\ell_p \odot M(\ell_p, \ell_{p,\infty}) = \ell_p \odot \ell_{\infty} = \ell_p \subsetneqq \ell_{p,\infty}.
	\end{equation*}
	This clearly means that $\ell_p$ does not factorize through $\ell_{p,\infty}$. On the other hand, using duality argument, we can also show that
	\begin{align*}
		\ell_{p,1} \odot M(\ell_{p,1}, \ell_{p})
			& = \ell_{p,1} \odot M((\ell_{p})^{\times}, (\ell_{p,1})^{\times}) \\
			& = \ell_{p,1} \odot M(\ell_q, \ell_{q,\infty}) \\
			& = \ell_{p,1} \odot \ell_{\infty} \\
			& = \ell_{p,1} \subsetneqq \ell_{p},
	\end{align*}
	where $1/p + 1/q = 1$. Consequently, also $\ell_{p,1}$ does not factorize through $\ell_p$.
	\demo
\end{remark}

Based on this two\footnote{To the best of our knowledge, all known counterexamples to the factorization problem are of this type (being precise,
in the case of K{\" o}the function spaces we have one more possibility, namely, the space $M(X,Y)$ can also be trivial, but this is really clich{\' e}d).}
examples, it is not entirely unreasonable to suppose that if $M(X,Y) \neq L_{\infty}(\Omega)$ then the factorization problem does have the positive solution.
We will show that this is far from the truth even in the class of rearrangement invariant spaces.

\begin{example} \label{PROPOSITION: Appendix A - counterexample factorization}
	{\it There exists an Orlicz function, say $\widetilde{M}$, with the property that the corresponding Orlicz space $\ell_{\widetilde{M}}$
		does not factorize through $\ell_2$, that is,}
		\begin{equation*}
			\ell_{\widetilde{M}} \odot M(\ell_{\widetilde{M}}, \ell_2) \neq \ell_2,
		\end{equation*}		
	{\it but at the same time $M(\ell_{\widetilde{M}}, \ell_2) \neq \ell_{\infty}$.}
\end{example}
\begin{proof}
	First of all, let's define the Orlicz function $\widetilde{M}$ in the following\footnote{Roughly speaking the space $\ell_{\widetilde{M}}$
	is \enquote{very close} to $\ell_2$, but at the same time the Orlicz function $\widetilde{M}$ goes to zero \enquote{much faster} (in the sense of the $\Delta_2$-condition)
	than the function $t \rightsquigarrow t^2$.} way
	\begin{equation*}
		\widetilde{M}(t) \coloneqq
		\begin{cases}
			\hfil \frac{2}{n!}t - \frac{1}{(n!)^2} & \text{ for } \quad \hfil \frac{n+2}{2(n+1)!} \leqslant t < \frac{n+1}{2n!},\ n \geqslant 1\\
			\hfil t^2  & \text{ for } \quad \hfil 1 \leqslant t < \infty.
		\end{cases}
	\end{equation*}
	We divide the proof into three parts.
	
	{\bf First step}. We will start by calculating the space $M(\ell_{\widetilde{M}}, \ell_2)$. Since the space $\ell_2$ can obviously be treated
	as an Orlicz space, say $\ell_N$, generated by the function $N(t) = t^2$, so our problem comes down to calculating the space of pointwise
	multipliers between two Orlicz sequence spaces. Luckily, Djakov and Ramanujan \cite{DR00} come to our aid. Based on their results,
	we conclude that
	\begin{equation*}
		M(\ell_{\widetilde{M}}, \ell_2) = \ell_{N \ominus \widetilde{M}},
	\end{equation*}
	where
	\begin{equation*}
		( N \ominus \widetilde{M} )(t) \coloneqq \sup_{0 \leqslant s \leqslant 1}\left[ N(st) - \widetilde{M}(s) \right]
			\quad \text{ for } \quad t \geqslant 0
	\end{equation*}
	is the so-called generalized Young conjugate of $\widetilde{M}$ with respect to $N$ (see \cite[Theorem~3]{DR00}; see also \cite{LT21}).
	Now, we clearly have
	\begin{equation*}
		( N \ominus \widetilde{M} )(t) = \sup\limits_{0 \leqslant s \leqslant 1}\left[ N(st) - \widetilde{M}(t) \right]
			= \sup\limits_{\substack{n \in \mathbb{N}\\ \frac{n+2}{2(n+1)!} \leqslant s < \frac{n+1}{2n!}}} \left[ (st)^2 - \frac{2}{n!}s + \frac{1}{(n!)^2} \right].
	\end{equation*}
	However, it is not difficult to see that the functions that appear under the supremum do not attain their maximum values within the respective intervals,
	so the values at the ends of the intervals should be considered, that is,
	\begin{align*}
		( N \ominus \widetilde{M} )(t)
			& = \max\limits_{\ n \in \mathbb{N}}\left[ \left( \frac{n+2}{2(n+1)!} \right)^{2} t^2 - \frac{2}{n!}\frac{n+2}{2(n+1)!}
				+ \frac{1}{(n!)^{2}}, \left( \frac{n+1}{2n!} \right)^{2} t^2 - \frac{2}{n!}\frac{n+1}{2n!} + \frac{1}{(n!)^{2}}\right]\\
			& = \max\limits_{\ n\in \mathbb{N}}\left[ \frac{(n+2)^{2}t^{2} - 8(n+2)(n+1) + 4(n+1)^{2}}{4\left( (n+1)! \right)^{2}},\frac{(n+1)^{2}t^{2}-4n}{4(n!)^{2}} \right].
	\end{align*}
	Since the function
	\begin{equation*}
		t \rightsquigarrow \frac{(n+2)^2 t^2 - 8(n+2)(n+1) + 4(n+1)^2}{4\left( (n+1)! \right)^2}
	\end{equation*}
	takes only negative values for $t \geqslant 0$ and $n \in \mathbb{N}$, so
	\begin{equation*}
		( N \ominus \widetilde{M} )(t) = \max\limits_{\ n\in \mathbb{N}}\left[ \frac{(n+1)^{2}t^{2}-4n}{4(n!)^{2}} \right].
	\end{equation*}
	Hence, it is not very hard to see that
	\begin{equation*}
		( N \ominus \widetilde{M} )(t) = 
		\begin{cases}
			\hfil 0 &\text{ for } \hfil t = 0\\
			\quad \hfil \frac{(n+1)^2t^2 - 4n}{4(n!)^2}
			& \text{ for } \quad \hfil \frac{2\sqrt{n+\varepsilon_n}}{n+1} < t \leqslant \frac{2\sqrt{n-1+\varepsilon_{n-1}}}{n},\ n > 1,
		\end{cases}
	\end{equation*}
	where $\varepsilon_n = O(n^{-2})$ (here, the letter \enquote{$O$} stands from the {\it big Oh notation} or the {\it Bachmann--Landau notation};
	we don't need to be more precise than that, as we pay attention only to \enquote{small} arguments and we consider all the equalities between spaces
	up to the equality of norms).
	
	{\bf Second step}. Next we will show that the space $M(\ell_{\widetilde{M}}, \ell_2)$ is not separable, but $M(\ell_{\widetilde{M}}, \ell_2) \neq \ell_{\infty}$.
	Indeed, since the function $N \ominus \widetilde{M}$ does not satisfy the $\Delta_2$-condition (to see this just put $u_n \coloneqq\frac{2\sqrt{n+\varepsilon_n}}{n+1}$
	for $n \in \mathbb{N}$ and observe that $\frac{u_{n+1}}{u_n} \rightarrow 1$ as $n \rightarrow \infty$, while $\frac{\varphi(u_{n+1})}{\varphi(u_{n})} \approx \frac{1}{n} \rightarrow 0$
	as $n \rightarrow \infty$), so
	\begin{equation*}
		M(\ell_{\widetilde{M}}, \ell_2)_o = \left( \ell_{N \ominus \widetilde{M}} \right)_o \neq \ell_{N \ominus \widetilde{M}} = M(\ell_{\widetilde{M}}, \ell_2)
	\end{equation*}
	(see \cite[Proposition~4.a.4, p.~138]{LT77} or \cite[Corollary, pp.~21--22]{Ma89}).
	On the other hand, it is straightforward to see that an Orlicz space, say $\ell_F$, coincides with $\ell_{\infty}$ if, and only if,
	the Young function $F$ vanishes in some neighborhood of zero. However,
	\begin{equation*}
		\inf \left\{ t > 0 \colon ( N \ominus \widetilde{M} )(t) > 0 \right\} = 0,
	\end{equation*}
	so this is not the case and
	\begin{equation*}
		M(\ell_{\widetilde{M}}, \ell_2) = \ell_{N \ominus \widetilde{M}} \neq \ell_{\infty}.
	\end{equation*}
	
	{\bf Third step}. Finally, we will prove that $\ell_{\widetilde{M}}$ does not factorize $\ell_2$.
	To show this we will use the characterization from \cite[Theorem 2]{LT21} which states that for a pair of Orlicz functions,
	say $M_0$ and $M_1$, the factorization $\ell_{M_0}\odot M(\ell_{M_0}, \ell_{M_1}) = \ell_{M_1}$ holds if, and only if, there exist
	constants $C, D > 0$ with
	\begin{equation*}
		C M_1^{-1}(t) \leqslant M_0^{-1}(t)(M_1 \ominus M_0)^{-1}(t) \leqslant DM_1^{-1}(t) \quad \text{ for all } \quad 0 < t < 1.
	\end{equation*}
	Consequently, it is enough to show that there is no constant $C > 0$ such that for all $0 < t < 1$ we have the inequality
	\begin{equation*}
		C N^{-1}(t)\leqslant \left( \widetilde{M} \right)^{-1}(t)\left( N \ominus \widetilde{M} \right)^{-1}(t).
	\end{equation*}
	Suppose otherwise. Then, for some constant $C > 0$ and all $n > 1$, we have
	\begin{align*}
		C \leqslant \frac{\left( \widetilde{M} \right)^{-1}\left(\frac{n}{(n!)^2}\right)\left( N \ominus \widetilde{M} \right)^{-1}\left(\frac{n}{(n!)^2}\right)}{N^{-1}\left(\frac{n}{(n!)^2}\right)}
		  & = \frac{\frac{\sqrt{n}}{n!}}{\frac{n+1}{2n!}}\left( N \ominus \widetilde{M} \right)^{-1}\left(\frac{n}{(n!)^2}\right) \\
		  & = 2\frac{\sqrt{n}}{n+1}\left( N \ominus \widetilde{M} \right)^{-1}\left(\frac{n}{(n!)^2}\right) \\
		  & \leqslant \frac{2\left( N \ominus \widetilde{M} \right)^{-1}\left(\frac{n}{(n!)^2}\right)}{\sqrt{n}}.
	\end{align*}
	But this means that
	\begin{equation*}
		\left( N \ominus \widetilde{M} \right)^{-1}\left(\frac{n}{(n!)^2}\right) \geqslant \frac{1}{2}C\sqrt{n} \quad \text{ for all } \quad n > 1.
	\end{equation*}
	Here we get a contradiction, because this is impossible.
\end{proof}

Note also that since the space $\ell_2$ is reflexive and, as the second part of the proof of Example~\ref{PROPOSITION: Appendix A - counterexample factorization}
shows, the space $M(\ell_{\widetilde{M}}, \ell_2)$ is not separable, so we can use Theorem~\ref{THM: Essential norm Kothe sequence} and easily get the following

\begin{corollary} \label{COR: Appendix A}
	{\it $\mathscr{L}(\ell_{\widetilde{M}}, \ell_2) \neq \mathscr{K}(\ell_{\widetilde{M}}, \ell_2)$.}
\end{corollary}

\begin{remark}[About Corollary~\ref{COR: Appendix A}]
	Remembering that $\ell_2 \hookrightarrow \ell_{\widetilde{M}}$ it is instructive to compare Corollary~\ref{COR: Appendix A}
	with our discussion about Pitt's compactness theorem in Remark~\ref{REMARK: general Pitt}.
	\demo
\end{remark}

\end{document}